\documentclass[preprint, 12pt, sort&compress]{elsarticle}

\usepackage{latexsym}
\usepackage{amsfonts}
\usepackage{graphicx,amssymb,natbib}
\usepackage{amsmath}
\usepackage{amssymb}
\usepackage[mathscr]{eucal}
\usepackage{enumerate}
\usepackage{amsthm}

\usepackage{bm}

\usepackage{color, soul}

\begin{document}

\newtheorem{tm}{Theorem}[section]
\newtheorem{pp}[tm]{Proposition}
\newtheorem{lm}[tm]{Lemma}
\newtheorem{df}[tm]{Definition}
\newtheorem{tl}[tm]{Corollary}
\newtheorem{re}[tm]{Remark}
\newtheorem{eap}[tm]{Example}

\newcommand{\pof}{\noindent {\bf Proof} }
\newcommand{\ep}{$\quad \Box$}

\newcommand{\al}{\alpha}
\newcommand{\be}{\beta}
\newcommand{\var}{\varepsilon}
\newcommand{\la}{\lambda}
\newcommand{\de}{\delta}
\newcommand{\str}{\stackrel}

\renewcommand{\proofname}{\bf Proof}

\allowdisplaybreaks

\begin{frontmatter}

\title{Properties of several fuzzy set spaces
\tnoteref{usc}
 }
\tnotetext[usc]{Project supported by
 Natural Science Foundation of Fujian Province of China(No. 2020J01706)}
\author{Huan Huang}
 \ead{hhuangjy@126.com }
\address{Department of Mathematics, Jimei
University, Xiamen 361021, China}

\date{}

\begin{abstract}
This paper discusses the properties the spaces of fuzzy sets in a metric space
equipped with
 the endograph metric and the sendograph metric, respectively.
We fist give some relations among
  the
endograph metric, the sendograph metric and the $\Gamma$-convergence, and then investigate the
level characterizations of the endograph metric and the
 $\Gamma$-convergence.
By using the above results, we give some relations among the endograph metric, the sendograph metric, the supremum metric and the $d_p^*$ metric, $p\geq 1$.
On the basis of the above results,
we
present the characterizations of total boundedness, relative compactness and compactness
in the space
of
compact positive $\al$-cuts fuzzy sets equipped with the endograph metric,
and
in
the space
of
 compact support fuzzy sets equipped with the sendograph metric, respectively.
 Furthermore, we
give
completions of
these metric spaces, respectively.

\end{abstract}

\begin{keyword}
  Endograph metric; Sendograph metric; Hausdorff metric; Total boundedness; Relative compactness; Compactness; Completion
\end{keyword}

\end{frontmatter}

\date{}

\section{Introduction}

A fuzzy set can be identified with its endograph. Also, a fuzzy set can be identified with its sendograph.
So convergence structures
on fuzzy sets can be defined on their endographs or sendographs.
The the endograph metric $H_{\rm end}$ convergence, the sendograh metric $H_{\rm send}$ convergence and the
$\Gamma$-convergence are this kind of convergence structures.
These three convergence structures
are related to each other
\cite{rojas, huang}.

The endograph metric on fuzzy sets
is shown to
has significant advantages
 \cite{kloeden, kupka}.
The sendograph metric has attracted deserving attentions \cite{fan,greco}.
Compactness is one of the central concepts in topology and analysis and useful in applications (see \cite{kelley, wa}).
There is a lot of work devoted to characterizations
of
compactness in various fuzzy set spaces endowed with different topologies \cite{fan,greco,greco3,huang,huang9,roman,trutschnig,wu2}.
 It is natural
to consider what the
 completion of a metric space is. The recent results on completions of fuzzy set spaces including \cite{huang, huang9}.

In \cite{huang},
we presented the relations and level characterizations of the endograph metric $H_{\rm end}$ and the $\Gamma$-convergence.
Based on this, we
have given the characterizations of total boundedness, relative compactness and compactness
of
fuzzy set spaces equipped with the endograph metric $H_{\rm end}$.
We also pointed out the completions of fuzzy set spaces
according to the endograph metric $H_{\rm end}$.

The common fuzzy sets used
in theoretical research and practical applications
are fuzzy sets in a metric space whose
positive cut set are nonempty compact sets.
Common compact fuzzy sets are common fuzzy sets whose support sets are compact.

Throughout this paper, we suppose that $X$ \emph{is a metric space endowed with a metric} $d$.

The symbols $F^1_{USCG}(X)$ and $F^1_{USCB} (X)$ are used to denote
the set of common fuzzy sets in $X$ and
the set of common compact fuzzy sets in $X$, respectively.
We
use $F^1_{USC}(X)$ to denote the set of normal and upper semi-continuous fuzzy sets in $X$.
$F^1_{USCB} (X)$ is a subset of $F^1_{USCG}(X)$.
$F^1_{USCG}(X)$ is a
subset of $F^1_{USC}(X)$.

The results in \cite{huang} are obtained
 on the realm of fuzzy sets in the $m$-dimensional Euclidean space $\mathbb{R}^m$.
$\mathbb{R}^m$ is a special type
of metric space (for simplicity, in this paper, $\mathbb{R}$ is used to denote the $1$-dimensional Euclidean space).
Of course, it is worth to study the
fuzzy sets in a metric space
\cite{jarn, greco, greco3}.
In this paper, the results are obtained on the realm of fuzzy sets in a general metric space $X$.
We mainly discuss
 $H_{\rm end}$ metirc and $H_{\rm send}$ metric on $F^1_{USC}(X)$
including
the relations among $H_{\rm end}$ metric, $H_{\rm send}$ metric and other convergence structures,
and properties of $H_{\rm end}$ metric and $H_{\rm send}$ metric.

The first part of this paper is the relations among the $H_{\rm end}$ metric, the $H_{\rm send}$ metric and the
$\Gamma$-convergence, the level characterizations of the endograph metric $H_{\rm end}$ and the $\Gamma$-convergence, and
the relations among the supremum metric $d_\infty$, the $H_{\rm end}$ metric, the $H_{\rm send}$ metric and
the $d_p^*$ metric.
The $d_p^*$ metric is an expansion of the $L_p$-type $d_p$ distance on $F^1_{USC}(X)$.
The conclusions on the relations among
the $H_{\rm end}$ metric, the $H_{\rm send}$ metric and
the $d_p^*$ metric are verified using the above results including the level characterization of the endograph metric $H_{\rm end}$.

To aid discussion, we introduce the sets $P^1_{USC}(X)$ and $P^1_{USCB}(X)$. $P^1_{USCB}(X)$ is a subset of $P^1_{USC}(X)$.
The
$F^1_{USC}(X)$ and $F^1_{USCB}(X)$ can be viewed as the subsets of $P^1_{USC}(X)$ and $P^1_{USCB}(X)$, respectively.

We
define the
$H_{\rm send}$ distance and the $H_{\rm end}$ distance on $P^1_{USC}(X)$,
and give the relations among
the $H_{\rm send}$ distance, the $H_{\rm end}$ distance and the Kuratowski convergence on $P^1_{USC}(X)$.
Then, as corollaries, we obtain the
 relations among
the $H_{\rm send}$ metric, the $H_{\rm end}$ metric and the $\Gamma$-convergence on $F^1_{USC}(X)$.

We discuss the level characterizations
of the $\Gamma$-convergence and the
endograph metric $H_{\rm end}$ on
fuzzy sets in $F^1_{USC}(X)$.
It is shown that
under some conditions,
the $\Gamma$-convergence of fuzzy sets can be decomposed to
the Kuratowski convergence of certain $\al$-cuts,
and
the $H_{\rm end}$ metric convergence of fuzzy sets
can be decomposed to
the Hausdorff metric convergence of certain $\al$-cuts.

The understanding of the relation among the $H_{\rm end}$ metric, the $H_{\rm send}$ metric and the $\Gamma$-convergence
 is beneficial for the understanding of themselves.
The
 level characterizations
 help to study these three convergence structures on fuzzy sets by using
the
properties of the corresponding $\al$-cuts.

A
 $H_{\rm send}$ metric convergent sequence
is
a
 $H_{\rm end}$ metric convergent sequence.
A
 $H_{\rm end}$ metric convergent sequence is also a $\Gamma$-convergent sequence.
So the
knowledge of the $\Gamma$-convergent sequences can help us to analyse the properties of the $H_{\rm end}$ convergent sequences and the $H_{\rm send}$ convergent sequences.
For this reason,
we
 give the level characterizations of a $\Gamma$-convergent sequence in this paper.

Based on the results in the first part, we give the other results of this paper.
 The second part of this paper
is the characterizations of total boundedness, relative compactness and compactness
in $(F^1_{USCG} (X), H_{\rm end})$ and $(F^1_{USCB} (X), H_{\rm end})$, respectively.
Here we mention
that
the characterization of
relatively compactness
in
$(F^1_{USCB} (X), H_{\rm send})$ has already been given by Greco \cite{greco}.

The total boundedness is the key property of compactness in metric space.
We show that a set $U$ in $(F^1_{USCG} (X), H_{\rm end})$ is totally bounded (respectively, relatively compact)
if and only if
for each $\al\in (0,1]$,
the union of all the $\al$-cuts of $u\in U$ is totally bounded (respectively, relatively compact) in $(X,d)$.
We also show that
 a set $U$ in $(F^1_{USCB} (X), H_{\rm send})$ is totally bounded
if and only if
the union of all the $0$-cuts of $u\in U$ is totally bounded in $(X, d)$.

It is shown that
for a set $U$ in $(F^1_{USCG} (X), H_{\rm end})$ or $(F^1_{USCB} (X), H_{\rm send})$,
the total boundedness, relative compactness and compactness
of $U$
are closely related to
 the total boundedness, relative compactness and compactness    of the union of all the $\al$-cuts of $u\in U$ in $(X,d)$, respectively.

We point out that some part of the proof of the characterizations in this paper is similar to the corresponding part in \cite{huang}.
But in general, since a set in $X$ need not has the properties of the set in $\mathbb{R}^m$,
 the proof of the conclusions in this paper requires deep understandings of the problem.

The third part is the completions of several common fuzzy set spaces under the $H_{\rm end}$ metric and the $H_{\rm send}$ metric, respectively.

Let
 $\widetilde{X}$ denote the completion of $X$.
We show that the
space $(P^1_{USCB} (\widetilde{X}), H_{\rm send})$ is a completion of the fuzzy set space $(F^1_{USCB} (X), H_{\rm send})$.
Then
we show that $(F^1_{USCG} (\widetilde{X}), H_{\rm end})$
is a completion of $(F^1_{USCB} (X), H_{\rm end})$.
So, of course,
 $(F^1_{USCG} (\widetilde{X}), H_{\rm end})$ is also a completion of $(F^1_{USCG} (X), H_{\rm end})$.

These conclusions indicate that in the case of the $H_{\rm end}$ metric, a completion of
 the space of common compact fuzzy set in $X$ is the space of common fuzzy set in $\widetilde{X}$;
in the case of the $H_{\rm send}$ metric, a completion of
 the space of common compact fuzzy set in $X$
is a metric space in which the space of common compact fuzzy set in $\widetilde{X}$
can be isometrically embedded,
and
each element of which
is a nonempty compact set in $\widetilde{X}\times[0,1]$.

The conclusions for the completions of the spaces of fuzzy set in $X$
given in this paper apply to not only the cases that $X$ is a complete metric space
but also the cases that $X$ is an incomplete metric space.

The remainder of this paper is organized as follows.
In Section \ref{bas}, we recall and give some basic notions and fundamental results related to fuzzy sets
and
convergence structures on them.
 In Section \ref{fsm}, we discuss
the properties and relations of $H_{\rm send}$, $H_{\rm end}$ and Kuratowski convergence on
 $P^1_{USC} (X)$.
Based on this, we give some
 relations among
$H_{\rm end}$, $H_{\rm send}$
and $\Gamma$-convergence on $F^1_{USC} (X)$.
In Sections \ref{lcg} and \ref{lce}, we investigate the level characterizations
of
the $\Gamma$-convergence and the $H_{\rm end}$ convergence, respectively.
By using the above results, Section \ref{rem} discusses some relations
among the $d_\infty$ metric, the $d_p^*$ metric, the $H_{\rm end}$ metric
and the $H_{\rm send}$ metric.
In Section \ref{cmfuzzy},
on the basis of the conclusions in previous sections, we give characterizations of total boundedness, relative compactness and compactness
in $(F^1_{USCG} (X), H_{\rm end})$ and $(F^1_{USCB} (X), H_{\rm send})$, respectively.
In
Section 8, we give
completions of
$(F^1_{USCG} (X), H_{\rm end})$ and $(F^1_{USCB} (X), H_{\rm send})$, respectively.
At last, we draw the conclusions in Section 9.

\section{Fuzzy sets and convergence structures on them} \label{bas}

In this section, we recall and give some basic notions and fundamental results related to fuzzy sets
and
convergence structures on them.
Readers
can refer to \cite{wu, du,kle} for related contents.

A fuzzy set $u$ in $X$ can be seen as a function $u:X \to [0,1]$.
A
subset $S$ of $X$ can be seen as a fuzzy set in $X$. If there is no confusion,
 the fuzzy set corresponding to $S$ is often denoted by $\chi_{S}$; that is,
\[ \chi_{S} (x) = \left\{
                    \begin{array}{ll}
                      1, & x\in S, \\
                      0, & x\in X \setminus S.
                    \end{array}
                  \right.
\]
For simplicity,
for
$x\in X$, we will use $\widehat{x}$ to denote the fuzzy set  $\chi_{\{x\}}$ in $X$.
In this paper, if we want to emphasize a specific metric space $X$, we will write the fuzzy set corresponding to $S$ in $X$ as
$S_{F(X)}$, and the fuzzy set corresponding to $\{x\}$ in $X$ as $\widehat{x}_{F(X)}$.

The symbol $F(X)$ is used
to
denote the set of
all fuzzy sets in $X$.
For
$u\in F(X)$ and $\al\in [0,1]$, let $\{u>\al \} $ denote the set $\{x\in X: u(x)>\al \}$, and let $[u]_{\al}$ denote the \emph{$\al$-cut} of
$u$, i.e.
\[
[u]_{\al}=\begin{cases}
\{x\in X : u(x)\geq \al \}, & \ \al\in(0,1],
\\
{\rm supp}\, u=\overline{    \{ u > 0 \}    }, & \ \al=0,
\end{cases}
\]
where $\overline{S}$
denotes
the topological closure of $S$ in $(X,d)$.

For
$u\in F(X)$,
define
\begin{gather*}
{\rm end}\, u:= \{ (x, t)\in  X \times [0,1]: u(x) \geq t\},
\\
{\rm send}\, u:= \{ (x, t)\in  X \times [0,1]: u(x) \geq t\} \cap  ([u]_0 \times [0,1]).
\end{gather*}
 $
{\rm end}\, u$ and ${\rm send}\, u$
 are called the endograph and the sendograph of $u$, respectively.

The symbol $K(X)$ and
 $C(X)$ are used
to
 denote the set of all nonempty compact subsets of $X$ and the set of all nonempty closed subsets of $X$, respectively.

Let
$F^1_{USC}(X)$
denote
the set of all normal and upper semi-continuous fuzzy sets $u:X \to [0,1]$,
i.e.,
$$F^1_{USC}(X) :=\{ u\in F(X) : [u]_\al \in  C(X)  \  \mbox{for all} \   \al \in [0,1]   \}.  $$

We introduce some subclasses of $F^1_{USC}(X)$, which will be discussed in this paper.
Define
\begin{gather*}
F^1_{USCB}(X):=\{ u\in  F^1_{USC}(X): [u]_0 \in K(X) \},
\\
F^1_{USCG}(X):=\{ u\in  F^1_{USC}(X): [u]_\al \in K(X) \ \mbox{for all} \   \al\in (0,1] \}.
 \end{gather*}
Clearly,
 $$F^1_{USCB}(X) \subseteq  F^1_{USCG}(X)   \subseteq  F^1_{USC}(X).$$

The set of (compact) fuzzy numbers
are
denoted by
$E^m$. It
is
defined as
$$
E^m:=
\{ u\in F^1_{USCB}(\mathbb{R}^m):  [u]_\al  \mbox{ is a convex subset of } \mathbb{R}^m \mbox{ for } \al\in [0,1] \}.
$$
Fuzzy numbers have attracted much attention from theoretical research and practical applications
\cite{da, du, wu, gong, wang2, grzegorzewski2}.

Let
$(X,d)$ be a metric space.
 We
use $\bm{H}$ to denote the \emph{\textbf{Hausdorff distance}}
on
 $C(X)$ induced by $d$, i.e.,
$$
\bm{H(U,V)}  =   \max\{H^{*}(U,V),\ H^{*}(V,U)\}
$$
for arbitrary $U,V\in C(X)$,
where
  $$
H^{*}(U,V)=\sup\limits_{u\in U}\,d\, (u,V) =\sup\limits_{u\in U}\inf\limits_{v\in
V}d\, (u,v).
$$

The metric $\overline{d}$ on $X \times [0,1]$ is defined
as
$$  \overline{d } ((x,\al), (y, \beta)) = d(x,y) + |\al-\beta| .$$
If there is no confusion, we also use $H$ to denote the Hausdorff distance on $C(X\times [0,1])$ induced by $\overline{d}$.

\begin{re}
{\rm

$\rho$ is said to be a \emph{metric} on $Y$ if $\rho$ is a function from $Y\times Y$ into $\mathbb{R}$
satisfying
positivity, symmetry and triangle inequality. At this time, $(Y, \rho)$ is said to be a metric space.

  $\rho$ is said to be an \emph{extended metric} on $Y$ if $\rho$ is a function from $Y\times Y$ into $\mathbb{R} \cup \{+\infty\} $
satisfying
positivity, symmetry and triangle inequality. At this time, $(Y, \rho)$ is said to be an extended metric space.

We can see that for arbitrary metric space $(X,d)$, the Hausdorff distance $H$ on $K(X)$ induced by $d$ is a metric.
So
the Hausdorff distance $H$ on $K(X\times [0,1])$ induced by $\overline{d}$ on $X\times [0,1]$
is a metric. In these cases, we call the Hausdorff distance the Hausdorff metric.

The Hausdorff distance $H$ on $C(X)$ induced by $d$ on $X$
is an extended metric, but probably not a metric,
because
$H(A,B)$ could be equal to $+\infty$ for certain metric space $X$ and $A, B \in C(X)$.
Clearly, if $H$ on $C(X)$ induced by $d$
is not a metric, then $H$ on $C(X\times [0,1])$ induced by $\overline{d}$
is also not a metric.
So
the Hausdorff distance $H$ on $C(X\times [0,1])$ induced by $\overline{d}$ on $X\times [0,1]$
 is an extended metric but probably not a metric.
In the cases that the Hausdorff distance $H$ is an extended metric, we call the Hausdorff distance the Hausdorff extended metric.

We can see that $H$ on $C(\mathbb{R}^m)$ is an extended metric but not a metric,
and then the same is $H$ on $C(\mathbb{R}^m\times [0,1])$.

In this paper, for simplicity,
 we refer to both the Hausdorff extended metric and the Hausdorff metric as the Hausdorff metric.

}
\end{re}

The Hausdorff metric has the following important properties.

\begin{tm} \cite{roman,kle}\label{bfc} Let $(X,d)$ be a metric space and let $H$ be the Hausdorff metric induced by $d$.
Then
\\
 ({\romannumeral1})\ $ (X,d)$ is complete $\Longleftrightarrow$  $(K(X), H)$ is complete;
 \\
  ({\romannumeral2})\ $ (X,d)$ is separable $\Longleftrightarrow$  $(K(X), H)$ is separable;
  \\
  ({\romannumeral3})\ $ (X,d)$ is compact $\Longleftrightarrow$  $(K(X), H)$ is compact.
\end{tm}

Rojas-Medar and Rom\'{a}n-Flores \cite{rojas} introduced the Kuratowski convergence
of a sequence of sets in a metric space.

Let $(X,d)$ be a metric space.
Let $C$ be a set in $X$ and
$\{C_n\}$ a sequence of sets in $X$.
 $\{C_n\}$ is said to \emph{\textbf{Kuratowski converge}} to
$C$ according to $(X,d)$, if
$$
C
=
\liminf_{n\rightarrow \infty} C_{n}
=
\limsup_{n\rightarrow \infty} C_{n},
$$
where
\begin{gather*}
\liminf_{n\rightarrow \infty} C_{n}
 =
 \{x\in X: \  x=\lim\limits_{n\rightarrow \infty}x_{n},    x_{n}\in C_{n}\},
\\
\limsup_{n\rightarrow \infty} C_{n}
=
\{
 x\in X : \
 x=\lim\limits_{j\rightarrow \infty}x_{n_{j}},x_{n_{j}}\in C_{n_j}
\}
 =
 \bigcap\limits_{n=1}^{\infty}   \overline{   \bigcup\limits_{m\geq n}C_{m}    }.
\end{gather*}
In this case, we'll write
\bm{  $C=\lim^{(K)}_{n\to\infty}C_n    $ } according to $(X,d)$.
If
there is no confusion, we will not emphasize
the metric space
$(X,d)$ and write  $\{C_n\}$ \emph{\textbf{Kuratowski converges}} to
$C$ or \bm{  $C=\lim^{(K)}_{n\to\infty}C_n    $ } for simplicity.

\begin{re} \label{ksc}
{\rm
Definition 3.1.4 in \cite{kle} gives the definitions of
$\liminf C_{n}$, $\limsup C_{n}$
 and
$\lim C_{n}$
for
 a net of subsets $\{C_n, n\in D\}$ in a topological space.
When $\{C_n, n=1,2,\ldots\}$ is
 a sequence of subsets of a metric space,
$\liminf C_{n}$, $\limsup C_{n}$
 and
$\lim C_{n}$
according to Definition 3.1.4 in \cite{kle}
are
$\liminf_{n\rightarrow \infty} C_{n}$, $\limsup_{n\rightarrow \infty} C_{n}$
and $\lim^{(K)}_{n\to\infty}C_n  $
according to
 the above definitions, respectively.

  }
\end{re}

Rojas-Medar and Rom\'{a}n-Flores \cite{rojas} have introduced
the
$\Gamma$-convergence
on
$F^1_{USCB} (\mathbb{R}^m)$ by using the Kuratowski convergence. Similarly, we can define
 the $\Gamma$-convergence of a sequence
of fuzzy sets
on $F^1_{USC} (X)$.

  Let $u$, $u_n$, $n=1,2,\ldots$, be fuzzy sets in $F^1_{USC} (X)$.
   $\{u_n\}$ is said to \bm{$\Gamma$}-\emph{\textbf{converge}}
  to
  $u$, denoted by \bm{$u = \lim_{n\to \infty}^{(\Gamma)}  u_n$},
  if
  ${\rm end}\, u= \lim_{n\to \infty}^{(K)}  {\rm end}\, u_n$
according to
$(X \times [0,1], \overline{d})$.

Let $(X,d)$ be a metric space and let $u \in  F_(X)$. Then from basic analysis, the following three properties
are
equivalent:
\begin{description}
  \item[(\romannumeral1)]  $u$ is upper semi-continuous;
  \item[(\romannumeral2)] ${\rm end}\, u$ is closed in $(X\times [0,1],   \overline{d})$;
  \item[(\romannumeral3)] ${\rm send}\, u$ is closed in $(X\times [0,1],  \overline{d})$.
\end{description}

The endograph metric $H_{\rm end} $ and the sendograph metric $H_{\rm send} $
can be defined on $F^1_{USC}(X)$ as usual.
For $u,v \in F^1_{USC}(X)$,
\begin{gather*}
\bm{  H_{\rm end}(u,v)    }: =  H({\rm end}\, u,  {\rm end}\, v ),
\\
\bm{  H_{\rm send}(u,v)    }: =  H({\rm send}\, u,  {\rm send}\, v ).
  \end{gather*}

The endograph metric $H_{\rm end}$ and the sendograph metric $H_{\rm send}$ are defined by using
the Hausdorff
metric on $C(X \times [0,1])$ induced by $\overline{d}$ on $X \times [0,1]$.

The $d_\infty$ metric
  on $ F^1_{USC} (X)$
is defined as
$$d_\infty (u,v) :=   \sup\{ H([u]_\al, [v]_\al) :\al\in [0,1]  \}.$$

\begin{pp}
Let $u,v \in F^1_{USC} (X)$. Then
\begin{equation}\label{smr}
  d_\infty(u,v) \geq H_{\rm send}(u,v) \geq H_{\rm end}(u,v).
\end{equation}
\end{pp}

\begin{proof} (The proof is routine.)
  Note that
for each $(x,\al) \in {\rm send}\, u$,
\begin{align*}
  &  \overline{d}((x,\al),  {\rm send}\, v)
= \inf_{(y, \beta) \in {\rm send}\, v} \overline{d}((x,\al), (y, \beta))  \\
&\leq \inf_{(y, \alpha) \in {\rm send}\, v} \overline{d}((x,\al), (y, \alpha))  \\
&= \inf_{(y, \alpha) \in {\rm send}\, v} d(x, y)  \\
   & =   \inf_{y\in [v]_\alpha}   d(x, y)= d(x, [v]_\al).
 \end{align*}
Thus
\begin{align} \label{sger}
& H^*({\rm send}\, u,    {\rm send}\, v) =  \sup_{(x,\al)\in {\rm send}\, u } \overline{d}((x,\al),  {\rm send}\, v)  =  \sup_{\al\in [0,1]}\sup_{(x,\al)\in {\rm send}\, u } \overline{d}((x,\al),  {\rm send}\, v)  \nonumber \\
& \leq  \sup_{\al\in [0,1]} \sup_{(x,\al)\in {\rm send}\, u }  d(x, [v]_\al)
 = \sup_{\al\in [0,1]}\sup_{x \in [u]_\al }  d(x, [v]_\al)  \nonumber \\
&\leq  \sup_{\al\in [0,1]} H^*([u]_\al, [v]_\al) \leq   d_\infty (u,v).
\end{align}
Similarly $ H^*({\rm send}\, v,    {\rm send}\, u) \leq   d_\infty (v,u) = d_\infty (u,v)$ (Clearly
$ H^*({\rm send}\, v,    {\rm send}\, u) \leq   d_\infty (v,u) $ can also be obtained by interchanging $u$ with $v$ in $H^*({\rm send}\, u,    {\rm send}\, v) \leq   d_\infty (u,v)$ in \eqref{sger}).
So we have
$$ H_{\rm send}(u,v) =  \max\{H^*({\rm send}\, u,    {\rm send}\, v), H^*({\rm send}\, v,    {\rm send}\, u)\} \leq   d_\infty (u,v).$$

The conclusion
$H_{\rm end} (u,v) \leq H_{\rm send} (u,v)$
follows from
Theorem \ref{pseu} (\romannumeral1).
\end{proof}

\begin{re}
  {\rm

We can see that $H_{\rm end}$ is a metric on $F^1_{USC}(X)$ with $H_{\rm end}(u,v) \leq 1$ for all $u,v \in F^1_{USC}(X)$.
 Both $d_\infty$
and
$H_{\rm send}$ are metrics on $F^1_{USCB}(X)$.
However,
each one of
$d_\infty$ and $H_{\rm send}$ on $F^1_{USC}(X)$ is an extended metric but probably not a metric. See also
Remark 3.3 in \cite{huang17}.

We can see that both $d_\infty$ and $H_{\rm send}$ on $F^1_{USCG}(\mathbb{R}^m)$ are not metrics, they are extended metrics.

For simplicity,
in this paper, we call
$H_{\rm send}$ on $F^1_{USC}(X)$ the $H_{\rm send}$ metric or the sendograph metric $H_{\rm send}$.
We call
$d_\infty$ on $F^1_{USC}(X)$ the $d_\infty$ metric or the supremum metric $d_\infty$.

}
\end{re}

\section{$H_{\rm end}$, $H_{\rm send}$ and Kuratowski convergence on $P^1_{USC} (X)$ } \label{fsm}

In this section, we introduce $P^1_{USC} (X)$ and its subset $P^1_{USCB} (X)$,
and
 define the $H_{\rm send}$ distance and the $H_{\rm end}$ distance on $P^1_{USC} (X)$.
We discuss
the properties and relations of $H_{\rm send}$, $H_{\rm end}$ and Kuratowski convergence on
 $P^1_{USC} (X)$.
Based on this, we give some
 relations among
$H_{\rm end}$, $H_{\rm send}$
and $\Gamma$-convergence on $F^1_{USC} (X)$.

Since $F^1_{USC} (X)$ and $F^1_{USCB} (X)$ can be seen as subsets of $P^1_{USC} (X)$ and $P^1_{USCB} (X)$, respectively,
the conclusions on $P^1_{USC} (X)$ and $P^1_{USCB} (X)$
are useful
for the discussions of fuzzy sets in this paper.

For $u \subseteq  X \times [0,1]$ and $\al\in [0,1]$,
let
 $\langle u \rangle_\al := \{ x:   (x,\al) \in u\}$.
 $P^1_{USC} (X)$ and $P^1_{USCB} (X)$
are
subsets of the power set of $X\times [0,1]$ defined by
\begin{gather*}
\begin{split}
P^1_{USC} (X): = \{u  \subseteq   X\times [0,1]:    \langle u \rangle_\al = &  \bigcap_{\beta <\al} \langle u \rangle_\beta    \  \mbox{for all} \  \al\in (0,1];
\\
&   \langle u \rangle_\al \in C(X) \ \mbox{for all} \  \al\in [0,1] \},
\end{split}
\\
\begin{split}
P^1_{USCB} (X): = \{u  \in  P^1_{USC} (X): \langle u\rangle_\al \in K(X) \ \mbox{for all} \  \al\in [0,1] \}.
\end{split}
\end{gather*}
Clearly $P^1_{USCB} (X) \subseteq P^1_{USC} (X)$.

From the basic analysis, we can obtain that
\\
(\romannumeral1) if $u\in  P^1_{USC} (X)$,
then $u\in C(X\times [0,1])$;
\\
(\romannumeral2)
if $u\in C(X\times [0,1])$, then $\langle u \rangle_\al \in C(X) \cup \{\emptyset\}$ for all $\al\in [0,1]$;
\\
(\romannumeral3)
if $u\in  P^1_{USCB} (X)$, then $u\in K(X\times [0,1])$;
\\
(\romannumeral4)
if $u\in K(X\times [0,1])$, then $\langle u \rangle_\al \in K(X) \cup \{\emptyset\}$ for all $\al\in [0,1]$.

So we have
\begin{gather*}
\begin{split}
P^1_{USC} (X) = \{u  \in  C(X\times [0,1]):   \langle u \rangle_1 \not=\emptyset,  \langle u \rangle_\al =   \bigcap_{\beta <\al} \langle u \rangle_\beta    \  \mbox{for all} \  \al\in (0,1] \},
\end{split}
\\
\begin{split}
P^1_{USCB} (X) = \{u  \in  P^1_{USC} (X):  u  \in  K(X\times [0,1])\}.
\end{split}
\end{gather*}

 We define the $H_{\rm send}$ distance and the $H_{\rm end}$ distance on $P^1_{USC} (X)$:
\begin{gather*}
 H_{\rm send} (u,v)  :=  H(u,v),
  \\
  H_{\rm end} (u,v)  :=  H(\underline{u}, \underline{v}),
\end{gather*}
where $H$ is the Hausdorff
metric on $C(X \times [0,1])$ induced by $\overline{d}$ on $X \times [0,1]$,
and
$\underline{u}  :=   u \cup  (  X \times \{0\}   )$.
Clearly,
$H_{\rm send}$ is an extended metric on $P^1_{USC} (X)$
and
$H_{\rm end} \leq 1$ on $P^1_{USC} (X)$.

If $X$ contains more than one point, then $H_{\rm end}$ is a pseudometric on $P^1_{USC} (X)$; that is,
$H_{\rm end}$
satisfies symmetry and triangle inequality,
 and
for all $u,v\in P^1_{USC}(X)$,
$0 \leq H_{\rm end}(u,v) < +\infty$ and $H_{\rm end}(u,u)=0$.
However, at this time,
$H_{\rm end}$ is a not a metric on $P^1_{USC} (X)$
because $H_{\rm end}$ do not satisfy
the
positivity. An example is given as follows.
Let $x,y\in X$ with $x\not=y$.
Define $u\in P^1_{USCB} (X)$ given by $\langle u\rangle_\al = \{x\}$ for all $\al\in [0,1]$,
and
$v\in P^1_{USCB} (X)$
 given by
$\langle v\rangle_\al = \{x\}$ for $\al\in (0,1]$ and $\langle v\rangle_0 = \{x,y\}$.
Then $u \not= v$
and
$H_{\rm end}(u,v) = 0$.

Consider
the function
$f: F^1_{USC} (X) \to P^1_{USC} (X)$
given by
$f(u) = {\rm send}\, u$.
Then
\begin{itemize}
  \item
$f$ is an isometric embedding
of
$(F^1_{USC} (X), H_{\rm send})$ in $(P^1_{USC} (X), H_{\rm send})$.

  \item    $f|_{F^1_{USCB}(X) }$ is an isometric embedding
of
$(F^1_{USCB} (X), H_{\rm send})$ in $(P^1_{USCB} (X), H_{\rm send})$.
\end{itemize}
Hence $(F^1_{USC} (X), H_{\rm send})$
can be embedded isometrically in $(P^1_{USC} (X), H_{\rm send})$,
and
$(F^1_{USCB} (X), H_{\rm send})$ can be embedded isometrically in $(P^1_{USCB} (X), H_{\rm send})$.

For $u\in  F^1_{USC} (X) $,
we define
$\overrightarrow{u}:= f(u)$. Then $\overrightarrow{u} \in P^1_{USC} (X) $.

For
$v \in P^1_{USC} (X)$, we
define
$v' \in f(F^1_{USC} (X)) \subseteq P^1_{USC} (X)$ given by
\[
\langle v' \rangle_\al =
\left\{
  \begin{array}{ll}
\langle v \rangle_\al, & \al\in (0,1],
\\
\overline{\cup_{\al>0} \langle v \rangle_\al}, & \al=0.
  \end{array}
\right.
\]
Define $\overleftarrow{v} := f^{-1} (v')$.
Then
$\overleftarrow{v}\in F^1_{USC} (X)$.

For $u,v\in F^1_{USC} (X)$, $H_{\rm send}(u,v) = H_{\rm send}(\overrightarrow{u}, \overrightarrow{v})$
and
$H_{\rm end}(u,v) = H_{\rm end}(\overrightarrow{u}, \overrightarrow{v})$.
For $u,v\in P^1_{USC} (X)$,
$H_{\rm end}(u,v) = H_{\rm end}(\overleftarrow{u}, \overleftarrow{v})$.

For a subset $U$ of $F^1_{USC}(X)$, we use
$\overrightarrow{U}$ to denote
the set
$\{\overrightarrow{u}: u\in U\}$.

For $w \in P^1_{USC} (X)$, the following are equivalent:
(\romannumeral1)
$w\in \overrightarrow{F^1_{USC}(X)}$;
(\romannumeral2)
$\langle w \rangle_0 = \overline{\bigcup_{\delta>0} \langle w\rangle_\delta}$;
(\romannumeral3)
$w=w'=\overrightarrow{\overleftarrow{w}}$;
(\romannumeral4)
$\langle w \rangle_0 = [\overleftarrow{w}]_0 $.

\begin{tm} \label{pseu}
Let $(X,d)$ be a metric space.
For $u, v \in P^1_{USC} (X)$:
\\
 (\romannumeral1) \ $H_{\rm end}(u, v) \leq H_{\rm send}(u, v)$;
\\
(\romannumeral2) \ $H(\langle u \rangle_0, \langle v\rangle_0) \leq H_{\rm send}(u, v)$;
\\
(\romannumeral3) \ If $H_{\rm end} (u,v) < 1$, then
$
  H_{\rm send}(u,v) \leq H_{\rm end}(u,v) + H(\langle u\rangle_0, \langle v\rangle_0)
$.
\\
For a sequence $\{u_n\}$ in $P^1_{USC} (X)$ and $u$ in $P^1_{USC} (X)$:
\\
(\romannumeral4) \
  $H_{\rm send}(u_n, u) \to 0$ if and only if $H_{\rm end}(u_n, u) \to 0$ and $H(\langle u_n\rangle_0, \langle u \rangle_0) \to 0$;
\\
(\romannumeral5) \ $\lim_{n\to \infty}^{(K)} u_n = u $ if and only if
$\lim_{n\to \infty}^{(K)}  \underline{u_n } = \underline{u} $
and
$ \lim_{n\to \infty}^{(K)} \langle u_n \rangle_0 = \langle u \rangle_0$.
\end{tm}

\begin{proof}
Clearly (\romannumeral1) and (\romannumeral2) are true.
The proof of (\romannumeral1) and (\romannumeral2) are routine.

To show (\romannumeral1),
let $(x,\al) \in  \underline{u}$. If $\al=0$, then $\overline{d} ((x,\al),  \underline{v}) =0 \leq H^*( u, v) $.
If $\al> 0$, then $(x,\al) \in  u$ and
$\overline{d} ((x,\al), \underline{v}) \leq H^*( u, v)  $.
From the arbitrariness
of $(x,\al)$ in $ \underline{u}$, we have
$H^*( \underline{u}, \underline{ v})  = \sup_{(x,\al)\in  \underline{u}}  \overline{d} ((x,\al),  \underline{v})  \leq H^*( u,    v) .$
Similarly
$H^*( \underline{v},     \underline{u})  \leq H^*(v,     u) .$
So
$H_{\rm end} (u,v) = H(\underline{u}, \underline{v}) \leq  H(u,v) =H_{\rm send} (u,v),$
i.e. (\romannumeral1) is true.

To show (\romannumeral2),
let $(x,0) \in u$.
If $(y,\al) \in v$, then $(y,0) \in v$ and
$\overline{d}((x,0), (y,\al) ) \geq  \overline{d}((x,0), (y,0) )$.
Hence
$\overline{d}((x,0), v) = \inf_{(y,\al)\in v} \overline{d}((x,0), (y,\alpha) )= \inf_{(y,0)\in v} \overline{d}((x,0), (y,0) ) = \inf_{y\in \langle v \rangle_0} d(x, y) = d(x, \langle v \rangle_0)$.
Since $(x,0) \in u$ if and only if $x\in \langle u \rangle_0$,
thus
$ H^*( \langle u \rangle_0,   \langle v \rangle_0) = \sup_{x\in \langle u \rangle_0}
d(x, \langle v \rangle_0) = \sup_{(x,0) \in  u} \overline{d}((x,0), v) \leq H^*(u,v) $.
Similarly $ H^*( \langle v \rangle_0,   \langle u \rangle_0) \leq  H^*(v,u) $.
So
$ H( \langle u \rangle_0,   \langle v \rangle_0) \leq  H(u,v) = H_{\rm send} (u,v)$, i.e. (\romannumeral2) is true.

To show (\romannumeral3),
let $(x,\alpha) \in  u$. Then $ d((x,\al),  v)
\leq
\al + d(x, \langle v\rangle_0)$.
If $\al> d((x,\al),  \underline{v})$,
then
$
d((x,\al),  v)
=
d((x,\al),  \underline{v}).
$
Hence
\[
d((x,\al),  v)
\leq
 \left\{
  \begin{array}{ll}
  H_{\rm end}(u, v),     & \al >  H_{\rm end}(u, v),
\\
  \al + H(\langle u \rangle_0, \langle v\rangle_0),  &   \al\in [0,1].
  \end{array}
\right.
\]
Thus for $u,v\in P^1_{USC} (X)$ with $H_{\rm end} (u,v) < 1$
\begin{equation}\label{ser}
  H_{\rm send}(u,v) \leq H_{\rm end}(u,v) + H(\langle u\rangle_0, \langle v\rangle_0).
\end{equation}
So
(\romannumeral3) is true.
We can see that if one of the following conditions (\romannumeral1) $H_{\rm end}(u,v) =0$,
  (\romannumeral2) $H(\langle u\rangle_0, \langle v\rangle_0) =0 $, and
 (\romannumeral3) $H(\langle u\rangle_0, \langle v\rangle_0) = +\infty$ hold,
then ``='' can be obtained in \eqref{ser}.
However, the converse is false.

(\romannumeral4) follows immediately from (\romannumeral1), (\romannumeral2) and (\romannumeral3).
Below we verify (\romannumeral5).

Suppose that
$\lim_{n\to \infty}^{(K)} u_n = u $. To show
$\lim_{n\to \infty}^{(K)}  \underline{u_n } = \underline{u} $
and
$ \lim_{n\to \infty}^{(K)} \langle u_n \rangle_0 = \langle u \rangle_0$,
we only need to show
that
\begin{gather*}
 \underline{u} \subseteq \liminf_{n\to\infty} \underline{u_n}, \
\limsup_{n\to\infty} \underline{u_n} \subseteq \underline{u},\\
\langle u\rangle_0 \subseteq \liminf_{n\to\infty} \langle u_n \rangle_0,\
\limsup_{n\to\infty} \langle u_n \rangle_0  \subseteq   \langle u \rangle_0.
\end{gather*}

Let $(x,\al) \in \underline{u}$. If $\al=0$, then clearly  $(x,\al) \in \liminf_{n\to\infty} \underline{u_n}$.
If
 $\al>0$, then $(x,\al) \in u$, and thus $(x,\al) \in  \liminf_{n\to\infty} u_n \subseteq \liminf_{n\to\infty} \underline{u_n}$.
So $\underline{u} \subseteq \liminf_{n\to\infty} \underline{u_n}$.

Let $(x,\al) \in \limsup_{n\to\infty} \underline{u_n}$.
If $\al=0$, then clearly  $(x,\al) \in \underline{u}$.
If
 $\al>0$, then $(x,\al) \in \limsup_{n\to\infty} u_n = u \subseteq \underline{u}$.
So $\limsup_{n\to\infty} \underline{u_n} \subseteq \underline{u}$.

Let $x\in \langle u\rangle_0 $.
Then $(x,0) \in u = \liminf_{n\to\infty} u_n $. Thus there is a sequence $\{(x_n, \al_n)\}$
such that
$(x_n, \al_n)\in u_n$, $n=1,2,\ldots$ and $(x,0)=\lim_{n\to \infty} (x_n,\al_n)$.
Hence
$x_n\in \langle u_n \rangle_0$ and $x=\lim_{n\to \infty} x_n$.
So
 $\langle u\rangle_0 \subseteq \liminf_{n\to\infty} \langle u_n \rangle_0$.

Let $x\in \limsup_{n\to\infty} \langle u_n \rangle_0$.
Then
there is a sequence $\{x_{n_i}\}$
such that
$x_{n_i} \in \langle u_{n_i} \rangle_0$, $i=1,2,\ldots$
and
$x= \lim_{i\to \infty} x_{n_i}$.
Thus
$(x,0) = \lim_{i\to \infty} (x_{n_i}, 0) \in \limsup _{n\to\infty}  u_n =u$.
Hence
$x\in \langle u \rangle_0$.
So
$\limsup_{n\to\infty} \langle u_n \rangle_0  \subseteq   \langle u \rangle_0$.

Suppose that
$\lim_{n\to \infty}^{(K)}  \underline{u_n } = \underline{u} $
and
$ \lim_{n\to \infty}^{(K)} \langle u_n \rangle_0 = \langle u \rangle_0$.
To show
$\lim_{n\to \infty}^{(K)} u_n = u $,
we only need to show that
\begin{equation*}
  u \subseteq \liminf_{n\to\infty} u_n, \  \limsup_{n\to\infty} u_n \subseteq u.
\end{equation*}

Let $(x,\al)\in u$.
If $\al=0$, then $x\in \langle u \rangle_0 = \lim_{n\to \infty}^{(K)} \langle u_n \rangle_0$.
Thus there is a sequence $\{x_n\}$ such that $x_n \in \langle u_n \rangle_0$, $n=1,2,\ldots$
and
$x = \lim_{n\to \infty} x_n$.
Hence
$(x,\alpha) = (x,0) = \lim_{n\to \infty} (x_n, 0) \in \liminf_{n\to\infty} u_n$.

If
$\al>0$, then from
 $(x,\al) \in \underline{u}= \lim_{n\to \infty}^{(K)}  \underline{u_n }$,
 there is a sequence  $\{(x_n,\al_n)\}$ such that $(x,\al) = \lim_{n\to \infty} (x_n,\al_n)$,
and
$(x_n,\al_n) \in  \underline{u_n}$ and $\al_n>0$ for $n=1,2,\ldots$.
Thus
$(x_n,\al_n) \in   u_n$, $n=1,2,\ldots$ and hence
$(x,\al) \in  \liminf_{n\to\infty} u_n $.

Let $(x,\al) \in \limsup_{n\to\infty} u_n$.
Then $(x,\al) \in \limsup_{n\to\infty} \underline{u_n} = \underline{u}$.
Note that
$x \in \limsup_{n\to\infty} \langle u_n \rangle_0 = \langle u \rangle_0$,
thus
$(x,\al) \in u$.

\end{proof}

Let $u,v \in  F^1_{USC} (X) $. Then $\overrightarrow{u}, \overrightarrow{v} \in P^1_{USC} (X) $, and the following holds
\begin{gather*}
 H_{\rm send} (u,v)  =  H_{\rm send}  (\overrightarrow{u}, \overrightarrow{v}),
  \\
  H_{\rm end} (u,v)  =  H_{\rm end} (\overrightarrow{u}, \overrightarrow{v}),
\\
H([u]_0, [v]_0) = H(\langle   \overrightarrow{u} \rangle_0, \langle   \overrightarrow{v} \rangle_0).
\end{gather*}
Thus (\romannumeral1), (\romannumeral2) and (\romannumeral3) in Theorem \ref{pseu} imply that
\begin{gather}
  H_{\rm end}(u, v) \leq H_{\rm send}(u, v),  \label{emsf}
\\
   H( [u]_0, [v]_0) \leq H_{\rm send}(u, v),  \label{umsf}
\end{gather}
 and if $H_{\rm end} (u,v) < 1$, then
\begin{equation}\label{secf}
  H_{\rm send}(u,v) \leq H_{\rm end}(u,v) + H([u]_0, [v]_0).
\end{equation}

\begin{pp} \label{sge}
Given $u$, $u_n$, $n=1,2,\ldots$ in $F^1_{USC} (X)$.
Then
\begin{enumerate}
\renewcommand{\labelenumi}{(\roman{enumi})}

      \item
  $H_{\rm send} (u_n, u) \to 0$ if and only if $  H_{\rm end} (u_n, u) \to 0 $ and  $ H([u_n]_0, [u]_0) \to 0$.

\item
 $   \lim_{n\to \infty}^{(K)} {\rm send}\, u_n= {\rm send}\,u $ if and only if
$\lim_{n\to \infty}^{(\Gamma)}  u_n = u $
and
$ \lim_{n\to \infty}^{(K)} [u_n]_0 = [u]_0$.

  \end{enumerate}
\end{pp}

\begin{proof}
  (\romannumeral1) follows immediately from
\eqref{emsf}, \eqref{umsf} and \eqref{secf}.
(\romannumeral1) can be seen as a special case of the clause (\romannumeral4) in Theorem \ref{pseu}.
 (\romannumeral2)
can be seen as
a special case of the clause (\romannumeral5) in Theorem \ref{pseu}.

\end{proof}

We suspect that
the following Theorem \ref{hkg}
is an already known conclusion, however we can't find this conclusion in the references that we can obtain.
  It can be proved in a similar fashion to Theorem 4.1 in \cite{huang}.
In this paper, we exclude the case that $C=\emptyset$.

\begin{tm} \cite{huang} \label{hkg}
Suppose that $C$, $C_n$ are sets in $C(X)$, $n=1,2,\ldots$. Then $H(C_n, C) \to 0$ implies that $\lim_{n\to \infty}^{(K)} C_n \, =C$.
\end{tm}

\begin{re} \label{hkr}
  {\rm
From Theorem \ref{hkg}, we obtain that
 for $u$, $u_n$, $n=1,2,\ldots$ in $P^1_{USC} (X)$:
(\romannumeral1) $H_{\rm end}(u_n,u) \to 0$
implies that
$\lim_{n\to \infty}^{(K)}  \underline{u_n} = \underline{u} $;
(\romannumeral2) $H_{\rm send}(u_n,u) \to 0$
implies that
$\lim_{n\to \infty}^{(K)}  u_n = u$.

So
for $u$, $u_n$, $n=1,2,\ldots$ in $F^1_{USC} (X)$:
(\romannumeral3) $H_{\rm end}(u_n,u) \to 0$
implies that
$\lim_{n\to \infty}^{(\Gamma)}  u_n = u $;
(\romannumeral4) $H_{\rm send}(u_n,u) \to 0$
implies that
$\lim_{n\to \infty}^{(K)} {\rm send}\, u_n = {\rm send}\, u$.

}
\end{re}

The converses of the implications in clauses (\romannumeral1), (\romannumeral2), (\romannumeral3) and (\romannumeral4) in Remark \ref{hkr}
are false.
Let $u=[0, +\infty)_{F(\mathbb{R})}$ and for $n=1,2,\ldots$,
let $u_n= [0,n]_{F(\mathbb{R})}$.
Then $\lim_{n\to \infty}^{(K)} {\rm send}\, u_n = {\rm send}\, u$,
but
$H_{\rm end}(u_n,u) = 1 \not\to 0$.
So combined with Proposition \ref{sge},
the converses of the implications in (\romannumeral3) and (\romannumeral4) are false,
and
thus
the converses of the implications in (\romannumeral1) and (\romannumeral2) are false.

\section{ Level characterizations of $\Gamma$-convergence } \label{lcg}

In this section, we investigate the level characterizations
of
the $\Gamma$-convergence.
It is shown that
under some conditions,
the $\Gamma$-convergence of fuzzy sets can be decomposed to
the Kuratowski convergence of certain $\al$-cuts.

We need the following conclusion, which is Lemma 2.1 in \cite{huang}.

\begin{lm} \cite{huang} \label{infe}
Let $(X,d)$ be a metric space,
and
$C_{n}$, $n=1,2,\ldots$, be a sequence of sets in $X$.
Suppose that $x\in X$.
Then
\\
(\romannumeral1) \
  $x \in \liminf_{n\rightarrow \infty} C_{n}$
  if and only if
$\lim_{n\to \infty} d(x, C_n) = 0$,
\\
(\romannumeral2) \
  $x \in \limsup_{n\rightarrow \infty} C_{n}$
  if and only if there is a subsequence $\{C_{n_k}\}$ of $\{C_n\}$
such
that
$\lim_{k\to \infty} d(x, C_{n_k}) = 0$.
\end{lm}

Rojas-Medar and Rom\'{a}n-Flores \cite{rojas} have introduced
the following useful property
of the
$\Gamma$-convergence.

\begin{tm}
\cite{rojas} \label{Gcln}
   Suppose that $u$, $u_n$, $n=1,2,\ldots$, are fuzzy sets
  in $F^1_{USC} (X)$.
  Then
$\lim_{n\to \infty}^{(\Gamma)}  u_n = u $ if and only if for
all $\al\in (0,1],$
\begin{equation} \label{gcln}
   \{u>\alpha\}
   \subseteq
\liminf_{n\to \infty}[u_n]_{\alpha}
\subseteq
\limsup_{n\to \infty}[u_n]_\alpha
\subseteq
 [u]_\alpha.
\end{equation}

\end{tm}

\begin{proof}

\textbf{Sufficiency}. \ Suppose that \eqref{gcln} is true for all $\al\in (0,1]$.
To show that
$\lim_{n\to \infty}^{(\Gamma)}  u_n = u $, it suffices to prove that
${\rm end}\, u \subseteq \liminf_{n\to\infty} {\rm end}\, u_n $
and
 $\limsup_{n\to \infty}{\rm end}\, u_n \subseteq {\rm end}\, u $.

To show that ${\rm end}\, u \subseteq \liminf_{n\to \infty} {\rm end}\, u_n $,
let
$(x,\al) \in {\rm end}\, u$. We only need to show
that
$(x,\al) \in \liminf_{n\to \infty} {\rm end}\, u_n $.
It suffices to verify,
by
Lemma \ref{infe}, that
\begin{equation}\label{lim}
\lim_{n \to \infty} \overline{d}((x,\al), {\rm end}\, u_n  ) =0.
\end{equation}
If $\al=0$, then clearly \eqref{lim} is true.
Suppose $\al\in (0,1]$. Then
for each $k\in \mathbb{N}$,
$x\in \{u > \al(1-\frac{1}{2k})    \}$,
and therefore
from \eqref{gcln},
 $x\in \liminf_{n\to \infty} [u_n]_{\al(1-\frac{1}{2k})}$.
By
Lemma \ref{infe}, $\lim_{n \to \infty} d(x, [u_n]_{\al(1-\frac{1}{2k} )} ) =0$.
So
 there is an $N$ such that
$d(x,   [u_n]_{\al(1-\frac{1}{2k} )  }    ) < \frac{1}{2k} $
for all
$n\geq N$.
Hence
$\overline{d} ( (x,\al),   {\rm end}\, u_n   )  < \frac{1}{k}  $
for all
$n\geq N$.
From the arbitrariness of $k$, we thus have that \eqref{lim} is true.

To show that $\limsup_{n\to \infty}{\rm end}\, u_n \subseteq {\rm end}\, u $, let
$(x, \al) \in \limsup_{n\to \infty}{\rm end}\, u_n$.
it suffices
to verify that $(x, \al)  \in {\rm end}\, u $.
If $\al=0$, then clearly $(x, \al)  \in {\rm end}\, u $.
Suppose that $\al>0$. Then there is an sequence $\{(x_{n_k}, \al_{n_k}) \}_{k=1}^{+\infty}$
satisfying
 $(x_{n_k}, \al_{n_k}) \in {\rm end}\, u_{n_k}$ for $k=1,2,\ldots$
and
$\lim_{k\to \infty} \overline{d} (x_{n_k}, \al_{n_k}), (x,\al)   ) = 0$.
So for each $m\in \mathbb{N}$,
$x \in \limsup_{n\to \infty} [u_n]_{\al(1-\frac{1}{m})}$,
and hence, by
\eqref{gcln},
$x \in  [u]_{\al(1-\frac{1}{m})}$.
From the arbitrariness of $m$,
we have $x\in [u]_\al$, and thus $(x,\al) \in {\rm end}\, u$.

\textbf{Necessity}. \ Suppose that $\lim_{n\to \infty}^{(\Gamma)}  u_n = u $.
Let $\al\in (0,1]$.
To show
\eqref{gcln}, we only need to verify that
$\{u>\alpha\}
   \subseteq
\liminf_{n\to \infty}[u_n]_{\alpha}$
and $
\limsup_{n\to \infty}[u_n]_\alpha
\subseteq
 [u]_\alpha$.

Suppose that
$x\in \{u>\alpha\}$.
Then
there is a $\beta>\al$ such that $(x, \beta) \in {\rm end}\, u = \liminf_{n\to \infty} {\rm end}\, u_n$.
Hence
there is a sequence $\{(x_n, \beta_n)\}_{n=1}^{+\infty}$
satisfying
 $(x_n, \beta_n) \in  {\rm end}\, u_n$, $n=1,2,\ldots$ and $\lim_{n\to \infty} \overline{d} ((x_n, \beta_n), (x, \beta )  )=0  $.
Notice that there is an $N \in \mathbb{N}$ such that
$x_n \in \{u_n >\al \}$ for all $n\geq N$.
So
$x \in \liminf_{n\to \infty}[u_n]_{\alpha}$.
Thus we have
$\{u>\alpha\} \subseteq  \liminf_{n\to \infty}[u_n]_{\alpha}$.

Suppose that $x\in \limsup_{n\to \infty}[u_n]_\alpha$.
Then
there is an sequence $\{x_{n_j}\}_{j=1}^{+\infty}$ satisfying
$x_{n_j} \in [u_{n_j}]_{\alpha}$, $j=1,2,\ldots$
and
$\lim_{j\to \infty} d(x_{n_j}, x) = 0$.
Hence
$\lim_{j\to \infty} \overline{d}((x_{n_j},\al), (x,\al)) = 0$.
Notice
$(x_{n_j},\al) \in {\rm end}\, u_{n_j}$ and therefore $(x,\al) \in \limsup_{n\to \infty}{\rm end}\, u_n = {\rm end}\, u$.
So
$x\in [u]_\al$.
Thus we have
$
\limsup_{n\to \infty}[u_n]_\alpha
\subseteq
 [u]_\alpha$.

\end{proof}

\begin{re}
{\rm
Rojas-Medar and Rom\'{a}n-Flores (Proposition 3.5 in \cite{rojas})
presented
the
statement in Theorem \ref{Gcln} when $u$, $u_n$, $n=1,2,\ldots$, are fuzzy sets in $E^m$.
Since we can't find a proof for Proposition 3.5 in \cite{rojas}, we give a proof here.
}
\end{re}

Corollary 3.2.13 in \cite{kle}
states that
for each net of subsets $\{A_n, n\in D\}$ in a topological space,
  $\liminf C_{n}$ and $\limsup C_{n}$ according to Definition 3.1.4 in \cite{kle} are closed sets.
From the fact illustrated in Remark \ref{ksc}, we know that
Corollary 3.2.13 in \cite{kle} implies
the following Theorem \ref{infc}.

Theorem \ref{infc} is Theorem 2.1 in \cite{huang}. Of course, the conclusion that
$\limsup_{n\rightarrow \infty} C_{n}$ are closed sets in $(X, d)$ in Theorem \ref{infc}
 can also be deduced from
the fact that $\limsup_{n\rightarrow \infty} C_{n}= \bigcap\limits_{n=1}^{\infty}   \overline{   \bigcup\limits_{m\geq n}C_{m}    }
$.

\begin{tm} \cite{huang, kle} \label{infc}
Let $(X,d)$ be a metric space
and
let
$\{C_{n}\}$ be a sequence of sets in $X$.
  Then
    $\liminf_{n\rightarrow \infty} C_{n}$ and $\limsup_{n\rightarrow \infty} C_{n}$ are closed sets in $(X, d)$.
\end{tm}

\begin{tm} \label{Gclnre}
   Suppose that $u$, $u_n$, $n=1,2,\ldots$, are fuzzy sets
  in $F^1_{USC} ( X )$.
  Then
 $\lim_{n\to \infty}^{(\Gamma)}  u_n = u $ if and only if for
all $\al\in (0,1],$
\begin{equation*}
  \overline{     \{u>\alpha\}     }
   \subseteq
\liminf_{n\to \infty}[u_n]_{\alpha}
\subseteq
 \limsup_{n\to \infty}[u_n]_\alpha
\subseteq
 [u]_\alpha.
\end{equation*}
\end{tm}

\begin{proof}
  The desired result follows from Theorems \ref{Gcln} and \ref{infc}.
\end{proof}

\begin{re}\label{sur}
{\rm
 Suppose that $u$, $u_n$, $n=1,2,\ldots$, are fuzzy sets
  in $F^1_{USC} ( X )$.
 If $\lim_{n\to \infty}^{(\Gamma)}  u_n = u $,
then by Theorems
\ref{infc} and \ref{Gclnre},
 $[u]_0 =  \overline{\{ u>0 \}}
   \subseteq
\liminf_{n\to \infty}[u_n]_{0}$.
In this case
$[u]_0
\subsetneqq
\liminf_{n\to \infty}[u_n]_0
$
could happen.
For example, let $u= \widehat{0}_{F(\mathbb{R})}$
and for $n=1,2,\ldots$,
define $u_n \in F^1_{USCB}(\mathbb{R})$ as
\[
 u_n = \left\{
            \begin{array}{ll}
              1, & x=0,\\
     1/n, & x\in (0,1],\\
0,& \mbox{otherwise}.
            \end{array}
          \right.
\]
Then $H_{\rm end}(u_n, u) \to 0$, and therefore from Remark \ref{hkr} $\lim_{n\to \infty}^{(\Gamma)}  u_n = u $.
And
$[u]_0 = \{0\}
\subsetneqq
[0,1]=
\liminf_{n\to \infty}[u_n]_0
$.

By Proposition \ref{sge},
$\lim_{n\to \infty}^{(\Gamma)}  u_n = u $ and
$[u]_0
\supseteq
\limsup_{n\to \infty}[u_n]_0
$
if and only if
 $   \lim^{(K)}_{n\to \infty} {\rm send}\, u_n= {\rm send}\,u $.
 }
\end{re}

Let $u\in F(X)$. Denote
\begin{gather*}
  D(u) :=         \{ \al\in (0,1):  [u]_\al \nsubseteq \overline{\{u>\al\}} \},  \\
   P(u) :=         \{ \al\in (0,1):  \overline{\{u>\al\}} \subsetneqq  [u]_\al \}.
\end{gather*}
A
 number $\al$ in $P(u)$
is
called a platform point of $u$.
Clearly $P(u) \subseteq D(u)$.
$P(u)\subsetneqq D(u) $ could happen.
See Example \ref{pdr}.

\begin{eap}\label{pdr}
{\rm
  let $u\in F(\mathbb{R})$ defined by
\[u(x)
=\left\{
   \begin{array}{ll}
    1,       & x\in (0,1),\\
     0.6,     & x\in [1,3],  \\
     0,        &  x\in \mathbb{R}\setminus (0,3].
   \end{array}
 \right.
\]
Then $P(u)=\emptyset$ and $D(u)=\{0.6 \}$.
So $P(u) \subsetneqq D(u)$.
}
\end{eap}

 Theorem 5.1 in \cite{huang} says that
$D(u)$
is at most countable when
$u\in F (\mathbb{R}^m)$.
In a similar manner, we can show the following Theorem \ref{ldc}.

\begin{tm} \label{ldc}
   Let $u\in F(l^2)$. Then the set $D(u)$
   is at most countable.
\end{tm}

\begin{proof}
  The proof is similar to
that of Theorem 5.1 in \cite{huang}.
A sketch of the proof is given below.

Similarly as in \cite{huang}, for $u\in F(l^2)$, $t\in l^2$ and $r\in \mathbb{R}^+$,
we can define
$S_{u, t, r}(\cdot, \cdot) :  \mathbf{S}^1 \times [0,1] \to \{-\infty \} \cup \mathbb{R} $
by
\[
S_{u, t, r} (e,\al)
=\left\{
   \begin{array}{ll}
 -\infty,  & \mbox{if} \  [u]_\al \cap \overline{B(t,r)} = \emptyset,
\\
\sup \{     \langle e, x-t \rangle :  x\in  [u]_\al \cap \overline{B(t,r)}  \}, & \mbox{if} \  [u]_\al \cap \overline{B(t,r)} \not= \emptyset,
   \end{array}
 \right.
\]
where
$\mathbf{S}^1 :=\{ e\in l^2: \|e\|=1\}$ and $\overline{B(t,r)}:= \{x\in l^2: \|x-t\| \leq r \}$.

Similarly as in \cite{huang},
we can define $D(u,t,r,e)$, which is the discontinuous point of $S_{u, t, r}(e, \cdot)$.

Proceed as in the proof of Lemma A.1. in \cite{huang}, we can show
the conclusion corresponding to Lemma A.1. in \cite{huang}:
$D(u,t,r) = \cup_{e\in \mathbf{S}^1} D(u,t,r,e)$ is at most countable.

Here we mention that
 the narrative of the proof of formula (A.6) in Theorem 5.1 in \cite{huang} can be slightly simplified.
The detailed operations are performed as follows:
replace
Lines 1 and 2 from the bottom in Page 82 and Lines 1 and 2 in Page 83 in \cite{huang}
by
\begin{itemize}
  \item
Since $2\langle a, b \rangle =  \| a \|^2 + \|b\|^2 - \|a-b\|^2$ for each $a,b \in \mathbb{R}^m$, then
$$
\langle e, x-q\rangle = \frac{    \langle   y-q,    x-q   \rangle    }    {  \|y-q\|    }  =  \frac{ \|x-q\|^2 + \|y-q\|^2 -\|x-y\|^2 } {2 \|y-q\| }
$$
\end{itemize}

Note that
$2\langle a, b \rangle =  \| a \|^2 + \|b\|^2 - \|a-b\|^2$ for each $a,b \in l^2$.
So proceed as in the proof of Theorem 5.1 in \cite{huang}, we obtain the conclusion corresponding to Theorem 5.1 in \cite{huang}: $D(u)$
is at most countable.

\end{proof}

The following Theorem \ref{sdc} is
a generalization of
 Theorem 5.1 in \cite{huang} and Theorem \ref{ldc}.
Its proof is based on the well-known result:
\begin{itemize}
  \item each separable metric space is homeomorphic to a subspace
of the \textbf{Hilbert space} $\bm{l^2}:= \{(x_i)_{i=1}^{+\infty}: \sum_{i=1}^{+\infty} x_i^2 < +\infty \}$.
\end{itemize}

\begin{tm}\label{sdc}
 Let $(X,d)$ be a metric space and $u \in F(X)$. If $([u]_0, d)$ is separable, then the set
 $D(u)$
 is at most countable.
\end{tm}

\begin{proof}
Let $f$ be a homeomorphism
from
$([u]_0, d)$ to a subspace
of
$l^2$.
Consider $u_f \in F(l^2)$ defined by
\[
u_f(t)
=\left\{
   \begin{array}{ll}
     u(f^{-1} (t)), & t\in  f([u]_0), \\
          0, & t\in l^2 \setminus   f([u]_0).
   \end{array}
 \right.
\]
Then by Theorem \ref{ldc}, $D(u_f)$ is at most countable.
To show
that
$D(u)$
 is at most countable, it suffices to show that
$ D(u)= D(u_f)$.

For $S \subseteq [u]_0$,
let
$ \overline{S}^{[u]_0}$ denote the topological closure of $S$ in $([u]_0, d)$.
For
 $W \subseteq f([u]_0)$,
let
 $\overline{W}^{l^2} $ denote
the topological closure of $W$ in $l^2$,
and let
$\overline{W}^{f([u]_0)} $ denote the topological closure of $W$ in $f([u]_0)$, here we see $f([u]_0)$ as a metric subspace of $l^2$.

Note that
for each $x\in [u]_0$, $u(x)=u_f(   f(x))$.
So
for each $\al\in [0,1]$ and each $x\in [u]_0$,
 $$x\in [u]_\al    \Leftrightarrow   f(x) \in [u_f]_\al, $$
and
$$x \notin \overline{\{u>\al\}} \Leftrightarrow  x \notin \overline{\{u>\al\}}^{[u]_0}
  \Leftrightarrow    f(x) \notin \overline{\{u_f>\al\}}^{f([u]_0)}
\Leftrightarrow
f(x) \notin \overline{\{u_f>\al\}}^{l^2}. $$
This
 implies
$D(u) = D(u_f)$,
since
\\
$D(u)=\{\al \in (0,1):  \mbox{ there exists } x \in [u]_\al \mbox{ such that } x \notin \overline{\{u>\al\}}     \}$, and
\\
$ D(u_f)=\{\al\in (0,1): \mbox{ there exists }  f(x) \in [u_f]_\al \mbox{ such that } f(x) \notin \overline{\{u_f>\al\}}^{l^2}   \}$.

\end{proof}

\begin{tl}\label{smc}
   Let $(X,d)$ be a separable metric space and $u \in F(X)$. Then the set
 $D(u)$
 is at most countable.
\end{tl}

\begin{proof}
  Since every subspace of a separable metric space is separable, $([u]_0,d)$ is separable. Thus,
by Theorem \ref{sdc}, $D(u)$ is at most countable.

\end{proof}

\begin{re}{\rm
It is well-known that
  both $\mathbb{R}^m$ and $l^2$ are separable metric space.
Thus
both Theorem 5.1 in \cite{huang}
and
Theorem \ref{ldc}
are special cases of Corollary \ref{smc}, which is a corollary of Theorem \ref{sdc}.
So
Theorem \ref{sdc} is a generalization
of Theorem 5.1 in \cite{huang}
and
Theorem \ref{ldc}.

}
\end{re}

\begin{re} \label{spc}
{\rm
 Theorems \ref{ldc}, \ref{sdc} and Corollary \ref{smc}
remain true if $D(u)$ is replaced by $P(u)$, since $P(u) \subseteq D(u)$.

If $u\in F^1_{USC} (X)$, then clearly $P(u)= D(u)$.
}
\end{re}

\begin{tm} \label{gusc}
    Suppose that $u$, $u_n$, $n=1,2,\ldots$, are fuzzy sets
  in $F^1_{USC} ( X )$. Then the following statements are true.
  \begin{enumerate}
  \renewcommand{\labelenumi}{(\roman{enumi})}

      \item If there is a dense set $P$ in $(0,1)$ such that $[u]_\al=\lim^{(K)}_{n\to \infty} [u_n]_\al $ for $\al\in P$, then $u = \lim_{n\to \infty}^{(\Gamma)}  u_n$.

    \item    If $u = \lim_{n\to \infty}^{(\Gamma)}  u_n $, then $[u]_\al=\lim^{(K)}_{n\to \infty} [u_n]_\al $ for all $\al\in (0,1) \setminus P(u)$.

\end{enumerate}
\end{tm}

\begin{proof}
 The proof of (\romannumeral1)
is similar to
  ``(\romannumeral2) $\Rightarrow$ (\romannumeral1)''
in the proof of
   Theorem 6.2 in \cite{huang}.
 (\romannumeral2) follows immediately from Theorem \ref{Gclnre}.

\end{proof}

The following theorem gives a condition under which
the $\Gamma$-convergence of fuzzy sets can be
decomposed to
the Kuratowski convergence of certain $\al$-cuts.

\begin{tm} \label{gdm}
 Suppose that $u$, $u_n$, $n=1,2,\ldots$, are fuzzy sets
  in $F^1_{USC} ( X )$. If $([u]_0, d)$ is separable, then the following are equivalent:

\begin{enumerate}
  \renewcommand{\labelenumi}{(\roman{enumi})}

    \item $\lim_{n\to \infty}^{(\Gamma)}  u_n = u $;

      \item $\lim^{(K)}_{n\to \infty} [u_n]_\al = [u]_\al$ holds a.e. on $\al \in (0,1)$;

    \item    $\lim^{(K)}_{n\to \infty} [u_n]_\al = [u]_\al$ holds for all $\al\in (0,1) \setminus P(u)$;

    \item     There is a dense subset $P$ of $(0,1) \backslash P(u)$ such that $\lim^{(K)}_{n\to \infty} [u_n]_\alpha = [u]_\al$ holds for
$\al\in P$;

 \item There is a countable dense subset $P$ of $(0,1) \backslash P(u)$ such that $\lim^{(K)}_{n\to \infty} [u_n]_\alpha = [u]_\al$ holds for $\al\in P$.

\end{enumerate}
\end{tm}

\begin{proof}
(\romannumeral1)$\Rightarrow$(\romannumeral3) is (\romannumeral2) of Theorem \ref{gusc}.
Clearly
 (\romannumeral3)$\Rightarrow$(\romannumeral4)$\Rightarrow$(\romannumeral5).
(\romannumeral1) of Theorem \ref{gusc} implies that
(\romannumeral5)$\Rightarrow$(\romannumeral1)
and
(\romannumeral2)$\Rightarrow$(\romannumeral1).

We shall complete the proof by showing that (\romannumeral3)$\Rightarrow$(\romannumeral2).
This follows from the fact that $P(u)$ is at most countable,
which
is pointed out by Theorem \ref{sdc} and Remark \ref{spc}.

\end{proof}

\begin{re} \label{pvr}
{\rm

By Lemma \ref{gnc}, if $u\in F^1_{USCG}(X)$ then $P(u)$ is at most countable.
So from the proof of Theorem \ref{gdm}, we know that
 Theorem \ref{gdm}
remains true if  ``$([u]_0, d)$ is separable'' is replaced by ``$u\in F^1_{USCG}(X)$''.

Here we mention that the condition ``$u\in F^1_{USCG}(X)$'' implies the condition ``$([u]_0, d)$ is separable''.
Let
$u\in F^1_{USCG}(X)$. Then for each $\al\in (0,1]$, $([u]_\al, d)$ is separable, because each compact metric space is separable.
Since $[u]_0 = \overline{\cup_{n=1}^{+\infty} [u]_{1/n} }$,
we have that
 $([u]_0, d)$ is separable.

  }
\end{re}

\begin{tl}
 \label{gdmc}

Let $(X, d)$ be a metric subspace of $\mathbb{R}^m$ or $l^2$,
and
 let $u$, $u_n$, $n=1,2,\ldots$ be fuzzy sets
  in $F^1_{USC} ( X )$. Then the following are equivalent:
\begin{enumerate}
  \renewcommand{\labelenumi}{(\roman{enumi})}

    \item $\lim_{n\to \infty}^{(\Gamma)}  u_n = u $;

      \item $\lim^{(K)}_{n\to \infty} [u_n]_\al = [u]_\al$ holds a.e. on $\al \in (0,1)$;

    \item    $\lim^{(K)}_{n\to \infty} [u_n]_\al = [u]_\al$ holds for all $\al\in (0,1) \setminus P(u)$;

    \item     There is a dense subset $P$ of $(0,1) \backslash P(u)$ such that $\lim^{(K)}_{n\to \infty} [u_n]_\alpha = [u]_\al$ holds for
$\al\in P$;

 \item There is a countable dense subset $P$ of $(0,1) \backslash P(u)$ such that $\lim^{(K)}_{n\to \infty} [u_n]_\alpha = [u]_\al$ holds for $\al\in P$.

\end{enumerate}

\end{tl}

\begin{proof}
   Since $\mathbb{R}^m$ and $l^2$ are separable spaces, and every subspace of a separable metric space is separable,
then
 $([u]_0, d)$ is separable.
Thus the desired results follow immediately from Theorem \ref{gdm}.

Here we mention that $\mathbb{R}^m$ can be seen as a metric subspace of $l^2$.

\end{proof}

\begin{re}\label{gdn}
{\rm
Example \ref{rcp} shows
that
there exist $u$ and $u_n$, $n=1,2,\ldots$ in
$ F^1_{USC} (\prod_{x\in (0,1]} [0,3])$
such that
\\
(\romannumeral1) $  H_{\rm send} (u_n, u) \to 0$,
\\
(\romannumeral2) $(0,1) \setminus P(u) = \emptyset$,
\\
(\romannumeral3) $\{[u_n]_\al \}$ does not Kuratowski converge
to
$[u]_\al$ when $\al\in (0,1]$,
\\
(\romannumeral4) $u = \lim_{n\to \infty}^{(\Gamma)}  u_n$.

The above (\romannumeral1)-(\romannumeral3) are shown in Example \ref{rcp}.
From
Proposition \ref{sge} and
Remark \ref{hkr}, the $H_{\rm send}$ convergence implies the $\Gamma$-convergence on $F^1_{USC}(\prod_{x\in (0,1]} [0,3])$.
Hence (\romannumeral1) implies (\romannumeral4).

From (\romannumeral2),
for each $\{v_n\}$ in $ F^1_{USC} (\prod_{x\in (0,1]} [0,3])$,
 the statement
   ``$\{[v_n]_\al \}$ Kuratowski converges
to
$[u]_\al$ for $\al \in (0,1) \setminus P(u)$ '' is true.
However
``$u = \lim_{n\to \infty}^{(\Gamma)} v_n$'' is not necessarily hold.

So from
Example \ref{rcp} we know:
the converse
 of the implication in statement (\romannumeral1)
of
Theorem \ref{gusc} does not hold;
 the converse of the implication in statement (\romannumeral2) of
Theorem \ref{gusc} does not hold;
 the condition ``$([u]_0, d)$ is separable''
can not be deleted in
Theorem \ref{gdm}.

}
\end{re}

\section{Level characterizations of endograph metric convergence}\label{lce}

In this section, we discuss the level characterizations of endograph metric convergence.
It is shown that under some condition,
the $H_{\rm end}$ metric convergence of fuzzy sets
can be decomposed to
the Hausdorff metric convergence of certain $\al$-cuts.

Let $u$ be a fuzzy set in $F^1_{USC} (X)$. Denote
$$P_0(u):   =\{\al\in (0,1) :   \lim_{\beta \to \al}  H([u]_\beta,  [u]_\al) \not= 0 \}.$$
Clearly
$P(u) \subseteq   P_0(u)$ for $u\in F^1_{USC} (X)$.
$P(u) \subsetneqq   P_0(u)$ could happen.
See Example \ref{rse}.

\begin{eap} \label{rse}
{\rm
Let $u \in F^1_{USC} (\mathbb{R}^2)$ defined by
\[
[u]_\al=\{0\} \cup  \{z: \arg z \in [\al,1]\} \ \mbox{for each} \ \al\in [0,1],
\]
here we write each $(x,y) \in \mathbb{R}^2$ as a complex number $z = x+iy$.
Then $P(u) = \emptyset$ and $P_0(u) = (0,1)$. So $P(u) \subsetneqq   P_0(u)$.

This example also shows that for $u\in F^1_{USC}(X)$,
$P_0(u)$ need not be at most countable even $X$ is a separable metric space.
  }
\end{eap}

\begin{lm} \cite{huang17} \label{c}
Let $U_n \in K(X)$ for $n=1,2,\ldots$.
\\
(\romannumeral1) \ If $U_1\supseteq U_2 \supseteq \ldots \supseteq U_n\supseteq \ldots$, then
$U=\bigcap_{n=1}^{+\infty}  U_n    \in K(X)$ and $H(U_n, U) \to 0$ as $n\to +\infty$.
\\
(\romannumeral2) \ If $V_1\subseteq V_2 \subseteq \ldots \subseteq V_n\subseteq \ldots$
and
$V = \overline{\bigcup_{n=1}^{+\infty}  V_n }   \in K(X)$, then $H(V_n, V) \to 0$ as $n\to +\infty$.
\end{lm}

\begin{proof}
This is Lemma 4.4 in \cite{huang17}.
Here we give a proof using Theorem \ref{bfc}.

Since for $n=1,2,\ldots$, $U_n$ is closed in $X$, then $U$ is a closed subset of $U_1\in K(X)$. Hence $U\in K(X)$.
Since $U_n \in K(U_1)$ for $n=1,2,\ldots$, then by Theorem \ref{bfc}, $\{U_n\}$ has a subsequence which
converges to
$D\in K(U_1)$.
Then clearly $H(U_n, D) \to 0$, and thus by Theorem \ref{hkg}, $D = \lim^{(K)}_{n\to\infty} U_n = U$.
So (\romannumeral1) is true.

Since $V_n \in K(V)$ for $n=1,2,\ldots$, then by Theorem \ref{bfc}, $\{V_n\}$ has a subsequence which
converges to
$C\in K(V)$.
Then clearly $H(V_n, C) \to 0$, and thus by Theorem \ref{hkg}, $C = \lim^{(K)}_{n\to\infty} V_n = V$.
So (\romannumeral2) is true.

In chinaXiv:202107.00011v3, which is an early version of this paper submitted in 2021-08-01,
we gave this proof of (\romannumeral2) and pointed out that (\romannumeral1) can be shown in a similar way.

\end{proof}

\begin{lm} \label{gnc}
 For $u\in F^1_{USCG}(X)$:
\\
(\romannumeral1) \  $\lim_{\beta\to \al-} H([u]_\beta, [u]_\al) =0$ holds for $\al\in (0,1]$;
\\
(\romannumeral2) \ $\lim_{\gamma\to \al+} H([u]_\gamma, \overline{\{u>\al\}}) =0$ holds for $\al\in (0,1)$;
\\
(\romannumeral3) \ $\lim_{\delta\to \al} H([u]_\delta, [u]_\al) =0$ holds for $\al\in (0,1)\setminus P(u)$ and $P(u) =   P_0(u)$;
\\
(\romannumeral4) \  $P_0 (u)$ is at most countable.
\end{lm}

\begin{proof}
  Lemma \ref{c} implies (\romannumeral1) and (\romannumeral2).
 (\romannumeral1) and (\romannumeral2) imply (\romannumeral3).
(\romannumeral4) is Lemma 6.12 in \cite{huang17}.

\end{proof}

\begin{re}
{\rm

(\romannumeral4) in Lemma \ref{gnc}
can also be shown in such a way:
let $u\in F^1_{USCG}(X)$, then by
Remark \ref{pvr}, $([u]_0, d)$ is separable, and thus by Theorem \ref{sdc} and Remark \ref{spc},
$P(u)$ is countable. So from (\romannumeral3) in Lemma \ref{gnc},
we obtain that
 $P_0 (u)$ is at most countable.
}
\end{re}

\begin{tm} \label{lre}
  Let $u$, $u_n$, $n=1,2,\ldots$, be fuzzy sets in $F^1_{USC} (X)$.
\\
(\romannumeral1) \ $\lim_{n\to \infty} H^*({\rm end}\, u, {\rm end}\, u_n) = 0$ if and only if
for each $\al\in [0,1)$ and $\xi\in (0, 1-\alpha]$,
$\lim_{n\to \infty} H^*([u]_{\al+\xi},    [u_n]_\al) = 0$.
\\
(\romannumeral2) \ $\lim_{n\to\infty} H^*{(\rm end}\, u_n, {\rm end}\, u) = 0$ if and only if
for each $\al\in (0,1]$ and $\zeta \in (0, \al]$,
$\lim_{n\to \infty} H^*([u_n]_\al, [u]_{\al-\zeta}) = 0$.

\end{tm}

\begin{proof}
We only prove (\romannumeral1). (\romannumeral2) can be proved similarly.

\textbf{Necessity}.
Assume that $\lim_{n\to \infty} H^*({\rm end}\, u, {\rm end}\, u_n) = 0$. Let $\al\in [0,1)$ and $\xi\in (0, 1-\alpha]$.
Then for each $\varepsilon \in (0, \xi)$, there
exists an $N(\varepsilon)$ such that for all $n\geq N$,
\begin{equation*}
H^*({\rm end}\, u, {\rm end}\, u_n)  < \varepsilon,
\end{equation*}
and then
$$H^*( [u]_{\al+\xi},  [u_n]_\al) < \varepsilon.$$
From the arbitrariness
of  $\varepsilon$ in $(0, \xi)$,
we have
$\lim_{n\to \infty} H^*([u]_{\al+\xi},    [u_n]_\al) = 0$.

{\textbf{Sufficiency}}.
Let $\varepsilon>0$. Select a $k \in \mathbb{N}$ with $2/k< \varepsilon$.
From (\romannumeral1), we have that for $l=2,\ldots, k$,
$\lim_{n\to\infty} H^*(  [u]_{l/k},  [u_n]_{(l-1)/k}    ) = 0$.
So
there is an $N(\varepsilon)$ such that for all $n\geq N$ and $l=2,\ldots, k$,
\begin{equation}\label{lse}
  H^*(  [u]_{l/k},  [u_n]_{(l-1)/k}    )  < \varepsilon.
\end{equation}

Let $(x,\al) \in {\rm end}\, u$. If $\al\leq \varepsilon$,
then
$d((x,\al), {\rm end}\, u_n) \leq \varepsilon$.
Suppose $\al > \varepsilon$.
Then we can
choose $l \in \{2, \ldots, k-1\}$ such that $l/k < \al  \leq (l+1)/k$.
Hence by \eqref{lse}, for $n\geq N$,
\begin{align*}
 & d((x,\al),  {\rm end}\, u_n ) \\
& \leq d(x, [u_n]_{(l-1)/k} ) + 2/k  \\
& < H^*([u]_{l/k},  [u_n]_{(l-1)/k}   ) + \varepsilon < 2\varepsilon.
\end{align*}
From the arbitrariness of $(x,\al) \in {\rm end}\, u$, it follows that
$H^* ( {\rm end}\, u, {\rm end}\, u_n ) < 2\varepsilon$ for all $n\geq N$.

Thus
$\lim_{n\to\infty} H^* ( {\rm end}\, u, {\rm end}\, u_n ) = 0$
from the arbitrariness of
 $\varepsilon>0$.

\end{proof}

\begin{tm} \label{lres}
  Let $u$, $u_n$, $n=1,2,\ldots$, be fuzzy sets in $F^1_{USC} (X)$.

(\romannumeral1) \ The following are equivalent:
\\
(\romannumeral1-1) \
$\lim_{n\to \infty} H^*({\rm end}\, u, {\rm end}\, u_n) = 0$;
\\
(\romannumeral1-2) \
For each $\al\in [0,1)$ and $\xi\in (0, 1-\alpha]$,
$\lim_{n\to \infty} H^*([u]_{\al+\xi},    [u_n]_\al) = 0$;
\\
(\romannumeral1-3) \
There is a dense subset $P$ of $[0,1)$
such that
for each $\al\in P$ and $\xi\in (0, 1-\alpha]$,
$\lim_{n\to \infty} H^*([u]_{\al+\xi},    [u_n]_\al) = 0$.

(\romannumeral2) \  The following are equivalent:
\\
(\romannumeral2-1) \
$\lim_{n\to\infty} H^*{(\rm end}\, u_n, {\rm end}\, u) = 0$;
\\
(\romannumeral2-2) \
for each $\al\in (0,1]$ and $\zeta \in (0, \al]$,
$\lim_{n\to \infty} H^*([u_n]_\al, [u]_{\al-\zeta}) = 0$;
\\
(\romannumeral2-3) \
There is a dense subset $P$ of $(0,1]$
such that for each $\al\in P$ and $\zeta \in (0, \al]$,
$\lim_{n\to \infty} H^*([u_n]_\al, [u]_{\al-\zeta}) = 0$.

\end{tm}

\begin{proof} \
Let
  $\al\in [0,1)$ and $\xi\in (0, 1-\alpha]$.
Choose $\beta \in P \cap [\al, \al+\xi)$.
Then
\begin{equation*}
  H^*([u]_{\al+\xi},    [u_n]_\al) \leq H^*([u]_{\al+\xi},    [u_n]_\beta).
\end{equation*}
Using this fact, we see that
(\romannumeral1-3)$\Rightarrow$(\romannumeral1-2).

(\romannumeral1-2)$\Rightarrow$(\romannumeral1-3) is obvious.
From Theorem \ref{lre}, we have
(\romannumeral1-1)$\Leftrightarrow$(\romannumeral1-2).
Thus (\romannumeral1) is proved.
Similarly, we can prove (\romannumeral2).

\end{proof}

\begin{re}{\rm
 Let $u$, $u_n$, $n=1,2,\ldots$, be fuzzy sets in $F^1_{USC} (X)$.
Clearly, (\romannumeral1-2) in Theorem \ref{lres} is equivalent to the following (\romannumeral1-2)$'$,
and
(\romannumeral2-2) in Theorem \ref{lres} is equivalent to the following (\romannumeral2-2)$'$:
\\
(\romannumeral1-2)$'$ \
For each $\al\in [0,1)$ there is a sequence $\{\xi_m\}$ in $(0, 1-\alpha]$ with $\xi_m \to 0+ $ satisfying
that
$\lim_{n\to \infty} H^*([u]_{\al+\xi_m},    [u_n]_\al) = 0$;
\\
(\romannumeral2-2)$'$ \
For each $\al\in (0,1]$ there is a sequence $\{\zeta_m\}$ in $(0, \al]$ with $\zeta_m \to 0+ $ satisfying
that
$\lim_{n\to \infty} H^*([u_n]_\al, [u]_{\al-\zeta_m}) = 0$.

Similarly, we can give
(\romannumeral1-3)$'$ and (\romannumeral2-3)$'$ which are equivalent to (\romannumeral1-3) and (\romannumeral2-3) in Theorem \ref{lres},
respectively. They are listed below.
\\
(\romannumeral1-3)$'$ \
There is a dense subset $P$ of $[0,1)$
such that
for each $\al\in P$ there exists a sequence $\{\xi_m\}$  in $(0, 1-\alpha]$ with $\xi_m \to 0+ $ satisfying
that
$\lim_{n\to \infty} H^*([u]_{\al+\xi_m},    [u_n]_\al) = 0$.
\\
(\romannumeral2-3)$'$ \
There is a dense subset $P$ of $(0,1]$
such that for each $\al\in P$  there exists a sequence $\{\zeta_m\}$ in $(0,\al]$ with $\zeta_m \to 0+ $ satisfying
that
$\lim_{n\to \infty} H^*([u_n]_\al, [u]_{\al-\zeta_m}) = 0$.

}
\end{re}

\begin{tl} \label{lrec}
  Let $u$, $u_n$, $n=1,2,\ldots$, be fuzzy sets in $F^1_{USC} (X)$.
Then
the following are equivalent:
\\
(\romannumeral1) \
$\lim_{n\to\infty} H_{\rm end} (u_n, u) = 0$;
\\
(\romannumeral2) \ For each $\al\in (0,1)$,
$\lim_{n\to\infty} H^*([u]_{\al+\xi},    [u_n]_\al) = 0$ when $\xi\in (0, 1-\alpha]$,
and
$\lim_{n\to\infty} H^*([u_n]_\al, [u]_{\al-\zeta}) = 0$ when $\zeta \in (0, \al]$;
\\
(\romannumeral3) \ There is a dense subset $P$ of $(0,1)$
such that
for each $\al\in P$,
$\lim_{n\to\infty} H^*([u]_{\al+\xi},    [u_n]_\al) = 0$ when $\xi\in (0, 1-\alpha]$,
and
$\lim_{n\to\infty} H^*([u_n]_\al, [u]_{\al-\zeta}) = 0$ when $\zeta \in (0, \al]$.
\end{tl}

\begin{proof} The desired result follows from Theorem \ref{lres}. The proof is routine.

Note that $\lim_{n\to\infty} H_{\rm end} (u_n, u) = 0$ if and only if
$\lim_{n\to \infty} H^*({\rm end}\, u, {\rm end}\, u_n) = 0$
and
$\lim_{n\to\infty} H^*{(\rm end}\, u_n, {\rm end}\, u) = 0$.
So
 Theorem \ref{lres} implies that (\romannumeral1)$\Rightarrow$(\romannumeral2)
and
 (\romannumeral3)$\Rightarrow$(\romannumeral1).
Clearly (\romannumeral2)$\Rightarrow$(\romannumeral3).
So
(\romannumeral1)$\Leftrightarrow$(\romannumeral2)$\Leftrightarrow$(\romannumeral3).

\end{proof}

\begin{lm} \label{emr}
 Let $u$, $u_n$, $n=1,2,\ldots$, be fuzzy sets in $F^1_{USC} (X)$, and
 let $P$ be a dense subset of
$(0,1)$.
\\
(\romannumeral1) \
  If for each $\al\in P$,
$\lim_{n\to\infty} H^*(\overline{\{u>\al\}},    [u_n]_\al) = 0$,
then
$\lim_{n\to \infty} H^*({\rm end}\, u, {\rm end}\, u_n) = 0$;
\\
(\romannumeral2) \
  If for each $\al\in P$,
$\lim_{n\to\infty} H^*([u_n]_\al, [u]_{\al}) = 0$,
then
$\lim_{n\to\infty} H^*{(\rm end}\, u_n, {\rm end}\, u) = 0$;
\\
(\romannumeral3) \
  If for each $\al\in P$,
$\lim_{n\to\infty} H^*(\overline{\{u>\al\}},    [u_n]_\al) = 0$
and
$\lim_{n\to\infty} H^*([u_n]_\al, [u]_{\al}) = 0$,
then
$\lim_{n\to\infty} H_{\rm end} (u_n, u) = 0$.
\end{lm}

\begin{proof}  The desired results follow from Theorem \ref{lres}. The proof is routine.

Clearly
for each $\al\in [0,1)$,
$\overline{\{u>\al\}} \supseteq [u]_{\al+\xi}$ when $\xi\in (0, 1-\alpha]$.
Thus the assumption
for each $\al\in P$,
$\lim_{n\to\infty} H^*(\overline{\{u>\al\}},    [u_n]_\al) = 0$
 implies
for each $\al\in P$ and $\xi\in (0, 1-\alpha]$,
$\lim_{n\to \infty} H^*([u]_{\al+\xi},    [u_n]_\al) = 0$. Then $u$, $u_n$, $n=1,2,\ldots$ satisfy
 the condition (\romannumeral1-3) in
Theorem \ref{lres}.
Hence by Theorem \ref{lres}, $\lim_{n\to \infty} H^*({\rm end}\, u, {\rm end}\, u_n) = 0$.
So (\romannumeral1) is true.

Clearly for each $\al\in (0,1]$,
$[u]_{\al}\subseteq
 [u]_{\al-\zeta}$ when $\zeta \in (0, \al]$.
Thus the assumption
for each $\al\in P$,
$\lim_{n\to\infty} H^*([u_n]_\al, [u]_{\al}) = 0$
implies that
for each $\al\in P$,
$\lim_{n\to\infty} H^*([u_n]_\al, [u]_{\al-\zeta}) = 0$ when $\zeta \in (0, \al]$.
 Then $u$, $u_n$, $n=1,2,\ldots$ satisfy the condition (\romannumeral2-3) in
Theorem \ref{lres}.
Hence by Theorem \ref{lres},
$\lim_{n\to\infty} H^*{(\rm end}\, u_n, {\rm end}\, u) = 0$.
So (\romannumeral2) is true.

(\romannumeral3) follows immediately from (\romannumeral1) and (\romannumeral2) since
$\lim_{n\to\infty} H_{\rm end} (u_n, u) = 0$ if and only if
$\lim_{n\to \infty} H^*({\rm end}\, u, {\rm end}\, u_n) = 0$
and
$\lim_{n\to\infty} H^*{(\rm end}\, u_n, {\rm end}\, u) = 0$.

We can see that
(\romannumeral3) follows immediately from Corollary \ref{lrec}.
This is another way to show (\romannumeral3).

\end{proof}

\begin{lm} \label{acm}
   Let $u$, $u_n$, $n=1,2,\ldots$, be fuzzy sets in $F^1_{USC} (X)$.
\\
(\romannumeral1) \ Let $\al\in [0,1)$. If $\lim_{n\to\infty} H^*({\rm end}\, u, {\rm end}\, u_n) = 0$, and
$\lim_{\gamma \to \alpha+} H( [u]_{\gamma}, \overline{\{u>\al\}}) = 0$,
then
 $\lim_{n\to\infty} H^*(\overline{\{u>\al\}},    [u_n]_\al) = 0$.
\\
(\romannumeral2) \ Let $\al\in (0,1]$. If $\lim_{n\to\infty} H^*{(\rm end}\, u_n, {\rm end}\, u) = 0$,
and
$\lim_{\beta \to \alpha-} H( [u]_\al, [u]_{\beta}) = 0$,
then
 $\lim_{n\to\infty} H^*( [u_n]_\al, [u]_\al) = 0$.
\end{lm}

\begin{proof} \ We only prove (\romannumeral1). (\romannumeral2) can be proved similarly.

  Let $\varepsilon>0$. Since $\lim_{\gamma \to \alpha+} H( [u]_{\gamma}, \overline{\{u>\al\}}) = 0$,
then there is a $\gamma( \alpha) \in (\al, 1]$
such that
$H(\overline{\{u>\al\}}, [u]_\gamma)<\varepsilon/2$.
By Theorem \ref{lre} (\romannumeral1),  $\lim_{n\to\infty} H^*({\rm end}\, u, {\rm end}\, u_n) = 0$ implies that
$\lim_{n\to \infty} H^*([u]_{\gamma},    [u_n]_\al) = 0$.
Then there is an $N\in \mathbb{N}$ such that
for all $n\geq N$,
$H^*([u]_{\gamma},    [u_n]_\al) < \varepsilon/2$.
Hence
for all $n\geq N$,
$$H^*(\overline{\{u>\al\}},    [u_n]_\al) \leq H(\overline{\{u>\al\}}, [u]_\gamma)  + H^*([u]_{\gamma},    [u_n]_\al) < \varepsilon.$$
From the arbitrariness of $\varepsilon>0$,
we thus have
$\lim_{n\to\infty} H^*(\overline{\{u>\al\}},    [u_n]_\al) = 0$.

\end{proof}

The assumption that
$\lim_{\gamma \to \alpha+} H( [u]_{\gamma}, \overline{\{u>\al\}}) = 0$
in
(\romannumeral1) of
Lemma \ref{acm} can not be omitted.
The assumption that
$\lim_{\beta \to \alpha-} H( [u]_\al, [u]_{\beta}) = 0$
in
(\romannumeral2) of
Lemma \ref{acm} also can not be omitted.
The following Examples \ref{snc} and \ref{fnc}
are counterexamples.

\begin{eap} \label{snc}
  {\rm
Define fuzzy sets $u$ and $u_n$, $n=1,2,\ldots$ in $F^1_{USC}(\mathbb{\mathbb{R}})$ given by
\[
[u]_\al
=\left\{
   \begin{array}{ll}
(-\infty, \frac{\frac{2}{3}}{\alpha-\frac{1}{3}}], & \al\in (\frac{1}{3}, 1], \\
    (-\infty, +\infty), & \al\in [0,\frac{1}{3}],
   \end{array}
 \right.
\]
and
\[
[u_n]_\al
=\left\{
   \begin{array}{ll}
(-\infty, \frac{1-\frac{1}{3}\frac{n-1}{n}}{\alpha-\frac{1}{3}\frac{n-1}{n}}], & \al\in (\frac{1}{3}\frac{n-1}{n}, 1], \\
    (-\infty, +\infty), & \al\in [0,   \frac{1}{3}\frac{n-1}{n}],
   \end{array}
 \right.
n=1,2,\ldots.
\]
Clearly
 $\lim_{\gamma \to \frac{1}{3}+} H( [u]_{\gamma}, \overline{\{u>\frac{1}{3}\}})
= \lim_{\gamma \to \frac{1}{3}+} H( [u]_{\gamma}, (-\infty, +\infty))
= +\infty \not\to 0$.

It can be seen that
$H_{\rm end} (u, u_n) \to 0$.
However
$$H^*(\overline{\{u > \frac{1}{3}\}},   [u_n]_\frac{1}{3} )
=H^*([u]_\frac{1}{3},   [u_n]_\frac{1}{3} )
=
H^*((-\infty, +\infty),\ (-\infty, \frac{3-\frac{n-1}{n}}{1-\frac{n-1}{n}}]) =  +\infty \not\to 0.$$

}
\end{eap}

\begin{eap} \label{fnc} {\rm
Define fuzzy sets $u$ and $u_n$, $n=1,2,\ldots$ in $F^1_{USC}(\mathbb{\mathbb{R}})$ given by
\[
[u]_\al
=
\left\{
  \begin{array}{ll}
    \{1\}, & \al=1, \\
   \{1\} \cup (-\infty, -\frac{1}{1-\al}], & \al\in [0,1),
  \end{array}
\right.
\]
and
\[
[u_n]_\al
 =
\left\{
  \begin{array}{ll}
    [u]_\al, & \al\in [0, 1-\frac{1}{n}], \\
   \mbox{} [u]_{1-\frac{1}{n}}, &\al\in [1-\frac{1}{n},1],
  \end{array}
\right.
n=1,2,\ldots.
\]
Clearly $\lim_{\beta \to 1-} H( [u]_1, [u]_{\beta}) = +\infty \not\to 0$.

We can see that
$H_{\rm end} (u, u_n) \to 0$.
However
$$H^*([u_n]_1, [u]_1) = H^*( \{1\} \cup (-\infty, -n],\  \{1\}) =  +\infty \not\to 0.$$

}
\end{eap}

\begin{tm} \label{uscg}
  Let $u$ be a fuzzy set in $F^1_{USCG} (X)$ and let $u_n$, $n=1,2,\ldots$, be fuzzy sets in $F^1_{USC} (X)$.

(\romannumeral1) \ The following are equivalent:
\\
(\romannumeral1-1) \
$\lim_{n\to\infty} H^*({\rm end}\, u, {\rm end}\, u_n) = 0$;
\\
(\romannumeral1-2) \
For each $\al\in (0,1)$,
$\lim_{n\to\infty} H^*(\overline{\{u>\al\}},    [u_n]_\al) = 0$;
\\
(\romannumeral1-3) \
There is a dense subset $P$ of $(0,1)$ such that
for each $\al\in P$,
$\lim_{n\to\infty} H^*(\overline{\{u>\al\}},    [u_n]_\al) = 0$.

(\romannumeral2) \ The following are equivalent:
\\
(\romannumeral2-1) \
$\lim_{n\to\infty} H^*{(\rm end}\, u_n, {\rm end}\, u) = 0$;
\\
(\romannumeral2-2) \
For each $\al\in (0,1]$,
$\lim_{n\to\infty} H^*([u_n]_\al, [u]_{\al}) = 0$;
\\
(\romannumeral2-3) \
There is a dense subset $P$ of $(0, 1]$
such that for each $\al\in P$,
$\lim_{n\to\infty} H^*([u_n]_\al, [u]_{\al}) = 0$.
\end{tm}

\begin{proof} \  The desired results follow from Lemma \ref{gnc},
Lemma \ref{acm} and
Theorem \ref{lres}. The proof is routine.

From Lemma \ref{gnc} (\romannumeral2),
 for $u\in F^1_{USCG}(X)$ and $\al\in (0,1)$,
$\lim_{\gamma\to \al+} H([u]_\gamma, \overline{\{u>\al\}}) =0$.
So by Lemma \ref{acm} (\romannumeral1),
(\romannumeral1-1) implies (\romannumeral1-2).
Clearly
(\romannumeral1-2) implies (\romannumeral1-3).
We shall complete the proof of (\romannumeral1) by showing that (\romannumeral1-3)$\Rightarrow$(\romannumeral1-1).
This follows from Lemma \ref{emr} (\romannumeral1).

Similarly, we can show (\romannumeral2).

From Lemma \ref{gnc} (\romannumeral1)
for $u\in F^1_{USCG}(X)$ and $\al\in (0,1]$,
 $\lim_{\beta\to \al-} H([u]_\beta, [u]_\al) =0$.
So by Lemma \ref{acm} (\romannumeral2),
(\romannumeral2-1) implies (\romannumeral2-2).
Clearly
(\romannumeral2-2) implies (\romannumeral2-3).
We shall complete the proof of (\romannumeral2) by showing that (\romannumeral2-3)$\Rightarrow$(\romannumeral2-1).
This follows from Lemma \ref{emr} (\romannumeral2).

\end{proof}

\begin{tm} \label{uscgre}
  Let $u$ be a fuzzy set in $F^1_{USCG} (X)$ and let $u_n$, $n=1,2,\ldots$, be fuzzy sets in $F^1_{USC} (X)$.
Then the following are equivalent:
\\
(\romannumeral1) \
 $\lim_{n\to\infty} H_{\rm end}(u_n, u) = 0$;
\\
(\romannumeral2) \
For each $\al\in (0,1)$,
$\lim_{n\to\infty} H^*(\overline{\{u>\al\}},    [u_n]_\al) = 0$
and
$\lim_{n\to\infty} H^*([u_n]_\al, [u]_{\al}) = 0$;
\\
(\romannumeral3) \
There is a dense subset $P$ of
$(0,1)$ such that for each $\al\in P$,
$\lim_{n\to\infty} H^*(\overline{\{u>\al\}},    [u_n]_\al) = 0$
and
$\lim_{n\to\infty} H^*([u_n]_\al, [u]_{\al}) = 0$;
\\
(\romannumeral4) \
For each $\al\in (0,1)\setminus P_0(u)$,
$\lim_{n\to\infty} H( [u]_{\al},    [u_n]_\al) = 0$;
\\
(\romannumeral5) \
There is a dense subset $P$ of
$(0,1)\setminus P_0(u)$ such that
for each $\al\in P$,
$\lim_{n\to\infty} H( [u]_{\al},    [u_n]_\al) = 0$;
\\
(\romannumeral6) \
There is a countable dense subset $P$ of
$(0,1)\setminus P_0(u)$ such that
for each $\al\in P$,
$\lim_{n\to\infty} H( [u]_{\al},    [u_n]_\al) = 0$;
\\
(\romannumeral7) \
 $H([u_n]_\al, [u]_\al) \rightarrow 0$ holds a.e. on $\al\in (0,1)$.

\end{tm}

\begin{proof}
  The desired results follow from Lemma \ref{gnc} and Theorem \ref{uscg}.
The proof is routine.

Note that $\lim_{n\to\infty} H_{\rm end} (u_n, u) = 0$ if and only if
$\lim_{n\to \infty} H^*({\rm end}\, u, {\rm end}\, u_n) = 0$
and
$\lim_{n\to\infty} H^*{(\rm end}\, u_n, {\rm end}\, u) = 0$.
Hence Theorem \ref{uscg} implies that
(\romannumeral1)$\Rightarrow$(\romannumeral2)
and
(\romannumeral3)$\Rightarrow$(\romannumeral1).
Clearly
(\romannumeral2)$\Rightarrow$(\romannumeral3).
So
(\romannumeral1)$\Leftrightarrow$(\romannumeral2)$\Leftrightarrow$(\romannumeral3).

Since for each $\al\in (0,1)\setminus P_0(u) $,
$\overline{\{u>\al\}}  = [u]_\al$,
then (\romannumeral2)$\Rightarrow$(\romannumeral4) is true.
Clearly
 (\romannumeral4)$\Rightarrow$(\romannumeral5)$\Rightarrow$(\romannumeral6).
Since $\overline{\{u>\al\}}\subseteq [u]_\al$,
(\romannumeral6)$\Rightarrow$(\romannumeral3) is true.

From Lemma \ref{gnc}, we have $P_0 (u)$ is at most countable,
 and therefore (\romannumeral4)$\Rightarrow$(\romannumeral7).
Since $\overline{\{u>\al\}}\subseteq [u]_\al$,
(\romannumeral7)$\Rightarrow$(\romannumeral3) is true.
So the proof is completed.

\end{proof}

\begin{tm} \label{aedr}
Let $u$, $u_n$, $n=1,2,\ldots$, be fuzzy sets in $F^1_{USC} (X)$.
If there is a dense subset $P$ of $(0,1)$ such that
  $H([u_n]_\al, [u]_\al) \to   0$ for each $\al \in P$,
  then
  $H_{\rm end}(u_n, u) \to 0$.
\end{tm}

\begin{proof} We proceed by contradiction. If  $H_{\rm end}(u_n, u) \not\to 0$,
then there is an $\varepsilon>0$
such that
$H_{\rm end}(u_{n_k}, u) >   \varepsilon$
for a subsequence $\{ u_{n_k} \}$
of
$\{u_n\}$.

Suppose  $H^*({\rm end}\, u_{n_k},   {\rm end}\, u) > \varepsilon$. Then there
exists
a sequence $(x_{n_k}, \al_{n_k}) \in {\rm end}\, u_{n_k}$
such that
\begin{equation}\label{contpe}
d(    (x_{n_k}, \al_{n_k}) ,     {\rm end}\, u) > \varepsilon.
  \end{equation}
With no loss of generality we can assume   $\al_{n_k} \to \al \geq   \varepsilon$.
Pick $\beta\in P$ satisfying $\alpha \in (\beta, \beta + \varepsilon/2)$.
Then there exists $K$ such that
$   \al_{n_k}  \in  (\beta,   \beta + \varepsilon/2) $    for all $k \geq  K$.
Thus
for each $k \geq K$,
\begin{align}
d  &  (    (x_{n_k}, \al_{n_k}) ,     {\rm end}\, u)     \nonumber
\\
& \leq   d    (    (x_{n_k}, \beta) ,     {\rm end}\, u)   +   \varepsilon/2     \nonumber
\\
& \leq    H([u_{n_k}]_\beta,  [u]_\beta)   +   \varepsilon/2   \label{xne}
\end{align}
Note that $ H([u_{n_k}]_\beta,  [u]_\beta)  \to 0$, thus \eqref{contpe} contradicts \eqref{xne}.
So
the supposition is false.

For $H^*({\rm end}\, u,   {\rm end}\, u_{n_k}) > \varepsilon$,
we can similarly derive a contradiction.

\end{proof}

\begin{re}{\rm
It can be seen
that
Theorem \ref{aedr} can also be deduced from Lemma \ref{emr} (\romannumeral3).
 Fan (Lemma 1 in \cite{fan3}) proved a result of Theorem \ref{aedr} type.

}
\end{re}

\begin{tm} \label{hepc}
   Let $u$, $u_n$, $n=1,2,\ldots$, be fuzzy sets in $F^1_{USC} (X)$. If
  $H_{\rm end}(u_n, u) \to 0$,
  then
    $H([u_n]_\al, [u]_\al)  \to  0$ for each $\al \in (0,1) \setminus   P_0 (u)$.
\end{tm}

\begin{proof}
 Let $\al \in (0,1) \setminus  P_0(u)$.
Given $\varepsilon>0$.        Then there exists a $\delta (\al, \varepsilon)  \in (0, \varepsilon/2) $
such that
$[\al-\delta, \al+\delta] \subset [0,1]$
and
\begin{equation}\label{usfc}
 H([u]_\beta,  [u]_\al) <   \varepsilon / 2
\end{equation} for all $\beta \in [\al-\delta, \al+\delta]$.

From    $H_{\rm end}(u_n, u) \to 0$,
there
exists an $N(\delta)$ such that
\begin{equation}\label{unds}
H_{\rm end} (u_n, u) < \delta
\end{equation} for all $n\geq N$.
Thus
$$H^*([u_n]_\al, [u]_{\al-\delta}) <  \delta  <   \varepsilon/2.$$
So,
 for each $n\geq N$,
\begin{align}\label{len}
 & H^*([u_n]_\al, [u]_{\al})  \nonumber  \\
& \leq  H^*([u_n]_\al, [u]_{\al-\delta})   +   H([u]_\al, [u]_{\al-\delta})  \nonumber    \\
& <  \varepsilon/2 + \varepsilon/2 = \varepsilon
  \end{align}

Similarly, it follows from \eqref{unds} that
$$H^*( [u]_{\al+\delta},  [u_n]_\al) < \delta < \varepsilon/2,$$
and
then, for each $n\geq N$,
\begin{align}\label{rehn}
 & H^*([u]_\al, [u_n]_{\al})  \nonumber  \\
& \leq   H([u]_\al, [u]_{\al+\delta})   +    H^*( [u]_{\al+\delta},    [u_n]_\al )     \nonumber    \\
& <  \varepsilon/2 + \varepsilon/2=\varepsilon.
  \end{align}

Combined with \eqref{len} and \eqref{rehn},
$$H([u]_\al,  [u_n]_\al) \to 0.$$
\end{proof}

\begin{re}{\rm
Note that
for each $\al\in (0,1) \setminus P_0(u)$,
$\lim_{\lambda \to \alpha}H( [u]_\al, [u]_{\lambda}) = 0$
and
$[u]_\al = \overline{\{u>\al \}}$.
Thus
Theorem \ref{hepc} can also be deduced from Lemma \ref{acm}.

}
\end{re}

The following theorem gives a condition under
which
 the $H_{\rm end}$ convergence of fuzzy sets
can be decomposed to
the Hausdorff metric convergence of certain $\al$-cuts.

\begin{tm} \label{aec}
Let $u$ be a fuzzy set in $F^1_{USCG} (X)$ and let $u_n$, $n=1,2,\ldots$, be fuzzy sets in $F^1_{USC} (X)$.
Then
the following are equivalent:
\begin{enumerate}
\renewcommand{\labelenumi}{(\roman{enumi})}

\item
  $H_{\rm end}(u_n, u) \to 0$;

  \item
 $H([u_n]_\al, [u]_\al) \rightarrow 0$ holds a.e. on $\al\in (0,1)$;

\item     $H ([u_n]_\al, [u]_\al) \to 0 $ for all $\al\in (0,1) \setminus P_0(u)$;

    \item   There is a dense subset $P$ of $(0,1) \backslash P_0(u)$ such that $H ([u_n]_\al, [u]_\al) \to 0 $ for
$\al\in P$;

\item     There is a countable dense subset $P$ of $(0,1) \backslash P_0(u)$
such that
 $H ([u_n]_\al, [u]_\al) \to 0 $ for
$\al\in P$.

  \end{enumerate}

\end{tm}

\begin{proof}

Theorem \ref{hepc} implies that (\romannumeral1)$\Rightarrow$(\romannumeral3).
Clearly
 (\romannumeral3)$\Rightarrow$(\romannumeral4)$\Rightarrow$(\romannumeral5).
 Theorem \ref{aedr} implies that
(\romannumeral5)$\Rightarrow$(\romannumeral1)
and
(\romannumeral2)$\Rightarrow$(\romannumeral1).

We shall complete the proof by showing that (\romannumeral3)$\Rightarrow$(\romannumeral2).
This follows from the fact that $P_0(u)$ is at most countable,
which
is pointed out by Lemma \ref{gnc}.

\end{proof}

\begin{re}\label{pec}
{\rm
By Lemma \ref{gnc}, $P_0(u)= P(u)$ for $u \in F^1_{USCG} (X)$.
So
$P_0(u)$ can be replaced by $P(u)$ in Theorem \ref{aec}.
  }
\end{re}

\begin{re}{\rm
  Clearly, Theorem \ref{uscgre} implies Theorem \ref{aec}.
}
\end{re}

 $F^1_{USCCON} (X)$ is a subset of $F^1_{USC} (X) $
defined by
 $$F^1_{USCCON} (X) := \{ u\in F^1_{USC} (X) : \mbox{ for each } \al\in(0,1], \ [u]_\al \mbox{ is connected in } X  \}.$$

 $F^1_{USCGCON} (X)$ is a subset of $F^1_{USCG} (X) $ defined by
 $$F^1_{USCGCON} (X) := F^1_{USCG} (X) \cap F^1_{USCCON} (X).$$

\begin{pp} \label{cne}
Let $C$ be a nonempty compact set in $\mathbb{R}^m$ and for $n=1,2,\ldots$ let $C_n$
be a nonempty connected and closed set in $\mathbb{R}^m$.
Then
$H(C_n, C) \to 0$ if and only if $\lim_{n\to\infty}^{(K)} C_n =C$.
 \end{pp}

\begin{proof}
  From Theorem \ref{hkg}, we have that $H(C_n, C) \to 0  \Rightarrow \lim_{n\to\infty}^{(K)} C_n =C$.

Now we show that $\lim_{n\to\infty}^{(K)} C_n =C \Rightarrow H(C_n, C) \to 0$.
We prove by contradiction.
Assume that $\lim_{n\to\infty}^{(K)} C_n =C$ but $H(C_n, C) \not\to 0$.
Then
$H^*(C, C_n) \not\to 0$ or $H^*( C_n, C) \not\to 0$.
We split the proof into two cases.

Case (\romannumeral1) $H^*(C, C_n) \not\to 0$.

In this case, there is an $\varepsilon>0$ and a subsequence $\{C_{n_k}\}$ of $\{C_n\}$
such that
$H^*(C, C_{n_k}) > \varepsilon$.
Thus for each $k=1,2,\ldots$ there exists
$x_k \in C$ such that $d(x_k, C_{n_k}) > \varepsilon$.
Since $C$ is compact,
there is a subsequence $\{x_{k_l}\}$ of $\{x_k\}$ which converges to $x\in C$.
Then
there is a $L(\varepsilon)$ such that
$d(x, x_{k_l}) < \varepsilon/2$ for all $l\geq L$. Hence
$d(x, C_{n_{k_l}}) \geq d(x_{k_l}, C_{n_{k_l}}) - d(x, x_{k_l}) > \varepsilon/2$. By Lemma \ref{infe} (\romannumeral1), this contradicts $x\in C = \liminf_{n\to\infty} C_n$.

Case (\romannumeral2) $H^*( C_n, C) \not\to 0$.

In this case,
 there is an $\varepsilon>0$ and a subsequence $\{C_{n_k}\}$ of $\{C_n\}$
such that
$H^*(C_{n_k}, C) > \varepsilon$.
Thus we have the following:
\\
(a) for each $k=1,2,\ldots$ there exists
$x_{n_k} \in C_{n_k}$ such that $d(x_{n_k}, C) > \varepsilon$.

Pick $y\in C$, since $C = \liminf_{n\to\infty} C_n$, we can find a sequence $\{y_n\}$
satisfying that
$y_n \in C_n$ for
$n=1,2,\ldots$
and
$\{y_n\}$ converges to $y$.
Hence
 there is a $N (\varepsilon)$ such that for all $n\geq N$,
$d(y_n, y) < \varepsilon$.
Since $d(y_n, C) \leq d(y_n, y) $,
we have the following:
\\
 (b) for all $n\geq N$,
$d(y_n, C) < \varepsilon$.

Let $k\in \mathbb{N}$ with $n_k \geq N$.
Define a function $f_k$ from $C_{n_k}$ to $\mathbb{R}$ given by $f_k (z)=d(z, C)$ for each $z \in C_{n_k}$.
 Then $f_k$  is a continuous function on $C_{n_k}$. Since
$C_{n_k}$ is a connected set in $\mathbb{R}^m$,
$f_k (C_{n_k})$
 is a connected set in $\mathbb{R}$.
Combined this fact with the above
clauses (a) and (b), we obtain
that there exists $z_{n_k} \in C_{n_k}$ such that
\begin{equation}\label{eun}
d(z_{n_k}, C) = \varepsilon.
\end{equation}
From \eqref{eun} and the compactness of
$C$, the set
 $\{   z_{n_k}, n_k \geq N   \}$ is bounded in $\mathbb{R}^m$, and thus $\{ z_{n_k}, n_k \geq N   \}$ has a cluster point $z$.
By \eqref{eun}, $d(z, C) = \varepsilon$, which contradicts
$z\in \limsup_{n\to\infty} C_n = C$.

\end{proof}

\begin{re}
{\rm

Proposition \ref{cne}
may be known, however we can't find this conclusion in the references that we can
obtain. So we give a proof here.

Let $A$ be a nonempty compact set in $\mathbb{R}^m$ and $B$ a nonempty closed set in $\mathbb{R}^m$.
If $H(A,B) < +\infty$, then $B$ is bounded and hence a compact set in $\mathbb{R}^m$.
In fact,
 $H(A,B) < +\infty$ if and only if $B$ is a compact set in $\mathbb{R}^m$.

From the above fact we know that
 for $C$ and $C_n$, $n=1,2,\ldots$, satisfying the
assumptions
of Proposition \ref{cne},
if $H(C_n, C) \to 0$ (by Proposition \ref{cne} $H(C_n, C) \to 0$
if and only if $\lim_{n\to\infty}^{(K)} C_n =C$),
  then clearly there is an $N \in \mathbb{N}$ such that for all $n\geq N$,
$H(C_n, C) < +\infty$ and thus for all $n\geq N$, $C_n$ is compact.
}
\end{re}

The following Proposition \ref{fce} is an immediate consequence of
Proposition \ref{cne}, Theorem \ref{aec}, and Theorem \ref{gdm} and Remark \ref{pvr}.

\begin{tm} \label{fce}
Let $u$ be a fuzzy set in $F^1_{USCG} (\mathbb{R}^m)$
and for $n=1,2,\ldots$,
let $u_n$ be a fuzzy set in $F^1_{USCCON} (\mathbb{R}^m)$.
Then
$H_{\rm end} (u_n, u) \to 0$ as $n\to\infty$ if and only if $\lim_{n\to\infty}^{(\Gamma)} u_n = u$.
\end{tm}

\begin{proof}

The proof is routine.

By let $X= \mathbb{R}^m$ in Theorem \ref{aec} we have the following:
\\
(\romannumeral1) $H_{\rm end} (u_n, u) \to 0$ if and only if
  $H ([u_n]_\al, [u]_\al) \to 0 $ holds a.e. on $\al\in (0,1)$.

By let $X= \mathbb{R}^m$ in Corollary \ref{gdmc}, we have
the following:
\\
 (\romannumeral2)
$\lim_{n\to \infty}^{(\Gamma)}  u_n = u $
if and only if
  $[u]_\al=\lim^{(K)}_{n\to \infty} [u_n]_\al $ holds a.e. on $\al\in (0,1)$.

Since
for each $\al\in (0,1]$ and $n\in \mathbb{N}$, $[u]_\al \in K(\mathbb{R}^m)$,  $[u_n]_\al \in C(\mathbb{R}^m)$ and $[u_n]_\al$ is connected in $\mathbb{R}^m$.
Thus by Proposition \ref{cne}, for each
 $\alpha \in (0,1]$,
$H ([u_n]_\al, [u]_\al) \to 0$
if and only if
$[u]_\al=\lim^{(K)}_{n\to \infty} [u_n]_\al $.
Combined this fact with the above clauses (\romannumeral1) and (\romannumeral2),
we have
that
$H_{\rm end} (u_n, u) \to 0$ if and only if $\lim_{n\to \infty}^{(\Gamma)}  u_n = u$.

\end{proof}

\begin{re}
{\rm
Theorem \ref{fce} would be false if $\mathbb{R}^m$
were replaced by a general metric space $X$.

}
\end{re}

Here we mention that
for $u \in F^1_{USCG}(\mathbb{R})$ and a sequence $\{u_n\}$
in $F^1_{USCCON}(\mathbb{R})$,
$H(u_n, u) \to 0$ does not imply that
there is an $N$
satisfying that for all $n\geq N$
$u_n \in F^1_{USCGCON}(\mathbb{R})$.

The following Example \ref{mke} is a such example, which shows
that
there exists a
$u \in F^1_{USCG}(\mathbb{R})$ and a sequence $\{u_n\}$
in $F^1_{USCCON}(\mathbb{R})$
such that
\\
(\romannumeral1)
$H(u_n, u) \to 0$,
\\
(\romannumeral2) for each $n=1,2,\ldots$, $u_n\notin F^1_{USCGCON}(\mathbb{R})$.

\begin{eap} \label{mke}
  {\rm
Let $u= \widehat{1}_{F(\mathbb{R})} \in F^1_{UCCG}(\mathbb{R})$.
For $n=1,2,\ldots$,
define $u_n \in F^1_{USC}(\mathbb{R})$ as follows
\[u_n (t) = \left\{
              \begin{array}{ll}
                1, & t=1,\\
              1/n, & t\not=1.
              \end{array}
            \right.
\]
Then for $n=1,2,\ldots$,
\[ [u_n]_\al = \left\{
              \begin{array}{ll}
                \{1\}, & \al \in (1/n,1],\\
              \mathbb{R}, & \al\in [0,1/n].
              \end{array}
            \right.
\]
So for each $n=1,2,\ldots$, $u_n \in  F^1_{USCCON}(\mathbb{R})$ but $u_n \notin  F^1_{USCGCON}(\mathbb{R})$.
It can be seen that $H_{\rm end}(u, u_n) = 1/n \to 0$.

}
\end{eap}

 Theorem 9.2 in \cite{huang} discusses the compatibility of
the endograph metric $H_{\rm end}$ and the $\Gamma$-convergence.
The following Corollary \ref{fer} is an immediate corollary of Theorem 9.2 in \cite{huang}.

\begin{tl} \label{fer}
  Let $u$ be a fuzzy set in $F^1_{USCG} (\mathbb{R}^m)$
and for $n=1,2,\ldots$,
let $u_n$ be a fuzzy set in $F^1_{USCGCON} (\mathbb{R}^m)$.
Then
$H_{\rm end} (u_n, u) \to 0$ as $n \to \infty$ if and only if $\lim_{n\to\infty}^{(\Gamma)} u_n = u$.
\end{tl}
We can see that Theorem \ref{fce} is an improvement of Corollary \ref{fer}.

Corollary \ref{fer} is
the normal fuzzy set case of
 Theorem 9.2 in \cite{huang}, which is an important special case of
 Theorem 9.2 in \cite{huang}.

\begin{re}\label{edn}
{\rm

Example \ref{rcp} shows
that
there exist $u$ and $u_n$, $n=1,2,\ldots$ in
$ F^1_{USC} (\prod_{x\in (0,1]} [0,3])$
such that
\\
(\romannumeral1) $  H_{\rm send} (u_n, u) \to 0$,
\\
(\romannumeral2) $(0,1) \setminus P_0 (u) = \emptyset$,
\\
(\romannumeral3) $H([u_n]_\al, [u]_\al) =1$ for all $\al\in (0,1]$ and $n=1,2,\ldots$,
\\
(\romannumeral4) $  H_{\rm end} (u_n, u) \to 0$.

The above (\romannumeral1)-(\romannumeral3) are shown in Example \ref{rcp}.
From
Proposition \ref{sge},
 (\romannumeral1) implies (\romannumeral4).

From (\romannumeral2),
for each $\{v_n\}$ in $ F^1_{USC} (\prod_{x\in (0,1]} [0,3])$,
  the statement
   ``$H([v_n]_\al, [u]_\al)  \to  0$ for each $\al \in (0,1) \setminus   P_0 (u)$'' is true.
However
 ``$H_{\rm end}(v_n, u) \to 0$'' is not necessarily hold.

So from
Example \ref{rcp} we know:
the converse
 of the implication
of
Theorem \ref{aedr} does not hold;
 the converse of the implication in
Theorem \ref{hepc} does not hold;
``$u$ be a fuzzy set in $F^1_{USCG} (X)$''
can not be replaced by
``$u$ be a fuzzy set in $F^1_{USC} (X)$'' in Theorem \ref{aec}.

}
\end{re}

\section{Relations among metrics on $F^1_{USC}(X)$} \label{rem}

In this section, we discuss the relation
among the $H_{\rm end}$ metric,
the $H_{\rm send}$ metric, the $d_\infty$ metric, and the $d_p^*$ metric on $F^1_{USC}(X)$.
Most of the results are proved based on
 the level characterizations of $H_{\rm end}$ given in Section \ref{lce}.

For $u,v\in F^1_{USC}(X)$, the $d_p$ distance given by
$$d_p(u,v)= \left(\int_0^1 H([u]_\al, [v]_\al)^p  \,   d\al   \right)^{1/p}$$
 is well-defined if and only if
$H([u]_\al, [v]_\al)$ is a measurable function of $\al$ on $[0, 1]$.
In the sequel, we suppose that the $d_p$ distance satisfying
 $p \geq 1$.

Since $H([u]_\al, [v]_\al)$ could be a non-measurable function
of $\al$ on
$[0,1]$ (see Example 2.13 in \cite{huang17c}),
we introduce the
 $d_p^*$ distance on $F^1_{USC} (X)$, $p\geq 1$, in
\cite{huang17},
which is defined by
\begin{gather*}
\begin{split}
d_p^*(u,v) := \inf \{  \  \left(\int_0^1 f(\al)^p  \,   d\al   \right)^{1/p}
 :
& \ f \mbox{ is a measurable function from } [0,1] \mbox{ to } \mathbb{R} \cup \{+\infty\}    ;
\\
&   \  f(\al)  \geq H([u]_\al, [v]_\al) \mbox{ for } \al\in [0,1]  \  \}
\end{split}
\end{gather*}
for
$u,v \in F^1_{USC} (X)$.

$d_p^*$ is an extended metric but probably not a metric on $F^1_{USC}(X)$.
In this paper, we call $d_p^*$ on $F^1_{USC}(X)$ the $d_p^*$ metric for simplicity.
 See also
Remark 3.3 in \cite{huang17}.

Clearly for $u,v\in F^1_{USC}(X)$,
\begin{equation}\label{spr}
  d_\infty(u,v) \geq d_p^*(u,v).
\end{equation}
The proof of \eqref{spr} is routine. Set $d_\infty(u,v) =\xi \in \mathbb{R} \cup \{+\infty\}$.
Define $f: [0,1] \to \mathbb{R} \cup \{+\infty\} $ by $f(\al)= \xi$ for each $\al\in [0,1]$.
Hence $f$ is a measurable function from $[0,1]$ to $\mathbb{R} \cup \{+\infty\} $ and
$f(\al)  \geq H([u]_\al, [v]_\al)$ for $ \al\in [0,1] $.
So
$d_p^*(u,v) \leq \left(\int_0^1 f(\al)^p  \,   d\al   \right)^{1/p} = \xi $.
Thus \eqref{spr} is true.

If $d_p(u,v)$ is well-defined for $u,v \in F^1_{USC} (X)$, then
$d_p^*(u,v) = d_p(u,v)$.
The $d_p^*$ metric is
 an expansion of the $d_p$ distance on
$F^1_{USC} (X)$.

The $d_p$ distance is well-defined on $F^1_{USC}(\mathbb{R}^m)$ (see \cite{huang17}, Proposition 2.2 in \cite{huang17c}).
The $d_p$ distance is well-defined on $F^1_{USCG}(X)$ (see \cite{huang17}, Proposition 2.7 in \cite{huang17c}).
We can see
that
the $d_p$ distance is a metric on $F^1_{USCB}(X)$.

Let $u\in F^1_{USCG}(\mathbb{R})$ defined by
\[
u(x)=
\left\{
  \begin{array}{ll}
0, & x<1, \\
1, & x=1, \\
   1/n, & x\in (n^2, (n+1)^2],\ n=1,2,\ldots.
  \end{array}
\right.
\]
Then $d_p(u, \widehat{1}_{F(\mathbb{R})}) = + \infty$.
The $d_p$ distance on $F^1_{USCG}(\mathbb{R}^m)$ is an extended metric but not a metric,
and the $d_p$ distance on $F^1_{USC}(\mathbb{R}^m)$ is an extended metric but not a metric.
In this paper, we call
the $d_p$ distance on $F^1_{USC}(\mathbb{R}^m)$ or $F^1_{USCG}(X)$ the $d_p$ metric for simplicity.

\begin{tm} \label{dpeu} Let $(X,d)$ be a metric space and let $u, v \in F^1_{USC} (X)$.
Then
\begin{equation}\label{dpe}
  d_p^* (u,v) \geq   \left(  \frac{H_{\rm end} (u,v)^{p+1} }{p+1}   \right)^{1/p}.
\end{equation}

\end{tm}

\begin{proof}
To show the desired result, we only need to show that for each $r>0$,
if $H_{\rm end} (u,v) > r$ then $d_p^* (u,v) \geq \left( \frac{r^{p+1}}{p+1}  \right)^{1/p}$.

Let $r>0$.
Assume that $H_{\rm end} (u,v) > r$. Then without loss of generality we suppose that $H^*({\rm end}\,u,    {\rm end}\,v) > r$, then there is an $(x,\beta) \in {\rm end}\, u$, such that
$d((x,\beta), {\rm end}\, v ) >r$.
This implies that $\beta > r$ and
$d(x, [v]_{\alpha}) > r-(\beta-\al)$ when $\alpha\in [\beta-r, \beta]$.
Hence
$H^*([u]_\alpha, [v]_{\alpha}) > r-(\beta-\al) $ when $\alpha\in [\beta-r, \beta]$.

Let $f$
be a measurable function on $[0,1]$
with
 $f(\al) \geq H([u]_\al, [v]_\al  )$ for $\al\in [0,1]$.
Then
\begin{align*}
& \left( \int_0^1  f(\al)^p \,d\al \right)^{1/p}
 \geq \left( \int_{\beta-r}^{\beta}  f(\alpha)^p     \,d\alpha \right)^{1/p} \\
&>     \left(    \int_{\beta-r}^{\beta}  (r-(\beta-\al))^p   \,d \alpha \right)^{1/p} \\
&= \left( \frac{r^{p+1}}{p+1}  \right)^{1/p}.
\end{align*}
So
$d_p^* (u,v) \geq \left( \frac{r^{p+1}}{p+1}  \right)^{1/p}$.

\end{proof}

  The ``='' can be obtained in \eqref{dpe}.
\begin{eap}
{\rm
Define $u$ and $v$ in $F^1_{USCB}(\mathbb{R})$ as
\[
u(x)
=
\left\{
  \begin{array}{ll}
1, & x=0,\\
0.5-x, & x\in (0, 0.5],\\
0, & \mbox{otherwise},
  \end{array}
\right.
\
v(x)
=
\left\{
  \begin{array}{ll}
1, & x=0,\\
0.5, & x\in (0, 0.5],\\
0, & \mbox{otherwise}.
  \end{array}
\right.
\]
Then $H_{\rm end} (u,v) = 0.5$
and
\[
H([u]_\al, [v]_\al)
=
\left\{
  \begin{array}{ll}
   0, & \al\in (0.5,1],\\
\al, & \al \in [0, 0.5].
  \end{array}
\right.
\]
Thus
$d_p(u,v) = (\int_0^{0.5}  \al^p  \,d\al)^{1/p} = \left(\frac{0.5^{p+1}}{p+1} \right)^{1/p} =  \left(  \frac{H_{\rm end} (u,v)^{p+1} }{p+1}   \right)^{1/p}$.

}
\end{eap}

We can see that the following Theorem \ref{gdpn} can also be deduced by the
Theorem \ref{dpeu}.

\begin{tm} \label{gdpn}
  Let $u\in F^1_{USC} (X)$ and for each positive integer $n$, let $u_n\in F^1_{USC} (X)$.
  If $d_p^*(u_n, u) \to 0$, then
  $H_{\rm end} (u_n, u) \to 0$.
\end{tm}

\begin{proof} We prove by contradiction.
    If
  $H_{\rm end} (u_n, u) \not\to 0$, then there is an $\varepsilon>0$ and a subsequence $\{   v_n   \}$
of $\{u_n\}$
such that
\begin{equation}\label{gec}
  H_{\rm end} (v_n,  u ) \geq \varepsilon.
\end{equation}

From the definition of $d_p^*$, there exist a sequence $\{f_n\}$ of
measurable function
from $[0,1]$ to $\mathbb{R} \cup \{+\infty\}$
such that for each $n\in \mathbb{N}$
\begin{gather}
 H([v_n]_\al,  [u]_\al)     \leq    f_n (\al)     \mbox{ for all }    \al\in [0,1],\label{smn}
 \\
 \left(\int_0^1   f_n (\al)^p   d\, \al     \right)^{1/p} \leq \frac{n+1}{n} d_p^*(v_n,  u). \label{tce}
\end{gather}
Since $d_p^*(v_n,  u) \to 0$, by \eqref{tce}, we have
$\left(\int_0^1   f_n (\al)^p   d\, \al     \right)^{1/p} \to 0$.
Thus there is a subsequence $\{f_{n_k}\}$ of $\{f_n\}$
such that
$\{f_{n_k} \}$ a.e. converges to $0$ on $[0,1]$.
Hence by \eqref{smn},
$H([v_{n_k}]_\al, [u]_\al) \rightarrow 0$ holds a.e. on $\al\in (0,1)$.
From
Theorem \ref{aedr}, this
  implies
   $H_{\rm end} (v_{n_k}, u) \to 0$, which
contradicts \eqref{gec}.

\end{proof}

 Theorem 4.1 in \cite{huang17c} says that
for $u\in F^1_{USCG} (X)$ and $v\in F^1_{USC} (X)$,
$H([u]_\al, [v]_\al)$ is a measurable function of $\al$ on $[0, 1]$.
So $d_p^*(u,v)
=
d_p(u,v)$ for $u\in F^1_{USCG} (X)$ and $v\in F^1_{USC} (X)$.

\begin{tm} \label{ecm} Suppose that $u\in  F^1_{USCG}(X)$ and $u_n \in F^1_{USC} (X)$, $n=1,2,\ldots$,
and
that there is a measurable function $F$ on $[0,1]$
such that $\int_0^1 F^p(\al) \,d\al < +\infty$
and
$H([u_n]_\al,  [u]_\al) \leq  F(\al)  $ for $n=1,2,\ldots$.
If
 $H_{\rm end} (u_n, u) \to 0$, then $d_p(u_n, u) \to 0$.
  \end{tm}

\begin{proof} \ By Theorem \ref{aec}, $H_{\rm end} (u_n, u) \to 0$
if and only if
 $H([u_n]_\al, [u]_\al) \rightarrow 0$ holds a.e. on $\al\in (0,1)$.
So the desired result follows from the Lebesgue's Dominated Convergence
Theorem.

\end{proof}

\begin{tl}\label{lcu}
Let $u\in  F^1_{USCG}(X)$ and for $n=1,2,\ldots$, let $u_n \in F^1_{USC} (X)$.
If $H_{\rm end}(u_n, u) \to 0$ and $\bigcup_{n=1}^{+\infty} [u_n]_0$ is bounded,
then
$d_p(u_n, u) \to 0$.
\end{tl}

\begin{proof}
Since $H_{\rm end}(u_n, u) \to 0$ and $\bigcup_{n=1}^{+\infty} [u_n]_0$ is bounded,
then from Remark \ref{sur}, $[u]_0
   \subseteq
\liminf_{n\to \infty}[u_n]_{0} \subseteq \overline{\bigcup_{n=1}^{+\infty} [u_n]_0}$ is bounded.
Hence there is an $M\in \mathbb{R}$ such that $d_\infty(u_n,u) \leq M$; that is, for all $\al\in [0,1]$,
 $H([u_n]_\al, [u]_\al)\leq M$. Thus
by Theorem \ref{ecm},
$d_p(u_n, u) \to 0$.

\end{proof}

\begin{pp} \label{spu} Let
  $u\in F^1_{USCB} (X)$ and for each positive integer $n$, let $u_n\in F^1_{USC} (X)$.
Then
$H_{\rm send} (u_n, u) \to 0$
if and only if
$H([u_n]_0, [u]_0) \to 0$ and $d_p(u_n, u) \to 0$.
\end{pp}

\begin{proof}From
Proposition \ref{sge}, $H_{\rm send} (u_n, u) \to 0$
if and only if
$H([u_n]_0, [u]_0) \to 0$ and $H_{\rm end} (u_n, u) \to 0$.
To prove the desired result, we only need to show
that
$$H([u_n]_0, [u]_0) \to 0 \mbox{ and } d_p(u_n, u) \to 0
\Leftrightarrow
H([u_n]_0, [u]_0) \to 0 \mbox{ and } H_{\rm end} (u_n, u) \to 0.$$

From Theorem \ref{gdpn}, ``$\Rightarrow$'' is true.
To show ``$\Leftarrow$'',
suppose that
 $H([u_n]_0, [u]_0) \to 0$ and $H_{\rm end} (u_n, u) \to 0$.
Then
there exist an $N \in \mathbb{N}$ such that $\bigcup_{n \geq N} [u_n]_0$ is bounded.
Thus
by Corollary \ref{lcu}, $d_p(u_n, u) \to 0$. So``$\Leftarrow$'' is true.

\end{proof}

 The above results and the examples in this paper show
that
the following statements are
true for $u$, $u_n$ in $F^1_{USC}(X)$, $n=1,2,\ldots$.

\begin{description}

  \item[(\romannumeral1)] $d_p^*(u_n, u) \to 0$ imply
  $H_{\rm end} (u_n, u) \to 0$ (a stronger conclusion is given in Theorem \ref{dpeu}).

  \item[(\romannumeral2)] Suppose that $u \in F^1_{USCB}(X)$.
By Proposition \ref{spu},
$H_{\rm send}(u_n, u) \to 0$ if and only if $ H([u_n]_0, [u]_0]) \to 0 $ and $d_p(u_n, u) \to 0$.

Suppose that $u$
and $\{u_n\}$ is in $F^1_{USCB}(X)$. Example \ref{pnu} below shows that $d_p(u_n, u) \to 0$ does not necessarily imply that
 $\bigcup_{n=1}^{+\infty} [u_n]_0$ is bounded.
Hence $d_p(u_n, u) \to 0$ does not necessarily imply
$H_{\rm send}(u_n, u) \to 0$ since $H_{\rm send}(u_n, u) \to 0$ implies that $\bigcup_{n=1}^{+\infty} [u_n]_0$ is bounded for $u$ and $\{u_n\}$ in $F^1_{USCB}(X)$
(In fact, by Theorem \ref{rcfe}, $H_{\rm send}(u_n, u) \to 0$ implies that $\bigcup_{n=1}^{+\infty} [u_n]_0$ is relatively compact).
Take $u$, $\{u_n\}$ in Remark \ref{sur}, we can see that $d_p(u_n, u) \to 0$ and $\bigcup_{n=1}^{+\infty} [u_n]_0$ is compact
need not imply $H_{\rm send}(u_n, u) \to 0$.

  \item[(\romannumeral3)] Suppose that $u \in F^1_{USCG}(X)$.
By Theorem \ref{ecm}, under certain conditions,
$H_{\rm end}(u_n, u) \to 0$ imply
$d_p(u_n, u) \to 0$.
However, even the sequence $\{u_n\}$ is in $F^1_{USCG}(X)$,
$H_{\rm send}(u_n, u) \to 0 $ does not necessarily imply $d_p(u_n, u) \to 0$.
See the following Example \ref{snp}.
By Corollary \ref{lcu}, $H_{\rm end}(u_n, u) \to 0$ and $\bigcup_{n=1}^{+\infty} [u_n]_0$ is bounded
imply
$d_p(u_n, u) \to 0$.

  \item[(\romannumeral4)] Suppose that $u\in F^1_{USC}(X)$.
Even
$H_{\rm send}(u_n, u) \to 0 $ and $\bigcup_{n=1}^{+\infty} [u_n]_0$ is bounded do not necessarily imply $d^*_p(u_n, u) \to 0$.
Below we illustrate that a such example is given in Example \ref{rcp}.

\end{description}

Example \ref{rcp} shows
that
there exist $u$ and $u_n$, $n=1,2,\ldots$ in
$ F^1_{USC} (\prod_{x\in (0,1]} [0,3])$
such that
\\
(\romannumeral1) $  H_{\rm send} (u_n, u) \to 0$,
\\
(\romannumeral2)  $[u]_0 = [u_n]_0 = \{ \xi\in \prod_{x\in (0,1]} [0,1] : \xi_x \to 0 \ \mbox{as} \ x\to 0  \}$
for all $n=1,2,\ldots$,
\\
(\romannumeral3) $H([u_n]_\al, [u]_\al) =1$ for all $\al\in (0,1]$ and $n=1,2,\ldots$,
\\
(\romannumeral4)
$\sup \{d(x,y): x,y \in  \bigcup_{n=1}^{+\infty} [u_n]_0 \}=1$, and so $\bigcup_{n=1}^{+\infty} [u_n]_0 $ is bounded,
\\
(\romannumeral5) $d_p(u_n, u)=1$ for all $n=1,2,\ldots$, and so $d_p(u_n, u)\not\to 0$.

The above (\romannumeral1)-(\romannumeral3) are shown in Example \ref{rcp}.
 (\romannumeral2) implies (\romannumeral4).
(\romannumeral3) implies (\romannumeral5).

\begin{eap} \label{pnu}
  {\rm
Let $u=\widehat{0}_{F(\mathbb{R})} \in F^1_{USCB} (\mathbb{R})$. For $n=1,2,\ldots$,
define $u_n \in  F^1_{USCB} (\mathbb{R})$ as
\[
u_n=
\left\{
  \begin{array}{ll}
   1, & x=0, \\
1/n^2, & 0< x \leq n^{1/p}, \\
   0, & \hbox{otherwise.}
  \end{array}
\right.
\]
Then
$d_p(u_n, u) = \left(       \int_0^{1/n^2} ( n^{1/p})^p \, d\al      \right)^{1/p} = (\frac{1}{n})^{1/p} \to 0$.
However,
 $\bigcup_{n=1}^{+\infty} [u_n]_0= [0,+\infty)$ is not bounded.

}
\end{eap}

\begin{eap} \label{snp}
  {\rm
 Define $u\in F^1_{USCG}(\mathbb{R})$ as follows:
$$[u]_\al= [0,1/\al].$$
For
each $n=1,2,\ldots$, define $u_n \in  F^1_{USCG}(\mathbb{R})$ as follows:
\[[u_n]_\al
=\left\{
   \begin{array}{ll}
     [0,1/\al], & \al\in (1/n,1], \\
    \mbox{} [0,1/\al+n^2], & \al\in [0, 1/n].
   \end{array}
 \right.
\]
Then $H_{\rm send} (u_n,u) \to 0$ and $d_p(u_n,u)=n^{2-1/p} \not\to 0 $.

}
\end{eap}

In \cite{huang17}, we obtain that the Skorokhod metric convergence imply the sendograph metric convergence
on $F^1_{USC} (X)$
 (see Theorem 8.1 in \cite{huang17}),
and that
Skorokhod metric convergence need not imply the $d_p$ convergence on a subset of $F^1_{USCG}(X)$ (see the end of Section 5 in \cite{huang17}). From the above conclusions,
we can also deduce that
the sendograph metric convergence need not imply the $d_p$ convergence on $F^1_{USCG}(X)$.

\section{Characterizations of compactness in $(F^1_{USCG} (X), H_{\rm end})$ and $(F^1_{USCB} (X), H_{\rm send})$}
\label{cmfuzzy}

Based on the conclusions in previous sections, we give characterizations of
total boundedness, relative compactness and compactness
in
 $(F^1_{USCG} (X), H_{\rm end})$ and $(F^1_{USCB} (X), H_{\rm send})$.

\begin{itemize}
 \item A subset $Y$ of a topological space $Z$ is said to be \emph{compact} if for every set $I$
and every family of open sets, $O_i$, $i\in I$, such that $Y\subset \bigcup_{i\in I} O_i$ there exists
a finite family $O_{i_1}$, $O_{i_2}$ \ldots, $O_{i_n}$ such that
$Y\subseteq O_{i_1}\cup O_{i_2}\cup\ldots \cup O_{i_n}$.
In
the case of a metric topology, the criterion for compactness becomes that any sequence in $Y$ has a subsequence convergent in $Y$.

 \item
A \emph{relatively compact} subset $Y$ of a topological space $Z$ is a subset with compact closure. In the case of a metric topology, the criterion for relative compactness becomes that any sequence in $Y$ has a subsequence convergent in $X$.

 \item Let $(X, d)$ be a metric space. A set $U$ in $X$ is \emph{totally bounded} if and only if, for each $\varepsilon>0$, it contains a finite $\varepsilon$ approximation, where an $\varepsilon$ approximation to $U$ is a subset $S$ of $U$ such that $d(x,S)<\varepsilon$ for each $x\in U$.
An $\varepsilon$ approximation to $U$ is also called an $\varepsilon$-net of $U$.

\end{itemize}
Let $(X, d)$ be a metric space. A set $U$ is compact in $(X,d)$ implies that $U$ is relatively compact
in $(X,d)$, which in turn
implies that $U$ is totally bounded in $(X,d)$.

We use $(\widetilde{X}, \widetilde{d})$ to denote the completion of $(X, d)$.
We see $(X, d)$ as a subspace of $(\widetilde{X}, \widetilde{d})$.
Let
$S \subseteq \widetilde{X}$.
The symbol $\widetilde{\overline{S}}$ is used to denote
the closure of $S$
in
 $(\widetilde{X}, \widetilde{d})$.

As defined in Section \ref{bas}, we have
$K(\widetilde{X})$,
 $C(\widetilde{X})$, $F^1_{USC}(\widetilde{X})$,
$F^1_{USCG} (\widetilde{X})$, etc. according to $(\widetilde{X}, \widetilde{d})$.
For example,
  \begin{gather*}
   F^1_{USC}(\widetilde{X}) :=\{ u\in F(\widetilde{X}) : [u]_\al \in  C(\widetilde{X})  \  \mbox{for all} \   \al \in [0,1]   \}, \\
   F^1_{USCG} (\widetilde{X}) := \{ u\in  F(\widetilde{X}): [u]_\al \in K(\widetilde{X}) \ \mbox{for all} \   \al\in (0,1] \}.
 \end{gather*}

If there is no confusion,
 we also
use $H$ to denote the Hausdorff metric
on
 $C(\widetilde{X})$ induced by $\widetilde{d}$.
  We also
use $H$ to denote the Hausdorff metric
 on $C(\widetilde{X}\times [0,1])$ induced by $\overline{\widetilde{d}}$.
 We
 also use $H_{\rm end}$ to denote the endograph metric on $F^1_{USC}(\widetilde{X})$
given by using $H$ on
$C(\widetilde{X} \times [0,1])$.

Clearly,
the induced metric on $F^1_{USCG} (X)$ by the $H_{\rm end}$ on $F^1_{USC} (X)$
is the same as
the induced metric on $F^1_{USCG} (X)$ by
the
$H_{\rm end}$ on $F^1_{USC} (\widetilde{X})$.

We see
$ (F^1_{USCG} (X), H_{\rm end})$
as a metric subspace of
$ (F^1_{USCG} (\widetilde{X}), H_{\rm end})$.

\subsection{Characterizations of compactness in $(K(X), H)$}

In this subsection, we give characterizations of total boundedness, relative compactness and compactness in $(K(X), H)$.
The
results in this subsection are basis for contents in the sequel.

\begin{tm} \label{cpm} Let $(X, d)$ be complete and let $\{C_n\}$ be a Cauchy sequence in $(K(X), H)$.
Put $D_n= \bigcup_{l=1}^{n} C_l$
and
 $D= \overline{\bigcup_{l=1}^{+\infty} C_l}  $.
 Then
 $D \in K(X)$
and
  $H(D_n, D) \to 0$.
\end{tm}

\begin{proof}
  Note that for $k>j$,
  $$H(D_k, D_j) \leq \max\{H(C_i, C_j) : i=j+1,\ldots,k \}.$$
So
  $\{D_n\}$ is a Cauchy sequence in $(K(X), H)$.
   From Theorem \ref{bfc}, $(K(X), H)$ is complete,
  and thus by Theorem \ref{hkg}, $\{D_n\}$ converges to $\limsup_{n\to \infty} D_n = D \in K(X)$.

  \end{proof}

\begin{tm} \label{tbe}
  Let $(X,d)$ be a metric space and $\mathcal{D}\subseteq K(X)$.
   Then $\mathcal{D}$ is totally bounded in $(K(X), H)$
if and only if
$\mathbf{ D} =   \bigcup \{C:  C \in  \mathcal{D} \} $ is totally bounded in $(X,d)$.
\end{tm}

\begin{proof} If $\mathcal{D} = \emptyset$, then the desired result follows immediately.
Suppose that
$\mathcal{D} \not= \emptyset$.

\textbf{\emph{Necessity}}. \
 To show that $\mathbf{ D}$ is totally bounded. We only need to show that
each sequence in $\mathbf{ D}$ has a Cauchy subsequence.

 Given a sequence $\{x_n\}$
in   $\mathbf{D}$.
 Suppose that
$x_n \in C_n \in \mathcal{D}$.
Since $\mathcal{D}$ is totally bounded, then $\{C_n\}$ has a Cauchy subsequence $\{C_{n_k}\}$.
Hence,
 by Theorem \ref{cpm},
$\widetilde{\overline{\bigcup_{k=1}^{+\infty} C_{n_k} }}$ is in $K(\widetilde{X})$.
Thus
$\{x_n\} $
 has a Cauchy subsequence.

\textbf{\emph{Sufficiency}}. \   If   $\mathbf{ D} $ is totally bounded in $X$, then $\widetilde{\overline{\mathbf{D}}}$ is in $K(\widetilde{X})$.
So, by Theorem \ref{bfc}, $(K(\widetilde{\overline{\mathbf{D}}}), H)$ is compact, and
thus
 $\mathcal{D} $ is totally bounded.

\end{proof}

\begin{tm}\label{rce}
  Let $(X,d)$ be a metric space and $\mathcal{D}\subseteq K(X)$. Then $\mathcal{D} $ is relatively compact in $(K(X), H)$
if and only if
$\mathbf{ D} =   \bigcup \{C:  C \in  \mathcal{D} \} $ is relatively compact in $(X,d)$.
\end{tm}

\begin{proof} \ If $\mathcal{D} = \emptyset$, then the desired result follows immediately.
Suppose that
$\mathcal{D} \not= \emptyset$.

\textbf{\emph{Necessity}}. \   To show that $\mathbf{ D}$ is relatively compact. We only need to show that
each sequence in $\mathbf{ D}$ has a convergent subsequence in $X$.

Given a sequence $\{x_n\}$
in   $\mathbf{D}$.
 Suppose that
$x_n \in C_n \in \mathcal{D}$.
Since $\mathcal{D}$ is relatively compact, then $\{C_n\}$ has a subsequence $\{C_{n_k}\}$ which converges to $C$ in $K(X)$.
Hence,
 by Theorem \ref{cpm},
$\widetilde{\overline{\bigcup_{k=1}^{+\infty} C_{n_k} }}$ is in $K(\widetilde{X})$
(In fact, $\widetilde{\overline{\bigcup_{k=1}^{+\infty} C_{n_k} }}$ is in $K(X)$).
So $\{x_{n_k}\}$ has a subsequence which converges to $x$
in $\widetilde{\overline{\bigcup_{k=1}^{+\infty} C_{n_k} }}$,
and thus
$x\in C \subset X$.

\textbf{\emph{Sufficiency}}. \ If $\mathbf{ D} $ is relatively compact in $X$, then  $\overline{\mathbf{ D}} $ is in $K(X)$,
and
therefore $(K(\overline{\mathbf{ D}}), H)$ is compact.
Thus $\mathcal{D} \subset  K(\overline{\mathbf{ D}}) $ is relatively compact in $(K(X), H)$.

\end{proof}

\begin{lm} \label{coms}
   Let $(X,d)$ be a metric space and $\mathcal{D}\subseteq K(X)$. If $\mathcal{D} $ is compact in $(K(X), H)$,
then
$\mathbf{ D} =   \bigcup \{C:  C \in  \mathcal{D} \} $ is compact in $(X, d)$.
\end{lm}

\begin{proof} \ If $\mathcal{D} = \emptyset$, then the desired result follows immediately.
Suppose that
$\mathcal{D} \not= \emptyset$.
  To show that $\mathbf{ D}$ is compact. We only need to show that
each sequence in $\mathbf{ D}$ has a subsequence which converges to a point in $\mathbf{ D}$.

Given a sequence $\{x_n\}$
in   $\mathbf{D}$.
 Suppose that
$x_n \in C_n \in \mathcal{D}$.
Since $\mathcal{D}$ is compact, then $\{C_n\}$ has a subsequence $\{C_{n_k}\}$ converges to $C \in \mathcal{D}$.
Hence, by Theorem \ref{cpm},
$\widetilde{\overline{\bigcup_{k=1}^{+\infty} C_{n_k} }}$ is in $K(\widetilde{X})$
(In fact, $\widetilde{\overline{\bigcup_{k=1}^{+\infty} C_{n_k} }}$ is in $K(\mathbf{D})$).
So $\{ x_{n_k} \}$ has a subsequence which converges to $x$
in $\widetilde{\overline{\bigcup_{k=1}^{+\infty} C_{n_k} }}$. Thus $x \in   C   \subset \mathbf{D}$.

\end{proof}

\begin{re}
{\rm
The converse of the implication in Lemma \ref{coms} does not hold.
Let $(X,d) = \mathbb{R}$ and
$\mathcal{D} = \{ [0,x]:  x\in (0.3,1] \}   \subset   K(\mathbb{R})$.
Then $\mathbf{D} = [0,1] \in K(\mathbb{R})$.
But
$\mathcal{D}$ is not compact in $(  K(\mathbb{R}),    H   )$.
}
\end{re}

\begin{tm} \label{come}
 Let $(X,d)$ be a metric space and $\mathcal{D}\subseteq K(X)$. Then
    the following are equivalent:
    \\
    (\romannumeral1) \ $\mathcal{D} $ is compact in $(K(X), H)$;
        \\
 (\romannumeral2) \
$\mathbf{ D} =   \bigcup \{C:  C \in  \mathcal{D} \} $ is relatively compact in $(X, d)$
and
$\mathcal{D} $ is closed in $(K(X), H)$;
    \\
 (\romannumeral3) \
$\mathbf{ D} =   \bigcup \{C:  C \in  \mathcal{D} \} $ is compact in $(X, d)$
and
$\mathcal{D} $ is closed in $(K(X), H)$.
\end{tm}

\begin{proof}
Note that $\mathcal{D} $ is compact in $(K(X), H)$ if and only if $\mathcal{D} $ is relatively compact and closed in $(K(X), H)$.
 Then from Theorem \ref{rce} we have (\romannumeral1)$\Leftrightarrow$(\romannumeral2). Clearly (\romannumeral3)$\Rightarrow$(\romannumeral2). We
shall complete the proof by showing that (\romannumeral1)$\Rightarrow$(\romannumeral3),
which can be deduced by Lemma \ref{coms}.

\end{proof}

\begin{re}
{\rm
   After we gave
  conclusions and their proofs in this section, we found that Theorem \ref{rce} is Proposition 5 in \cite{greco}.
   Since we can't find the proof of Proposition 5,
   we give our proof here.

}
\end{re}

\subsection{Characterizations of compactness in $(F^1_{USCG} (X), H_{\rm end})$}

In this subsection, we give characterizations of total boundedness, relative compactness and compactness in $(F^1_{USCG} (X), H_{\rm end})$.

Let $U$ be a subset of $(X,d)$.
We say a subset $S$ of $X$ is a \emph{weak}
$\varepsilon$\emph{-net} of $U$
if $d(x,S)<\varepsilon$ for each $x\in U$.

We can see that an $\varepsilon$-net of $U$ must be a weak $\varepsilon$-net of $U$. An $\varepsilon$-net of $U$ is included in $U$, however a weak $\varepsilon$-net of $U$ is not necessarily be included in $U$.
For convenience, we use the term weak $\varepsilon$-net of a set $U$ in $(X,d)$.

For $x\in X$ and $\varepsilon>0$, let $ B(x,\varepsilon) := \{z\in X: d(x,z)<\varepsilon\}$.

The following known conclusions on totally bounded are useful in this paper.
\begin{itemize}
  \item  A set $U$ in $X$ is totally bounded if and only if for each sequence $\{x_n\}$ in $U$, it has a subsequence $\{x_{n_k}\}$ which is a Cauchy sequence.

  \item A set $U$ in $X$ is totally bounded if and only if for each $\varepsilon>0$,
there is a finite weak $\varepsilon$-net of $U$.

\end{itemize}

Below we give a proof of the last conclusion mentioned above, although we are convinced that the proof for this conclusion was already been given.
Readers who think it unnecessary to give the proof or this conclusion is obvious can skip this proof.

Clearly the necessity is true since an $\varepsilon$-net of $U$ is a weak $\varepsilon$-net of $U$.
To show the
sufficiency, it suffices to show how to construct a finite $\varepsilon$-net of $U$ via
a finite weak $\varepsilon/2$-net of $U$.

Let $S_{\varepsilon/2}$ be a finite weak $\varepsilon/2$-net of $U$.
Set $S'_{\varepsilon/2} := \{ y\in S_{\varepsilon/2}:  B(y, \varepsilon/2) \cap U \not= \emptyset  \}$.
Clearly
$S'_{\varepsilon/2}$ is a finite weak $\varepsilon/2$-net of $U$.
Then for each $y\in S'_{\varepsilon/2}$,
we choose an $x_{y}$ in $B(y, \varepsilon/2) \cap U$.
Let
 $T_{\varepsilon} := \{x_{y} : y\in S'_{\varepsilon/2} \}$. We claim that $T_\varepsilon$ is a finite $\varepsilon$-net
of $U$. Clearly $T_\varepsilon$ is a finite subset of $U$.
To complete the proof, we only need to show that
 $d(z, T_\varepsilon) < \varepsilon$ for each $z\in U$.
Let $z\in U$. Then there is a $y\in S'_{\varepsilon/2}$ with $d(z,y) < \varepsilon/2$,
and
thus
$d(z, T_\varepsilon) \leq d(z, x_y) \leq d(z, y) + d(y, x_y) < \varepsilon$.

Suppose that
$U$ is a subset of $F^1_{USC} (X)$ and $\al\in [0,1]$.
For
writing convenience,
we denote
\begin{itemize}
  \item  $U(\al):= \bigcup_{u\in U} [u]_\al$, and

\item  $U_\al : =  \{[u]_\al: u \in U\}$.
\end{itemize}

\begin{tm} \label{tbfegn}
  Let $U$ be a subset of $F^1_{USCG} (X)$. Then $U$ is totally bounded in $(F^1_{USCG} (X), H_{\rm end})$
if and only if
$U(\al)$
is totally bounded in $(X,d)$ for each $\al \in (0,1]$.
\end{tm}

\begin{proof}
   \textbf{\emph{Necessity}}.  Suppose that $U$ is totally bounded in $(F^1_{USCG} (X), H_{\rm end})$.
Let $\al \in (0,1]$.
To show
that
$U(\al)$
is totally bounded in $X$, we only need to show that
each sequence in $U(\al)$
has a Cauchy subsequence.

Given a sequence $\{x_n\} \subset U(\al)$. Suppose that $x_n \in  [u_n]_\al$,  $u_n  \in  U$, $n=1,2,\ldots$.
Then
$\{u_n\}$ has
 a Cauchy subsequence $\{u_{n_l}\}$. So given $\varepsilon\in (0, \al)$, there is a $K( \varepsilon) \in \mathbb{N}$
such that
\begin{equation*}\label{ends}
H_{\rm end} (u_{n_l},  u_{n_K})  <   \varepsilon
\end{equation*}
for all $l \geq K$.
Thus
\begin{equation}\label{cuts}
H^*([u_{n_l}]_\al,       [u_{n_K}]_{\al-\varepsilon}) <  \varepsilon
\end{equation}
for all $l \geq K$.
From \eqref{cuts} and the arbitrariness of $\varepsilon$,  $\bigcup_{l=1} ^{+\infty} [u_{n_l}]_\al $ is totally bounded in $(X,d)$.
Thus
$\{x_{n_l}\}$, which is a subsequence of $\{x_n\}$,
has a Cauchy subsequence, and so does $\{x_n\}$.

In the following,
we give
a more detailed proof for the above conclusion that
        $\bigcup_{l=1} ^{+\infty} [u_{n_l}]_\al $ is totally bounded in $(X,d)$, although
we think the proof given above for this conclusion is sufficient.

To show the desired result, it suffices to show that
for each $\lambda>0$, there exists a finite weak $\lambda$-net of  $ \bigcup_{l=1} ^{+\infty} [u_{n_l}]_\al$.

Let $\lambda>0$. Set $\varepsilon= \min\{\lambda/2, \al/2\} $.
Then $\varepsilon \in (0, \al)$. Hence there is a $K(\varepsilon)$ such that \eqref{cuts} holds for all $l\geq K$.
Since $\bigcup_{l=1} ^{K}[u_{n_l}]_{\al-\varepsilon} $ is compact,
there is a finite $\varepsilon$-net $\{z_j\}_{j=1}^{m}$ of $\bigcup_{l=1} ^{K}[u_{n_l}]_{\al-\varepsilon} $. We claim that $\{z_j\}_{j=1}^{m}$
is a finite weak $\lambda$-net of $\bigcup_{l=1} ^{+\infty} [u_{n_l}]_\al$.

Let $z \in \bigcup_{l=1} ^{+\infty} [u_{n_l}]_\al$.
If $z \in \bigcup_{l=1} ^{K} [u_{n_l}]_\al$, then clearly $d(z, \{z_j\}_{j=1}^{m}) < \varepsilon<\lambda$.
If $z \in \bigcup_{l=K+1} ^{+\infty} [u_{n_l}]_\al$,
then
by \eqref{cuts},
  there exists a $y_z \in [u_{n_K}]_{\al-\varepsilon}$
such that
$d(z, y_z) <\varepsilon $,
and so
$d(z, \{z_j\}_{j=1}^{m}) \leq  d(z, y_z) + d(y_z, \{z_j\}_{j=1}^{m} ) < 2\varepsilon \leq \lambda$.

\textbf{\emph{Sufficiency}}.
Suppose that $U(\al)$ is totally bounded in $X$ for each $\al\in (0,1]$.
By Theorem \ref{tbe},
$U(\al)$ is totally bounded in $X$
if and only if
 $U_{\al}$ is totally bounded in $(K(X), H)$.
Thus, by Theorem \ref{bfc}, we have the following affirmation:
\begin{itemize}
  \item
   Given $\al \in (0,1]$.
For each sequence $\{[u_n]_\al, n=1,2,\ldots\}$ in $U_{\al}$,
it has a subsequence
 $\{[u_{n_k}]_\al, k=1,2,\ldots\}$ which converges to $u_\al  \in K(\widetilde{X}) $ with respect to the Hausdorff metric $H$.
 \end{itemize}

To prove that $U$ is totally bounded, it suffices to show that
each sequence in $U$ has a convergent subsequence in $(F^1_{USCG} (\widetilde{X}), H_{\rm end})$.
Suppose that
$\{u_n\}$ is a sequence in $U$. Based on the above affirmation and Theorem \ref{aec},
  and
 proceeding similarly to the proof of the ``Sufficiency part'' of Theorem 7.1 in \cite{huang}, it can be shown
that
$\{u_n\}$
has a subsequence $\{v_n\}$ which converges to $v \in F^1_{USCG} (\widetilde{X})$ with respect to $H_{\rm end}$.

A sketch of the
proof of the existence of $\{v_n\}$ and $v$ is given as follows.

First, we construct a subsequence $\{v_n\}$ of $\{u_n\}$
such that
$[v_n]_q$ converges to $u_q \in K(\widetilde{X})$ according to the Hausdorff metric $H$ for all $q\in \mathbb{Q}'$, where $\mathbb{Q}' = \mathbb{Q} \cap (0,1]$.
Then
 we show that
 $v\in F^1_{USCG}(\widetilde{X})$ with
$[v]_\al = \bigcap_{q<\al, q\in \mathbb{Q}'} u_q$ for all $\al \in (0,1]$
satisfies that
$H_{\rm end} (v_n, v) \to 0$.

\end{proof}

\begin{re} {\rm Some of the implications in the proofs of this paper
are actually the equivalent.
For example, in the proof of Theorem \ref{tbfegn},
   $U(\al)$ is totally bounded in $X$ for each $\al\in (0,1]$
is equivalent to the affirmation after the ``$\bullet$''
}
\end{re}

\begin{tm} \label{rcfegn}
  Let $U$ be a subset of $F^1_{USCG} (X)$. Then $U$ is relatively compact in $(F^1_{USCG} (X), H_{\rm end})$
if and only if
$U(\al)$
is relatively compact in $(X, d)$ for each $\al \in (0,1]$.
\end{tm}

\begin{proof}
    \textbf{\emph{Necessity}}. Suppose that $U$ is relatively compact. Given $\al \in (0,1]$.
To show that
$U(\al)$
is relatively compact in $X$, we only need to show that each sequence
in $U(\al)$ has a convergent subsequence in $X$.

 Let $\{x_n\}$ be a sequence in $U(\al)$.
Suppose
that $x_n\in [u_n]_\al$, $u_n \in U$, $n=1,2,\ldots$.
Then
 there is a subsequence $\{ u_{n_k} \}$ of $\{u_n\}$ and $u \in F^1_{USCG} (X)$
  such that
$H_{\rm end} (u_{n_k}, u) \to 0$.
So, by Theorem \ref{aec}, $H([u_{n_k}]_\al, [u]_\al) \str{\rm a.e.}{\longrightarrow} 0$,
and therefore there is a $\beta \in (0, \al)$ such that
$H([  u_{n_k}  ]_\beta, [u]_\beta) \to 0$. Hence by Theorem \ref{rce}, $\bigcup_{k=1}^{+\infty}[  u_{n_k}  ]_\beta$ is relatively compact in $X$.
Thus
$\{x_{n_k}\}$ has a convergent subsequence in $X$, and so does $\{x_n\}$.

\emph{\textbf{Sufficiency}}.
Suppose that $U(\al)$ is relatively compact in $X$ for each $\al\in (0,1]$.
To show that $U$ is relatively compact
in $(F^1_{USCG} (X), H_{\rm end})$,
we only need to show that each sequence
in $U$ has a convergent subsequence in $(F^1_{USCG} (X), H_{\rm end})$.

By Theorem \ref{rce}, $U(\al)$ is relatively compact in $X$ if and only if $U_{\al}$ is relatively compact in $K(X)$.
Thus, we have the following affirmation:
\begin{itemize}
  \item Given $\al \in (0,1]$.
For each sequence $\{[u_n]_\al, n=1,2,\ldots\}$ in $U_{\al}$,
it has a subsequence
 $\{[u_{n_k}]_\al, k=1,2,\ldots\}$ which converges to $u_\al  \in K(X) $ with respect to the Hausdorff metric $H$.
 \end{itemize}
The remaining proof is similar to the corresponding part   of the ``Sufficiency part'' of Theorem \ref{tbfegn}.

We can also prove that $U$ is relatively compact
 in $(F^1_{USCG} (X), H_{\rm end})$ as follows.
From
the
``Sufficiency part'' of Theorem \ref{tbfegn}, we know
that
for each sequence $\{u_n\}$ in $U$,
there exists a subsequence $\{v_n\}$ of $\{u_n\}$ which converges to $v\in F^1_{USCG} (\widetilde{X})$.
From Theorem \ref{aec}
and
 the above statement after the ``$\bullet$'',
 we thus know
 that
 $v\in F^1_{USCG} (X)$.

\end{proof}

\begin{tm}\label{cfeg}
  Let $U$ be a subset of $F^1_{USCG} (X)$. Then the following are equivalent:
\begin{enumerate}
\renewcommand{\labelenumi}{(\roman{enumi})}

\item
 $U$ is compact in $(F^1_{USCG} (X), H_{\rm end})$;

\item  $U(\al)$
is relatively compact in $(X, d)$ for each $\al \in (0,1]$ and $U$ is closed in $(F^1_{USCG} (X), H_{\rm end})$;

\item  $U(\al)$
is compact in $(X, d)$ for each $\al \in (0,1]$ and $U$ is closed in $(F^1_{USCG} (X), H_{\rm end})$.
\end{enumerate}

\end{tm}

\begin{proof}
By Theorem \ref{rcfegn},
  (\romannumeral1) $\Leftrightarrow$ (\romannumeral2).
  Obviously (\romannumeral3) $\Rightarrow$ (\romannumeral2).
We
shall complete the proof by showing that
 (\romannumeral1) $\Rightarrow$
 (\romannumeral3). To do this,
suppose that $U$ is compact.
To verify (\romannumeral3), from the equivalence of (\romannumeral1) and (\romannumeral2),
we only need to
show that $U(\al)$ is closed in $(X,d)$ for each $\al\in (0,1]$.

Let $\al\in (0,1]$
and
let $\{x_n\}$ be a sequence in $U(\al)$ with $x_n \to x$.
Suppose that
$x_n \in [u_n]_\al$ and $u_n \in  U$ for $n=1,2,\ldots$.
Then
there exist subsequence $\{u_{n_k}\}$ of $\{u_n\}$ and $u\in U$
such that $H_{\rm end}(u_{n_k}, u) \to 0$.
So
by Remark \ref{hkr} $\lim_{n\to \infty}^{(\Gamma)}  u_{n_k} = u $ and therefore by Theorem \ref{Gclnre},
$\limsup_{n\to \infty}[u_{n_k}]_\alpha
\subseteq
 [u]_\alpha$.
Hence
$x\in [u]_\al$, and thus $x\in U(\al)$.

We can also show $x\in  [u]_\al \subseteq U(\al)$ in the following way.
From Theorem \ref{uscgre}, $H([u_{n_k}]_\al, [u]_\al) \rightarrow 0$ holds for $\al\in (0,1)\setminus P_0(u)$.
If $\al \in (0,1) \setminus P_0(u)$, then
$x\in [u]_\al$.
If $\al \in \{1\} \cup P_0(u)$, then for all $\beta\in  (0,\al) \setminus P_0(u)$, $x\in [u]_\beta$.
Thus
$x\in [u]_\al$.

\end{proof}

\subsection{Characterizations of compactness in $(P^1_{USCB} (X), H_{\rm send})$ and $(F^1_{USCB} (X), H_{\rm send})$}\label{pfc}

In this
subsection,
we give
 the characterizations of totally bounded sets, relatively compact sets and compact sets
in
$(P^1_{USCB} (X), H_{\rm send})$.
Then,
by treating
$(F^1_{USCB} (X), H_{\rm send})$ as a metric subspace of $(P^1_{USCB} (X), H_{\rm send})$,
 we give the characterizations of totally bounded sets and compact sets
in $(F^1_{USCB} (X), H_{\rm send})$.
The characterization of
relatively compact sets
in
$(F^1_{USCB} (X), H_{\rm send})$ has already been given in \cite{greco}.

Suppose that
$U$ is a subset of $P^1_{USC} (X)$ and $\al\in [0,1]$.
For
writing convenience,
we denote
\begin{itemize}
  \item  $U(\al):= \bigcup_{u\in U} \langle u \rangle_\al$, and

\item  $U_\al : =  \{\langle u \rangle_\al: u \in U\}$.
\end{itemize}

\begin{tm} \label{tbpu}
   Suppose that $U$ is a subset of $P^1_{USCB} (X)$. Then $U$ is totally bounded in $(P^1_{USCB} (X), H_{\rm send})$
if and only if
$U(0)$ is totally bounded in $(X,d)$.
\end{tm}

\begin{proof}
\textbf{\emph{ Necessity}}. \
 Suppose that $U$ is totally bounded. By clause (\romannumeral2) of Theorem \ref{pseu},
$U_0 $ is  totally bounded in $(K(X), H)$.
 From Theorem \ref{tbe}, this is equivalent to
$U(0)$ is totally bounded in $(X,d)$.

\textbf{\emph{Sufficiency}}. \
Suppose that $U(0)$ is totally bounded.
To show that $U$ is totally bounded in $(P^1_{USCB} (X), H_{\rm send})$, we
only need to prove
that each sequence in $U$ has a Cauchy subsequence with respect to $H_{\rm send}$.

Let
$\{u_n\}$ be a sequence in $U$. Note that $U(\al)$ is totally bounded for each $\al \in [0,1]$.
Then by Theorem \ref{tbfegn},
 $\{\overleftarrow{u_n}\}$ has a Cauchy subsequence $\{v_n\}$ in $(F^1_{USCB} (X), H_{\rm end})$.
From Theorem \ref{tbe}, $\{v_n\}$ has a subsequence $\{w_n\}$
such that
$\{ [w_n]_0 \}$
is a Cauchy sequence
in $(K(X), H)$.
 Thus by clauses (\romannumeral3) of Theorem \ref{pseu},
 $\{w_n\}$ is a Cauchy sequence in $(P^1_{USCB} (X), H_{\rm send})$.

\end{proof}

\begin{tm} \label{tbfe}
   Suppose that $U$ is a subset of $F^1_{USCB} (X)$. Then $U$ is totally bounded in $(F^1_{USCB} (X), H_{\rm send})$
if and only if
$U(0)$ is totally bounded in $(X,d)$.
\end{tm}

\begin{proof}
Note that $U$ is totally bounded in $(F^1_{USCB} (X), H_{\rm send})$
 if and only if
$\overrightarrow{U}$ is totally bounded in $(P^1_{USCB} (X), H_{\rm send})$,
and that
$U(0) = \overrightarrow{U}(0)$.
So the desired result follows from
Theorem \ref{tbpu}.

\end{proof}

\begin{tm}  \label{rcgu}
   Suppose that $U$ is a subset of $P^1_{USCB} (X)$. Then $U$ is relatively compact in $(P^1_{USCB} (X), H_{\rm send})$
if and only if
$U(0)$ is relatively compact in $X$.
\end{tm}

\begin{proof}
 \emph{\textbf{ Necessity}}. \  Suppose that $U$ is relatively compact. Then by clauses (\romannumeral2) of Theorem \ref{pseu},
$U_0 $ is  relatively compact in $(K(X), H)$.
 By Theorem \ref{rce},
$U(0)$ is  relatively compact in $X$.

\textbf{\emph{Sufficiency}}. \ To prove that $U$ is relatively compact, it suffices
to
show that each sequence in $U$
has a convergent subsequence
in $(P^1_{USCB} (X), H_{\rm send})$.

Let $\{u_n\}$ be a sequence in $U$.
Since $U(0)$ is relatively compact in $X$, then $U(\al)$ is relatively compact in $X$ for each $\al \in [0,1]$.
By Theorems \ref{rcfegn} and \ref{rce},
there is a subsequence $\{u_{n_k} \}$ of $\{u_n\}$,
 an $u\in F^1_{USCG} (X)$ and a $u_0 \in K(X)$
such that
$H_{\rm end} (\overleftarrow{u_{n_k}}, u) \to 0$ and $H(\langle u_{n_k}\rangle_0, u_0) \to 0$.

Set $w \in P^1_{USCB}(X)$
 given by
\[
\langle w \rangle_\al=\left\{
           \begin{array}{ll}
           [u]_\al, &  \al\in(0,1],
\\
             u_0,  & \al=0.
           \end{array}
         \right.
\]
Then $u=\overleftarrow{w}$, $H_{\rm end}(u_{n_k}, w) = H_{\rm end}(\overleftarrow{u_{n_k}}, u) \to 0$
and
$H(\langle u_{n_k}\rangle_0, \langle w\rangle_0) = H(\langle u_{n_k}\rangle_0, u_0) \to 0$ for $n=1,2,\ldots$
 Thus from (\romannumeral3) or (\romannumeral4) of Theorem \ref{pseu},
 $\{u_{n_k}\}$ converges to
 $w$ in $(P^1_{USCB} (X), H_{\rm send})$.

\end{proof}

\begin{itemize}
\item
$u \in F^1_{USC} (X)$ is said to be right-continuous at $0$ if $\lim_{\delta\to 0+} H([u]_\delta, [u]_0) =0$.

  \item
$U \subset F^1_{USC} (X)$ is said to be equi-right-continuous at $0$ if for each $\varepsilon>0$,
there is a $\delta>0$
such that
$H([u]_\delta, [u]_0) < \varepsilon$ for all $u \in U$.

\end{itemize}
By Lemma \ref{c},
for each
$u \in F^1_{USCB} (X)$, $u$ is right-continuous at $0$.

Theorem \ref{rcfe} below is presented in \cite{greco}.

\begin{tm} \cite{greco}   \label{rcfe}
   Suppose that $U$ is a subset of $F^1_{USCB} (X)$. Then $U$ is relatively compact in $(F^1_{USCB} (X), H_{\rm send})$
if and only if
$U(0)$ is relatively compact in $X$ and $U$ is equi-right-continuous at $0$.
\end{tm}

$\overrightarrow{F^1_{USCB}(X)}$ need not be a closed set of $P^1_{USCB}(X)$.
For instance, $F^1_{USCB}(D)$ given in Example \ref{nce} is not a closed set of $P^1_{USCB}(D)$.
We can see that
$\overrightarrow{F^1_{USCB}(X)}$ is a closed set of $P^1_{USCB}(X)$
if and only if
$X$ has only one element.

For a set $U$ in $F^1_{USCB}(X)$, suppose that
\\
(a) \ $U$
is relatively compact in $(F^1_{USCB}(X), H_{\rm send})$;
\\
(b) \
 $\overrightarrow{U}$
is relatively compact in $(P^1_{USCB}(X), H_{\rm send})$;
\\
(c) \ The topological closure of $\overrightarrow{U}$ in $(P^1_{USCB}(X), H_{\rm send})$ is a subset of $\overrightarrow{F^1_{USCB}(X)}$.

Then
(a) holds if and only if (b) and (c) hold.

$\overrightarrow{F^1_{USCB}(X)}$ is closed in $P^1_{USCB}(X)$ if and only if for each set $U$ in $F^1_{USCB}(X)$,
(c) holds.

The following
Proposition \ref{psf} illustrates the role of the condition ``$U$ is equi-right-continuous at $0$''
in
the characterization of the relative compactness for a set $U$ in $(F^1_{USCB}(X), H_{\rm send})$ given in Theorem \ref{rcfe}.

For $w \in P^1_{USCB} (X)$, the following are equivalent:
(\romannumeral1)
$w\in \overrightarrow{F^1_{USCB}(X)}$;
(\romannumeral2)
$w\in \overrightarrow{F^1_{USC}(X)}$;
(\romannumeral3)
$\lim_{\delta\to 0+} H(\langle w \rangle_\delta, \langle w \rangle_0) $=0.

\begin{pp} \label{psf}
  Let $U$ be a subset of $ F^1_{USCB}(X)$.
Suppose the following conditions (\romannumeral1), (\romannumeral2) and (\romannumeral3):
\\
(\romannumeral1) \
$U$ is relatively compact in $(F^1_{USCB}(X), H_{\rm send})$;
\\
(\romannumeral2) \
$U$ is equi-right-continuous at $0$;
\\
(\romannumeral3) \
The topological closure of $\overrightarrow{U}$ in $(P^1_{USCB}(X), H_{\rm send})$ is a subset of $\overrightarrow{F^1_{USCB}(X)}$.

Then the condition (\romannumeral1) implies the condition (\romannumeral2),
and the condition (\romannumeral2) implies the condition (\romannumeral3).
 If $\overrightarrow{U}$ is relatively compact in $(P^1_{USCB}(X), H_{\rm send})$,
then the conditions (\romannumeral1), (\romannumeral2) and (\romannumeral3) are equivalent to each other.

\end{pp}

\begin{proof}
By Theorem \ref{rcfe}, we know that
 (\romannumeral1)$\Rightarrow$(\romannumeral2).

To show
(\romannumeral2)$\Rightarrow$(\romannumeral3).
Suppose that $\{\overrightarrow{u_n}\}$ converges to $u$ in $(P^1_{USCB}(X), H_{\rm send})$.
Then
$\lim_{n\to \infty} H(\langle \overrightarrow{u_n} \rangle_0, \langle u \rangle_0 )=0$;
that is,
$\lim_{n\to \infty} H(   [u_n]_0, \langle u \rangle_0 )=0$.

Note that $\{u_n\}$ converges to $\overleftarrow{u}$ in $(F^1_{USCB}(X), H_{\rm end})$.
By
Theorem \ref{aec}, $H([u_n]_\al, [\overleftarrow{u}]_\al) \rightarrow 0$ holds a.e. on $\al\in (0,1)$.
From
 $U$ is equi-right-continuous at $0$,
 we have
$\lim_{n\to \infty} H([u_n]_0, [\overleftarrow{u}]_0 )=0 $.

Thus $\langle u \rangle_0 = [\overleftarrow{u}]_0 $,
and hence
$u\in \overrightarrow{F^1_{USCB} (X)}$.
So (\romannumeral2)$\Rightarrow$(\romannumeral3).

If $\overrightarrow{U}$ is relatively compact in $(P^1_{USCB}(X), H_{\rm send})$,
then clearly
 (\romannumeral3)$\Rightarrow$(\romannumeral1),
and
thus
the conditions (\romannumeral1), (\romannumeral2) and (\romannumeral3)
are equivalent to each other.

\end{proof}

\begin{re}{\rm

For conditions (\romannumeral1), (\romannumeral2) and (\romannumeral3) in Proposition \ref{psf},
 (\romannumeral2)  does not imply (\romannumeral1);
 (\romannumeral3) does not imply (\romannumeral2).

Let $\{u_n\}$ be a sequence of fuzzy sets in $F^1_{USCB}(\mathbb{R})$ defined by
$u_n(x)=[0,n]_{\mathbb{R}}$,
$n=1,2,\ldots$.
Then
 $\{u_n\}$ is equi-right-continuous at $0$ but $\{u_n\}$ is not
relatively compact
in $(F^1_{USCB}(X), H_{\rm send})$. So (\romannumeral2)  does not imply (\romannumeral1).

Let $\{v_n\}$ be a sequence of fuzzy sets in $F^1_{USCB}(\mathbb{R})$ defined by
\[v_n(x)=
\left\{
  \begin{array}{ll}
1, & x\in [0,n], \\
1/n, & x\in [-n,0],\\
   0, & x\in \mathbb{R}\setminus  [-n, n],
  \end{array}
\right.
n=1,2,\ldots.
\]
Then $\{\overrightarrow{v_n}\}$ has no limit in $(P^1_{USCB}(X), H_{\rm send})$ and hence is closed
in $(P^1_{USCB}(X), H_{\rm send})$. However $\{v_n\}$ is not equi-right-continuous at $0$.
 So (\romannumeral3) does not imply (\romannumeral2).
}
\end{re}

\begin{re}
  {\rm
Theorem \ref{rcgu} and Proposition \ref{psf} imply Theorem \ref{rcfe}.
This is because
by Proposition \ref{psf} we can obtain that for a subset $U$ of $F^1_{USCB} (X)$,
$U$ is relatively compact in $(F^1_{USCB}(X), H_{\rm send})$
if and only if
$\overrightarrow{U}$ is relatively compact in $(P^1_{USCB}(X), H_{\rm send})$
and
$U$ is equi-right-continuous at $0$.

}
\end{re}

\begin{tm} \label{cgu}
   Suppose that $U$ is a subset of $P^1_{USCB} (X)$. Then the following
are
equivalent:
\\
(\romannumeral1)  $U$ is compact in $(P^1_{USCB} (X), H_{\rm send})$;
\\
(\romannumeral2)    $U$ is closed in $(P^1_{USCB} (X), H_{\rm send})$, and $U(0)$ is relatively compact in $(X,d)$;
\\
(\romannumeral3)
  $U$ is closed in $(P^1_{USCB} (X), H_{\rm send})$, and $U(0)$ is compact in $(X,d)$.
\end{tm}

\begin{proof}
By Theorem \ref{rcgu}, we obtain that (\romannumeral1)$\Leftrightarrow$(\romannumeral2).
Clearly (\romannumeral3)$\Rightarrow$(\romannumeral2).
We shall complete the proof by showing that
(\romannumeral1)$\Rightarrow$(\romannumeral3).
 To do this, suppose
 that $U$ is compact in $(P^1_{USCB} (X), H_{\rm send})$. Then $U$ is closed in $(P^1_{USCB} (X), H_{\rm send})$.
To verify (\romannumeral3), we only need to show that
$U(0)$ is compact in $(X, d)$.

By clause (\romannumeral2) of Theorem \ref{pseu},
$U_0$ is compact in $(K(X), H)$.
Thus by Theorem \ref{come},
$U(0)$ is compact in $(X,d)$.

\end{proof}

\begin{tm} \label{cfe}
   Suppose that $U$ is a subset of $F^1_{USCB} (X)$. Then the following
are
equivalent:
\\
(\romannumeral1) $U$ is compact in $(F^1_{USCB} (X), H_{\rm send})$;
\\
(\romannumeral2)
  $U$ is closed in $(F^1_{USCB} (X), H_{\rm send})$, $U(0)$ is relatively compact in $X$
and
$U$ is equi-right-continuous at $0$;
\\
(\romannumeral3)
  $U$ is closed in $(F^1_{USCB} (X), H_{\rm send})$, $U(0)$ is compact in $X$
and
$U$ is equi-right-continuous at $0$.

\end{tm}

\begin{proof}
By Theorem \ref{rcfe}, we obtain that (\romannumeral1)$\Leftrightarrow$(\romannumeral2).
Clearly (\romannumeral3)$\Rightarrow$(\romannumeral2).
We shall complete the proof by showing that
(\romannumeral1)$\Rightarrow$(\romannumeral3). To do this, suppose
that $U$ is compact in $(F^1_{USCB} (X), H_{\rm send})$.
Since (\romannumeral1) implies (\romannumeral2),
to verify (\romannumeral3) we only need to show
that $U(0)$ is compact in $X$.

Note that
$U$ is compact in $(F^1_{USCB} (X), H_{\rm send})$ if and only if $\overrightarrow{U}$ is compact in $(P^1_{USCB} (X), H_{\rm send})$. Thus
by
Theorems \ref{cgu}, $U(0)=\overrightarrow{U}(0)$ is compact in $X$.

\end{proof}

\begin{re}
{\rm
For a subset $U$ of $F^1_{USCB} (X)$,
\\
(a) $U$ satisfies (\romannumeral1) of Theorem \ref{cfe}
if and only if
$\overrightarrow{U}$ satisfies (\romannumeral1) of Theorem \ref{cgu};
\\
(b)
$U$ satisfies (\romannumeral2) of Theorem \ref{cfe}
if and only if
$\overrightarrow{U}$ satisfies (\romannumeral2) of Theorem \ref{cgu};
\\
(c) $U$ satisfies (\romannumeral3) of Theorem \ref{cfe}
if and only if
$\overrightarrow{U}$ satisfies (\romannumeral3) of Theorem \ref{cgu}.

Clauses (b) and (c) can be obtained
by using
 Proposition \ref{psf} and Theorem \ref{rcgu}.
Theorem \ref{cgu} and clauses (a), (b) and (c) imply Theorem \ref{cfe}.

Since $U$ is compact in $(F^1_{USCB} (X), H_{\rm send})$ if and only if $\overrightarrow{U}$ is compact in $(P^1_{USCB} (X), H_{\rm send})$, we can use
Theorem \ref{cgu} to judge the compactness of a set $U$ in $(F^1_{USCB} (X), H_{\rm send})$.

}
\end{re}

\section{Completions
of $(F^1_{USCB} (X), H_{\rm send})$
and
$(F^1_{USCG} (X), H_{\rm end})$
}

In this section,
we show that $(P^1_{USCB} (\widetilde{X}), H_{\rm send})$ is a completion of
$(F^1_{USCB} (X), H_{\rm send})$.
We also show that
$(F^1_{USCG} (\widetilde{X}), H_{\rm end})$ is a completion of $(F^1_{USCB} (X), H_{\rm end})$,
and
thus a completion of $(F^1_{USCG} (X), H_{\rm end})$.

\begin{tm} \label{comfegn} Let $(X,d)$ be a metric space. Then the following
are
equivalent:
\begin{enumerate}
\renewcommand{\labelenumi}{(\roman{enumi})}

\item   $(X,d)$ is complete;

\item
  $(F^1_{USCG} (X), H_{\rm end})$ is complete.
\end{enumerate}
\end{tm}

\begin{proof}

(\romannumeral1) $\Rightarrow$ (\romannumeral2). \
Let $\{u_n\}$
 be a Cauchy sequence of $(F^1_{USCG} (X), H_{\rm end})$.
Then $U= \{u_n, n=1,2,\ldots\}$ is total bounded in $(F^1_{USCG} (X), H_{\rm end})$.
So
from the proof the sufficiency part of
Theorem \ref{tbfegn},
we know that $\{u_n\}$ has a convergent subsequence in $(F^1_{USCG} (X), H_{\rm end})$,
and thus
 $\{u_n\}$ is convergent in
$(F^1_{USCG} (X), H_{\rm end})$.

(\romannumeral2) $\Rightarrow$ (\romannumeral1). \
Let
 $\{x_n\}$ be a Cauchy sequence in $X$.
Note that $H_{\rm end} (\widehat{x},\,  \widehat{y}) = \min\{d(x,y),1\}$ for $x,y \in X$.
Then
$\{\widehat{x_n}\}$ is a Cauchy sequence in $(F^1_{USCG} (X), H_{\rm end})$,
and
therefore $\{\widehat{x_n}\}$ converges to $u\in F^1_{USCG} (X)$.
Thus there exists an $x\in X$
such
that $[u]_\al = \{x\}$ for all $\al\in [0,1]$ (i.e. $u = \widehat{x}$) and $d(x_n,x)\to 0$.

\end{proof}

\begin{re}
{\rm
 (\romannumeral2) $\Rightarrow$ (\romannumeral1) in Theorem \ref{comfegn} can also be shown as follows

$(X, d)$ is complete if and only if $(X, d^*)$ is complete, where $d^*(x,y) = \min\{d(x,y),1\}$ for $x,y \in X$.
Note that $H_{\rm end} (\widehat{x},\,  \widehat{y}) = d^*(x,y)$.
So the desired result follows
from the fact that
$(X, d^*)$ is isometric to the closed subspace
$(\widehat{X}, H_{\rm end})$
of $(F^1_{USCG} (X), H_{\rm end})$, where $\widehat{X} := \{\widehat{x}: x\in X\}$.
}
\end{re}

\begin{pp} \label{fcm}
Let $(X,d)$ be a metric space. Then
the following
are
equivalent:
\\
(\romannumeral1) \ $(X,d)$ is complete;
\\
(\romannumeral2) \
  $(F^1_{USCB} (X), d_\infty)$ is complete.

\end{pp}

\begin{proof}

 $(X,d)$ is isometric to $(\widehat{X}, d_\infty)$, which is a closed subspace
of   $(F^1_{USCB} (X), d_\infty)$. So (\romannumeral2)$\Rightarrow$(\romannumeral1) is proved.

To show (\romannumeral1)$\Rightarrow$(\romannumeral2),
suppose that $(X,d)$ is complete.
Let $\{u_n\}$ be a Cauchy sequence in  $(F^1_{USCB} (X), d_\infty)$.
Then for each $\al\in [0,1]$,
$\{ [u_n]_\al\}$ is a Cauchy sequence in $(K(X), H)$, and hence there is an $u(\al)\in K(X)$ such that
$H([u_n]_\al,  u(\al)) \to 0$.

As $\{u_n\}$ is a Cauchy sequence in  $(F^1_{USCB} (X), d_\infty)$,
 $H([u_n]_\al, u(\al)) $ converges uniformly to $0$ on $\al\in [0,1]$, denoted by
\begin{equation}\label{huc}
  H([u_n]_\al, u(\al)) \rightrightarrows 0 \,([0,1]).
\end{equation}
Clearly $u(\al) \subseteq u(\beta)$ for $0\leq \beta \leq \alpha \leq 1$.

If there is a $u\in F^1_{USCB} (X)$ such that $[u]_\al=u(\al)$, then by \eqref{huc},
$d_\infty(u_n, u) \to 0$, and so the proof is complete.
To prove the existence of a such $u\in F^1_{USCB} (X)$,
we only need to show that $\{u(\al), \al\in [0,1]\}$ has the following properties:
\\
(\romannumeral1) for each $\al\in (0,1]$, $u(\al) = \bigcap_{\beta < \al} u(\beta)$, and
\\
(\romannumeral2) $u(0) = \overline{\bigcup_{\gamma>0} u(\gamma)}$.

Since for $n=1,2,\ldots$ and $\al\in (0,1]$,
$\lim_{\beta\to \al-} H([u_n]_\al, [u_n]_\beta) = 0$ and
$\lim_{\gamma\to 0+} H([u_n]_\gamma, [u_n]_0) = 0$,
then by \eqref{huc},
$\lim_{\beta\to \al-} H(u(\al), u(\beta)) = 0$
and
$\lim_{\gamma\to 0+} H(u(\gamma), u(0)) = 0$.
Thus by Lemma \ref{c},
for each $\al\in (0,1]$, $u(\al) = \bigcap_{\beta < \al} u(\beta)$, and
$u(0) = \overline{\bigcup_{\gamma>0} u(\gamma)}$.
So
 (\romannumeral1) and (\romannumeral2) are true.

\end{proof}

Even if $(X,d)$ is complete,
$(F^1_{USCB} (X), H_{\rm send})$ need not be complete.
See Example \ref{nce} below.

\begin{eap} \label{nce}
  {\rm

let $D=\{0,1\}$
be a metric subspace of $\mathbb{R}$. Then $D$ is complete.
Let $u_n\in F^1_{USCB} (D)$, $n=1,2,\ldots$, be defined as
\[
u_n(x)=\left\{
         \begin{array}{ll}
           1, & x=0, \\
           1/n, & x=1.
         \end{array}
       \right.
\]
Then $\{\overrightarrow{u_n}\}$ converges
to $u\in P^1_{USCB} (D)$ defined by
\[
\langle u \rangle_\al =\left\{
         \begin{array}{ll}
           \{0\}, & \al\in (0,1], \\
           \{0,1\}, & \al=0.
         \end{array}
       \right.
\]
Thus
$\{u_n\}$ is a Cauchy sequence in $(F^1_{USCB} (D), H_{\rm send})$ and
 has no limit in $(F^1_{USCB} (D), H_{\rm send})$.
So
$(F^1_{USCB} (D), H_{\rm send})$ is not complete.

}
\end{eap}

We can see that $(F^1_{USCB} (X), H_{\rm send})$ is complete if and only if $X$ has only one element.

Theorem \ref{sce} below discusses the completeness of $(P^1_{USCB} (X), H_{\rm send})$
and then
Theorem \ref{scom} below gives
the completion of $(F^1_{USCB} (X), H_{\rm send})$.

\begin{tm} \label{sce} \ Let $(X,d)$ be a metric space. Then the following
are
equivalent:
\begin{enumerate}
\renewcommand{\labelenumi}{(\roman{enumi})}

\item   $X$ is complete;

\item
  $(P^1_{USCB} (X), H_{\rm send})$ is complete.
\end{enumerate}
\end{tm}

\begin{proof}

(\romannumeral1) $\Rightarrow$ (\romannumeral2). \
Let $\{u_n\}$
 be a Cauchy sequence in $(P^1_{USCB} (X), H_{\rm send})$.
Then $\{\overleftarrow{u_n}\} \subseteq F^1_{USCB} (X)$,
 and by (\romannumeral1) of Theorem \ref{pseu},
$H_{\rm end}(\overleftarrow{u_n}, \overleftarrow{u_m}) = H_{\rm end}(u_n, u_m) \leq  H_{\rm send}(u_n, u_m)$
for $n,m=1,2,\ldots$.
Hence
 $\{\overleftarrow{u_n}\}$
  is a Cauchy sequence in $(F^1_{USCG} (X), H_{\rm end})$.
From Theorem \ref{comfegn},
there is an $u \in F^1_{USCG} (X)$ such that
  $\{\overleftarrow{u_n}\}$ converges to $u$ in $(F^1_{USCG} (X), H_{\rm end}$).

 By (\romannumeral2) of Theorem \ref{pseu}, $ H(\langle u_n \rangle_0, \langle u_m \rangle_0) \leq H_{\rm send}(u_n, u_m)$ for $n,m=1,2,\ldots$.
So
 $\{ \langle u_n \rangle_0 \}$
is a Cauchy sequence in
 $(K(X), H)$.
From Theorem \ref{bfc},
there is an $u_0 \in K(X)$
such that $\{ \langle u_n \rangle_0 \}$ converges to $u_0$ in
 $(K(X), H)$.

Set $w \in P^1_{USCB}(X)$
 given by
\[
\langle w \rangle_\al=\left\{
           \begin{array}{ll}
           [u]_\al, &  \al>0,
\\
             u_0,  & \al=0.
           \end{array}
         \right.
\]
Then $u=\overleftarrow{w}$, $H_{\rm end}(u_n, w) = H_{\rm end}(\overleftarrow{u_n}, u)$
and
$H(\langle u_n\rangle_0, \langle w\rangle_0) = H(\langle u_n\rangle_0, u_0)$ for $n=1,2,\ldots$
 Thus from (\romannumeral3) or (\romannumeral4) of Theorem \ref{pseu},
 $\{u_n\}$ converges to
 $w$ in $(P^1_{USCB} (X), H_{\rm send})$.

(\romannumeral2) $\Rightarrow$ (\romannumeral1). \
Note that $d(x,y) = H_{\rm send} (\widehat{x},\,  \widehat{y})$.
So the desired result
follows from the fact
that $(X,d)$ is isometric to
a closed subspace of
 $(P^1_{USCB} (X), H_{\rm send})$.

\end{proof}

\begin{tm} \label{scom}
 $(P^1_{USCB} (\widetilde{X}), H_{\rm send})$ is a completion of $(F^1_{USCB} (X), H_{\rm send})$.
\end{tm}

\begin{proof}
From Theorem \ref{sce}, $(P^1_{USCB} (\widetilde{X}), H_{\rm send})$ is complete.
To show
that
$(P^1_{USCB} (\widetilde{X}), H_{\rm send})$ is a completion of $(F^1_{USCB} (X), H_{\rm send})$,
we only need to show
that
for each $u \in P^1_{USCB} (\widetilde{X})$ and each $\varepsilon>0$,
there is a
$w\in F^1_{USCB} (X)$
such that
$H_{\rm send}(u, \overrightarrow{w}) \leq \varepsilon$.
To show this is equivalent to
show the following affirmations (a) and (b):
\begin{description}
  \item[(a)] For each $u \in P^1_{USCB} (\widetilde{X})$ and each $\varepsilon>0$,
there exists a $v \in F^1_{USCB} (\widetilde{X})$
such that
$H_{\rm send} (u, \overrightarrow{v}) \leq \varepsilon$.

  \item[(b)] For each $v \in F^1_{USCB} (\widetilde{X})$ and each $\varepsilon>0$,
there exists a $w  \in F^1_{USCB} (X) $
such that $d_\infty (v, w) \leq \varepsilon$. Then by \eqref{smr} and \eqref{emsf},
$H_{\rm send} (v, w) \leq \varepsilon$ and $H_{\rm end} (v, w) \leq \varepsilon$.
\end{description}

Let $u \in P^1_{USCB} (\widetilde{X})$. Define $u_\varepsilon \in F^1_{USCB} (\widetilde{X})$, $\varepsilon>0$, given by
  \[
  [u_\varepsilon]_{\al} =\left\{
         \begin{array}{ll}
           \langle u \rangle_\al , & \al \in (\varepsilon, 1], \\
           \langle u \rangle_0, & \al \in  [0,\varepsilon ].
         \end{array}
       \right.
  \]
Then
$H_{\rm send} (u, \overrightarrow{u_\varepsilon}) \leq \varepsilon$.
So affirmation (a) is proved.

Let $v \in F^1_{USCB} (\widetilde{X})$.
We can choose
 a finite subset
$C_0$ of $X$ such that $H(C_0, [v]_0) < \varepsilon$.
Define
\begin{equation}\label{caf}
  C_\al: =   \{x\in C_0:  d (x, [v]_\al ) \leq \varepsilon    \}, \ \al\in (0,1].
\end{equation}
We affirm that $\{C_\al: \al \in [0,1]\}$ has the following properties
\\
(\romannumeral1) \ $C_\al \not= \emptyset$ for all $\al\in [0,1]$.
\\
(\romannumeral2) \ $H( C_\al, [v]_\al ) \leq \varepsilon$ for all $\al\in [0,1]$.
\\
(\romannumeral3) \  $C_\al = \bigcap_{\beta<\al} C_\beta$ for all $\al\in (0,1]$.
\\
(\romannumeral4) \ $C_0 = \bigcup_{\al>0} C_\al = \overline{\bigcup_{\al>0} C_\al }$.

For each $y\in [v]_\al $,
there
 exists $z_y\in C_0$ such that
$d(y,z_y) = d(y, C_0) < \varepsilon$. Hence $z_y\in C_\al$ and thus $C_\al \not= \emptyset$.
So
(\romannumeral1) is true.

To show (\romannumeral2), we only need
to show
that $ H ( C_\al, [v]_\al ) \leq \varepsilon  $
for $0<\al \leq 1$.
Let $\al\in (0,1]$.
From \eqref{caf},
 $ H^* ( C_\al, [v]_\al ) \leq \varepsilon  $.
In
the following, we show that
 $ H^* ( [v]_\al, C_\al) < \varepsilon  $.
In fact,
for each $y\in [v]_\al$,
$d(y, C_\al) \leq  d(y,z_y)  = d(y, C_0) $
(hence
$d(y, C_\al) = d(y, C_0)$).
Thus
 $ H^* ([v]_\al, C_\al) \leq  H([v]_0, C_0)  < \varepsilon  $.
So
(\romannumeral2) is proved.

Let $\al\in (0,1]$. Clearly $C_\al \subseteq \bigcap_{\beta<\al} C_\beta$.
By Lemma \ref{c},
 $\lim_{\beta \to \al-} H([v]_\al, [v]_\beta ) = 0$.
So for each $x\in X$,
$d(x, [v]_\alpha ) = \lim_{\beta \to \al-} d(x, [v]_\beta ) $,
and hence
$C_\al \supseteq \bigcap_{\beta<\al} C_\beta$.
Thus
$C_\al = \bigcap_{\beta<\al} C_\beta$.
So (\romannumeral3) is true.

Let $x \in C_0$. Then $d(x, [v]_0) < \varepsilon$.
Since
$[v]_0 =  \overline{\cup_{\al>0} [v]_\al}$,
there
exists $\al>0$ such that
$d(x, [v]_\al) < \varepsilon$ (in fact, for each $x\in X$,
$d(x, [v]_0 ) = \inf_{\al > 0} d(x, [v]_\al)$),
and thus
$x\in C_\al$.
So (\romannumeral4) is proved.

Set $w \in F(X)$ given by
$
  [w]_{\al} = C_\al
$
for all $\al \in [0,1]$.
Then
by (\romannumeral1), (\romannumeral3) and (\romannumeral4),
$w \in F^1_{USCB} (X)$.
From
(\romannumeral2), we have $d_\infty (v, w) \leq \varepsilon$, and then
$H_{\rm send} (v, w) \leq \varepsilon$.
So affirmation (b) is proved.

\end{proof}

\begin{pp}
  $(F^1_{USCB} (\widetilde{X}), d_\infty)$ is a completion of $(F^1_{USCB} (X),  d_\infty)$.
\end{pp}

\begin{proof}
By Proposition \ref{fcm},
$(F^1_{USCB} (\widetilde{X}), d_\infty)$ is complete.
 So from affirmation (b)
 in the proof of Theorem \ref{scom},
we have the desired result.

\end{proof}

\begin{tm} \label{pcom}
 $(P^1_{USCB} (\widetilde{X}), H_{\rm send})$ is a completion of $(P^1_{USCB} (X), H_{\rm send})$.
\end{tm}

\begin{proof}
  $(F^1_{USCB} (X), H_{\rm send})$ can be seen as a metric subspace of $(P^1_{USCB} (X), H_{\rm send})$.
$(P^1_{USCB} (X), H_{\rm send})$ is a metric subspace of $(P^1_{USCB} (\widetilde{X}), H_{\rm send})$.
So the desired result follows from Theorem \ref{scom}.

\end{proof}

\begin{tm} \label{ecom}
 $(F^1_{USCG} (\widetilde{X}), H_{\rm end})$ is a completion of $(F^1_{USCB} (X), H_{\rm end})$.

\end{tm}

\begin{proof}
  From Theorem \ref{comfegn}, $(F^1_{USCG} (\widetilde{X}), H_{\rm end})$ is complete.
To
 prove the desired result, it suffices to
show
that $F^1_{USCB} (X)$ is dense in $(F^1_{USCG} (\widetilde{X}), H_{\rm end})$.
By
affirmation (b) in the proof of Theorem \ref{scom},
to verify this it is enough to show
that for each $u \in F^1_{USCG} (\widetilde{X})$ and each $\varepsilon>0$,
there is a $v \in  F^1_{USCB} (\widetilde{X})$ such that $H_{\rm end} (u,v) \leq \varepsilon$.

Let $u \in F^1_{USCG} (\widetilde{X})$.
Define $u^\varepsilon \in  F^1_{USCB} (\widetilde{X})$, $\varepsilon>0$, given by
  \[
  [u^\varepsilon]_{\al} =\left\{
         \begin{array}{ll}
           [u]_\al , & \al \in (\varepsilon, 1], \\
          \mbox{} [u]_\varepsilon, & \al \in  [0,\varepsilon ].
         \end{array}
       \right.
  \]
Then
$H_{\rm end}  (u, u^\varepsilon) \leq \varepsilon$.

\end{proof}

\begin{tl} \label{fgecom}
 $(F^1_{USCG} (\widetilde{X}), H_{\rm end})$ is a completion of $(F^1_{USCG} (X), H_{\rm end})$.

\end{tl}

\begin{proof}
  Since $F^1_{USCB} (X) \subseteq F^1_{USCG} (X) \subseteq F^1_{USCG} (\widetilde{X})$,
the desired result
follows from
Theorem \ref{ecom}.

\end{proof}

\section{Conclusions}

In this paper, we discuss the properties and relations of
$H_{\rm end}$ metric and $H_{\rm send}$ metric on fuzzy sets in a metric space $X$.

To aid discussion, we introduce the sets $P^1_{USC}(X)$ and $P^1_{USCB}(X)$. $P^1_{USCB}(X)$ is a subset of $P^1_{USC}(X)$.
The
$F^1_{USC}(X)$ and $F^1_{USCB}(X)$ can be viewed as the subsets of $P^1_{USC}(X)$ and $P^1_{USCB}(X)$, respectively.
We
define the
$H_{\rm send}$ distance and the $H_{\rm end}$ distance on $P^1_{USC}(X)$,
and give the relations among
the $H_{\rm send}$ distance, the $H_{\rm end}$ distance and the Kuratowski convergence on $P^1_{USC}(X)$.
Then, as corollaries, we obtain the
 relations among
the $H_{\rm send}$ metric, the $H_{\rm end}$ metric and the $\Gamma$-convergence on $F^1_{USC}(X)$.

We
give the
level characterizations of $H_{\rm end}$ convergence and $\Gamma$-convergence on $F^1_{USC}(X)$.
By using the above results including the level characterizations of the $H_{\rm end}$ convergence,
we give
 the relations among the $H_{\rm end}$ metric, the $H_{\rm send}$ metric and
the $d_p^*$ metric.
We point out that the values of the metrics can be directly compared among certain ones in the supremum metric $d_\infty$, the $H_{\rm end}$ metric, the $H_{\rm send}$ metric and
the $d_p^*$ metric.

Based on above results,
we
give characterizations of compactness and completions of two kinds
of
fuzzy set spaces $(F^1_{USCG} (X), H_{\rm end})$
and
$(F^1_{USCB} (X), H_{\rm send})$, respectively.
We also investigate characterizations of compactness and completions
of
$(P^1_{USCB} (X), H_{\rm send})$.
$(F^1_{USCB} (X), H_{\rm send})$ can be treated as a metric subspace of $(P^1_{USCB} (X), H_{\rm send})$.

Note
 that $\mathbb{R}^m$ is complete and
 for a set $V$ in $\mathbb{R}^m$, the following are equivalent:
(\romannumeral1 ) $V$ is bounded;
(\romannumeral2) $V$ is totally bounded; (\romannumeral3) $V$
is relatively compact.
We can obtain the characterizations of compactness and completions
of
$(F^1_{USCB} (\mathbb{R}^m), H_{\rm send})$, $(F^1_{USCG} (\mathbb{R}^m), H_{\rm end})$ and $(P^1_{USCB} (\mathbb{R}^m), H_{\rm send})$
by using that of
$(F^1_{USCB} (X), H_{\rm send})$, $(F^1_{USCG} (X), H_{\rm end})$ and $(P^1_{USCB} (X), H_{\rm send})$ given in this paper.

The results in this paper
have potential applications in fuzzy set research involving
$H_{\rm end}$ metric and $H_{\rm send}$ metric.

\newpage

 \appendix

\section{Counterexamples}

In this section, we give an example
 to
illustrate the
conclusions in Sections \ref{lcg},
\ref{lce} and \ref{rem}.

Let $(X_j, d_j)$, $j \in J$, be metric spaces.
Define an extended metric $d$
on
$\prod_{j\in J} X_j$
as
\begin{equation}\label{upm}
d(x,y):= \sup\{d_j(x_j, y_j):  j \in J\}
\end{equation}
for $x=(x_j)_{j \in J}$ and $y=(y_j)_{j \in J}$.

We use the symbol $\prod_{j\in J} (X_j, d_j)$ to denote the extended metric space $(\prod_{j\in J} X_j, d)$.
If not mentioned specially, we suppose by default
that
$\prod_{j\in J} X_j$ is endowed with the extended metric $d$ given by \eqref{upm}.

Let $u_j  \in  F(X_j)$, $j \in J$.
Define
 $u \in F (\prod_{j \in J} X_j) $
as
\begin{equation}\label{pdf}
 [u]_\al = \prod_{j \in J} [u_j]_\al \   \mbox{for each} \    \al \in (0,1].
\end{equation}
We use
\bm{$ \prod_{j \in J} u_j$}
to denote the fuzzy set
 $u$ in $\prod_{j \in J} X_j$ given by \eqref{pdf}.

In \cite{huang17c}, we show
that
$u= \prod_{j \in J} u_j $ is well-defined and give the conclusion stated in Theorem \ref{pfn}.
For an extended metric space $(Y,\rho)$, we define
\begin{gather*}
  F_{USC} (Y)  =  \{ u\in F(Y): [u]_\al \ \mbox{is closed in} \ (Y,\rho) \ \mbox{for} \  \al\in (0,1]  \} \\
   F^1_{USC} (Y)  =  \{ u\in F(Y): [u]_\al \ \mbox{is nonempty and closed in} \ (Y,\rho) \ \mbox{for} \  \al\in (0,1]  \}
\end{gather*}

\begin{tm} \cite{huang17c} \label{pfne}
Let $J$ be a set, and for each $j \in J$, let $(X_j, d_j)$
be a
metric space.
If $u_j  \in  F_{USC}(X_j)$ for each $j \in J$,
then
 $u = \prod_{j \in J} u_j $
is a
fuzzy set
in $F_{USC} (\prod_{j \in J} X_j)$.

\end{tm}

\begin{tm} \label{pfn}
Let $J$ be a set, and for each $j \in J$, let $(X_j, d_j)$
be a
metric space.
If $u_j  \in  F^1_{USC}(X_j)$ for each $j \in J$,
then
 $u = \prod_{j \in J} u_j $
is a
fuzzy set
in $F^1_{USC} (\prod_{j \in J} X_j)$.

\end{tm}

\begin{proof}
  The desired result follows immediately from Theorem \ref{pfne} and the definition of $\prod_{j \in J} u_j $.

\end{proof}

\begin{tm} \label{fvm}
Let $J$ be a set, and for each $j \in J$, let $(X_j, d_j)$
be a
metric space,
$u_j  \in  F(X_j)$
and $u = \prod_{j \in J} u_j $.
Then for each $\xi = (\xi_j)_{j\in J} \in \prod_{j \in J} X_j$,
 $u(\xi) =  \inf_{j\in J} u_j (\xi_j) $.

\end{tm}

\begin{proof}

For each $\al\in (0,1]$, by \eqref{pdf}
\begin{align*}
 & u(\xi) \geq \al \\
& \Leftrightarrow  \xi\in [u]_\al = \prod_{j \in J} [u_j]_\al \\
& \Leftrightarrow \mbox{for each } j\in J, \ \xi_j \in [u_j]_\al \\
& \Leftrightarrow   \mbox{for each } j\in J, \ u_j(\xi_j) \geq \al \\
& \Leftrightarrow  \inf_{j\in J} u_j(\xi_j) \geq \al.
\end{align*}
Thus we obtain the desired result.

\end{proof}

 $[0,3]$ can be seen as a metric subspace of $\mathbb{R}$.
So
$\prod_{x\in (0,1]} [0,3]$ is a metric subspace with the metric $d$ given by \eqref{upm}.

The following Example \ref{rcp} shows
that
there exist $u$ and $u_n$, $n=1,2,\ldots$ in
$ F^1_{USC} (\prod_{x\in (0,1]} [0,3])$
such that
 $  H_{\rm send} (u_n, u) \to 0$,
$\{[u_n]_\al \}$ does not Kuratowski converge
to
$[u]_\al$ when $\al\in (0,1]$,
and
$H([u_n]_\al, [u]_\al) = 1$
 for all $\al\in (0,1]$ and $n=1,2,\ldots$.

\begin{eap}
 \label{rcp}
{\rm
 For $x\in (0,1]$, define
  $u_x\in F^1_{USC} ([0,3])$ as follows
  \[
u_x(t)=\begin{cases}
1, & \ t=0,
\\
x, & \ t\in (0,1],
\\
0, & \ t\notin [0,1].
\end{cases}
\]
Put
  $$u= \prod_{x\in (0,1]} u_x.$$
  Then, by Theorem \ref{pfn}, $u   \in     F^1_{USC} (\prod_{x\in (0,1]} [0,3])$
and
  \begin{equation}\label{uce}
    [u]_\al =\left\{
             \begin{array}{ll}
             \prod_{x\in (0,1)} \{0\}, & \al = 1,
\\
  \prod_{x\in (0,\alpha)} \{0\} \times  \prod_{x\in [\alpha, 1)} [0,1], & \al \in (0, 1).
             \end{array}
           \right.
  \end{equation}
 Thus
$P(u) = P_0(u) = (0,1)$.

For
$n=1,2,\ldots$, $x\in (0,1]$,
define $u_{n, x} \in F^1_{USC} ([0,3])$ as follows
 \[
u_{n, x} (t)=\begin{cases}
1, & \ t=0,
\\
\dfrac{-x}{n} t + x, & \ t\in (0,1],
\\
0, & \ t\notin [0,1].
\end{cases}
\]
We can see that for each $x\in (0,1]$, $H_{\rm send} (u_{n,x} , u_x) \to 0$ as $n\to \infty$.

Put
  $$u_n= \prod_{x\in (0,1]} u_{n,x}.$$
  We affirm that
\begin{description}
  \item[(a-\romannumeral1)] \  $[u]_0 = [u_n]_0 = \{ \xi = (\xi_x)_{x\in (0,1]} \in \prod_{x\in (0,1]} [0,1] : \xi_x \to 0 \ \mbox{as} \ x\to 0  \}$ for $n=1,2,\ldots$;

  \item[(a-\romannumeral2)] \ $  H_{\rm send} (u_n, u) \to 0$;

  \item[(a-\romannumeral3)]
For each $\al\in (0,1]$, there is a $\zeta \in [u]_\al$ such that
$d(\zeta, [u_n]_\al) = 1$ for all $n=1,2,\ldots$;

  \item[(a-\romannumeral4)]
$\{[u_n]_\al \}$ does not Kuratowski converge
to
$[u]_\al$
according to
$(\prod_{x\in (0,1]} [0,3], d)$ when $\al\in (0,1]$;

  \item[(a-\romannumeral5)] $H([u_n]_\al, [u]_\al) =1$ for all $\al\in (0,1]$ and $n=1,2,\ldots$.

\end{description}

Set $D := \{ \xi  = (\xi_x)_{x\in (0,1]}  \in \prod_{x\in (0,1]} [0,1] : \xi_x \to 0 \ \mbox{as} \ x\to 0  \}$.
Since $[u_n]_\al \subseteq [u_{n+1}]_\al \subseteq [u]_\al$
for $\al \in (0,1]$ and $n=1,2,\ldots$, then
we have that
 $[u_n]_0 \subseteq [u_{n+1}]_0 \subseteq [u]_0$ for $n=1,2,\ldots$.
So to show (a-\romannumeral1), we only need to show that $D \subseteq [u_1]_0$
and
$[u]_0 \subseteq D $.

To show $D \subseteq [u_1]_0$, it suffices to show that
for each $\xi\in D$ and $\varepsilon>0$,
there exists
an $\eta \in \cup_{\al>0} [u_1]_\al$ such that $d(\xi, \eta ) < \varepsilon$.

Let $\xi\in D$ and $\varepsilon>0$. Then there is a $\delta\in (0,1]$ such that
$\xi_x \leq \varepsilon$ when $x <\delta$.
Take a $\gamma \in (0, 1]$
such that
$[u_{1,\delta}]_\gamma \supseteq [0, 1-\varepsilon] $.
Then
$[u_{1, x}]_\gamma \supseteq [0, 1-\varepsilon] $ when $x\in [\delta,1]$.
So for each
$x\in [\delta,1]$, there is an $\eta(x) \in [u_{1, x}]_\gamma$
such that
$|\xi_x - \eta(x) | \leq \varepsilon$.

Define
$\eta=(\eta_x)_{x\in (0,1]} $ as
\[
\eta_x=\left\{
           \begin{array}{ll}
             \eta(x), & x\geq \delta, \\
         0, &   x< \delta.
           \end{array}
         \right.
\]
Then for each $x\in (0,1]$, $\eta_x \in [u_{1,x}]_\gamma$ and therefore $\eta\in [u_1]_\gamma \subset \cup_{\al>0} [u_1]_\al$.
We can see that
$d(\xi, \eta ) = \sup_{x\in (0,1]} d(\xi_x, \eta_x) \leq \varepsilon$.

To show $[u]_0 \subseteq D $, it suffices to show that if $\xi \notin D$,
then $\xi \notin [u]_0 $.
Since $[u]_0 \subseteq \prod_{x\in (0,1]} [0,1]$, to verify this,
it is enough to show that if $\xi\in \prod_{x\in (0,1]} [0,1] \setminus D$, then
$\xi \notin [u]_0 $.

Let $\xi\in \prod_{x\in (0,1]} [0,1] \setminus D$.
Then there is an $\varepsilon>0$ and a sequence $\{x_n\} $ in $[0,1] $
such that
$x_n \to 0$ and $\xi_{x_n} \geq \varepsilon$.
We claim that
for each $\al>0$
and $\zeta \in [u]_\al$, $d(\xi, \zeta) \geq \varepsilon$ because $\zeta_x = 0$ for $x<\al$.
Thus
$\xi \notin [u]_0 =  \overline{\cup_{\al>0} [u]_\al}$.

To show
(a-\romannumeral2), we only need to show that
\begin{gather}
H^* (  {\rm send}\,u_n,     {\rm send}\,u   ) \to 0, \label{hsrcnu}
\\
H^*({\rm send}\,u, {\rm send}\,u_n) \to 0. \label{hrscun}
\end{gather}

 For $n=1,2,\ldots$, since
$\mathrm{send} \, u_n \subset \mathrm{send} \, u$, then
\begin{equation*}
H^*({\rm send}\, u_n, {\rm send}\, u) = 0.
\end{equation*}
Thus \eqref{hsrcnu} is true.

To show \eqref{hrscun},
let
 $(\xi, \al) \in \mathrm{send} \, u$, where $\xi=(\xi_x)_{x\in (0,1]}$.
Put
$\alpha_n = \max \{ \al- 1/n, 0   \}$. We claim that $(\xi, \alpha_n) \in \mathrm{send} \, u_n$.

If
$\al_n =0$, then by affirmation (a-\romannumeral1),
$(\xi, \alpha_n) \in \mathrm{send} \, u_n$.

If
$\al_n >0$.
Note that for each $n=1,2,\ldots$, $x\in (0,1]$ and $t\in [0,1]$,
 $u_{n,x} (t) \geq u_x(t)-\frac{1}{n}$.
Then by Theorem \ref{fvm}
\begin{align*}
  u_n(\xi) = \inf_{x\in (0,1]}  u_{n,x}(\xi_x) \geq \inf_{x\in (0,1]}  u_{x}(\xi_x) - \frac{1}{n} =   u(\xi)- \frac{1}{n} \geq \al- \frac{1}{n}=\al_n.
\end{align*}
Thus
$(\xi, \alpha_n) \in \mathrm{send} \, u_n$.

Since
$
|\al-\alpha_n| \leq 1/n $,
then
$d((\xi, \al), \mathrm{send} \, u_n) \leq d((\xi, \al), (\xi, \alpha_n) )\leq 1/n$.
From
 the arbitrariness of $(\xi, \al) \in \mathrm{send} \, u$,
we
have
$H^*({\rm send}\,u, {\rm send}\,u_n) \leq 1/n$ (in fact, $H^*({\rm send}\,u, {\rm send}\,u_n) = 1/n$)
and
thus \eqref{hrscun} is true.

To show (a-\romannumeral3),
let $\al\in (0,1]$. Define $\zeta = (\zeta_x)_{x\in (0,1]}$ by
\[\zeta_x =
 \left\{
  \begin{array}{ll}
    1, & x=\al, \\
    0, & x=(0,1] \setminus \{ \al\}.
  \end{array}
\right.
\]
Then
by \eqref{uce},
$\zeta \in [u]_\al = \prod_{x\in (0,1]}[u_{x}]_\alpha$.

Note that for each $n=1,2,\ldots$ and $x\in (0,1]$,
it holds that $[u_{n,x}]_x= \{0\}$.
Then
for each $n=1,2,\ldots $ and $\xi \in  [u_n]_\al =  \prod_{x\in (0,1]}   [u_{n,x}]_\alpha$,
we have $\xi_\al = 0$.
and
thus for $n=1,2,\ldots$,
$$d(\zeta, [u_n]_\al) = 1.$$
Hence (a-\romannumeral3)
is true.

(a-\romannumeral4)
follows immediately from (a-\romannumeral3).

(a-\romannumeral5)
follows immediately from (a-\romannumeral3) and
the fact that
$[u]_\al \subset  \prod_{x\in (0,1]} [0,1]  $ and $[u_n]_\al  \subset \prod_{x\in (0,1]} [0,1] $, $n=1,2,\ldots$.

  }
\end{eap}

\section*{Acknowledgement}

The author would like to thank the three anonymous referees
for their invaluable comments and suggestions.

\end{document}